\documentclass[10pt,reqno]{amsart}
\usepackage{amsthm, amsmath,amsfonts,amssymb,euscript,hyperref,graphics,color,slashed}
\usepackage{graphicx}
\usepackage{mathrsfs}
\usepackage{comment}
\usepackage{import}
\usepackage{tikz}
\usepackage{latexsym}
\usepackage[makeroom]{cancel}
\def\ub{{\underline{u}}}

\def\inte#1{
\displaystyle\mathop{#1\kern0pt}^\circ }

\let\pa=\partial
\let\al=\alpha
\let\b=\beta

\let\lam=\lambda
\let\r=\rho

\let\f=\frac

\let\p=\psi
\let\om=\omega

\let\D=\Delta

\let\wt=\widetilde

\def\pa{\partial}

\def\virgp{\raise 2pt\hbox{,}}
\def\cdotpv{\raise 2pt\hbox{;}}

\def\eqdefa{\buildrel\hbox{\footnotesize def}\over =}

\def\Sp{\mathop{\rm Sp}\nolimits}
\def\C{\mathop{\mathbb C\kern 0pt}\nolimits}
\def\DD{\mathop{\mathbb D\kern 0pt}\nolimits}
\def\EE{\mathop{{\mathbb E \kern 0pt}}\nolimits}
\def\K{\mathop{\mathbb K\kern 0pt}\nolimits}
\def\N{\mathop{\mathbb N\kern 0pt}\nolimits}
\def\Q{\mathop{\mathbb Q\kern 0pt}\nolimits}
\def\R{\mathop{\mathbb R\kern 0pt}\nolimits}
\def\SS{\mathop{\mathbb S\kern 0pt}\nolimits}
\def\ZZ{\mathop{\mathbb Z\kern 0pt}\nolimits}
\def\TT{\mathop{\mathbb T\kern 0pt}\nolimits}
\def\P{\mathop{\mathbb P\kern 0pt}\nolimits}

\def\dv{\mbox{div}}

\def\curl{\mathop{\rm curl}\nolimits}

\def\na{\nabla}
\def\p{\partial}

\def\doubleint{{\int\!\!\!\!\int}}

\newcommand{\beq}{\begin{equation}}
\newcommand{\eeq}{\end{equation}}
\newcommand{\ben}{\begin{eqnarray}}
\newcommand{\een}{\end{eqnarray}}
\newcommand{\beno}{\begin{eqnarray*}}
\newcommand{\eeno}{\end{eqnarray*}}

\newtheorem{thm}{Theorem}[section]

\newtheorem{prop}{Proposition}[section]

\newcommand{\vv}[1]{\boldsymbol{#1}}

\def\v{v}
\def\b{b}
\def\div{\text{div}\,}
\def\curl{\text{curl}\,}
\def\Cp{{C}^+}
\def\Cm{{C}^-}
\def\Cpm{{C}^\pm}
\def\fp{{f}_+}
\def\fm{{f}_-}
\def\Sp{{S}^+}
\def\Sm{{S}^-}
\def\Lp{{L}_+}
\def\Lm{{L}_-}

\def\Zp{{Z}_+}
\def\Zm{{Z}_-}
\def\Zpm{{Z}_{\pm}}
\def\zp{{z}_+}
\def\zm{{z}_-}
\def\zpm{{z}_{\pm}}
\def\jp{j_+}
\def\jm{j_{-}}
\def\jpm{{j}_{\pm}}
\def\up{u_+}
\def\um{u_{-}}
\def\lambdap{{\lambda}_+}
\def\lambdam{{\lambda}_-}
\def\nup{\nu^+}
\def\num{\nu^{-}}
\def\rhop{{\rho}_+}
\def\rhom{{\rho}_-}

\def\wp{\langle w_+\rangle}
\def\wm{\langle w_{-} \rangle}
\def\wpm{\langle w_{\pm} \rangle}
\def\wmp{\langle w_{\mp} \rangle}
\def\Up{\langle u_+\rangle}
\def\Um{\langle u_{-} \rangle}
\def\Upm{\langle u_{\pm} \rangle}
\def\Ump{\langle u_{\mp} \rangle}
\def\zz{\mathfrak{Z}}
\DeclareMathOperator{\I}{I}

\def\doubleint{{\int\!\!\!\!\int}}

\newtheorem*{Main Theorem}{Main Theorem}
\newtheorem{theorem}{Theorem}[section]
\newtheorem{lemma}[theorem]{Lemma}
\newtheorem{proposition}[theorem]{Proposition}
\newtheorem{corollary}[theorem]{Corollary}
\newtheorem{definition}[theorem]{Definition}
\newtheorem{remark}[theorem]{Remark}

\numberwithin{equation}{section}

\begin{document}
\title[MHD and Alfv\'en waves]{On global dynamics of three dimensional magnetohydrodynamics: nonlinear stability of Alfv\'{e}n waves}

\author[Ling-Bing HE]{Ling-Bing He}
\address{Department of Mathematical Sciences, Tsinghua University\\ Beijing, China}
\email{lbhe@math.tsinghua.edu.cn}

\author[Li XU]{Li Xu}
\address{LSEC, Institute of Computational Mathematics, Academy of Mathematics and Systems Science, Chinese Academy of Science\\ Beijing, China}
\email{xuliice@lsec.cc.ac.cn}

\author[Pin YU]{Pin Yu}
\address{Department of Mathematical Sciences, Tsinghua University\\ Beijing, China}
\email{pin@math.tsinghua.edu.cn}

\thanks{The authors would like to thank Kevin G. Luli for numerous suggestions on improving the manuscript.
}

\begin{abstract}
Magnetohydrodynamics (MHD) studies the dynamics of magnetic fields in electrically conducting fluids. In addition to the sound wave and electromagnetic wave behaviors, magneto-fluids also exhibit an interesting phenomenon: They can produce the Alfv\'en waves, which were first described in a physics paper by Hannes Alfv\'en in 1942. Subsequently, Alfv\'en was awarded the Nobel prize for his fundamental work on MHD with fruitful applications in plasma physics, in particular the discovery of Alfv\'en waves.

\medskip
This work studies (and constructs) global solutions for the three dimensional incompressible MHD systems (with or without viscosity) in strong magnetic backgrounds. We present a complete and self-contained mathematical proof of the global nonlinear stability of Alfv\'en waves. Specifically, our results are as follows:
\medskip

\noindent $\bullet$   We obtain asymptotics for global solutions of the ideal system (i.e.,viscosity $\mu=0$) along characteristics; in particular, we have a scattering theory for the system.

\noindent $\bullet$  We construct the global   solutions (for small viscosity $\mu$) and we show that as $\mu\rightarrow 0$, the viscous solutions converge in the classical sense to the zero-viscosity solution. Furthermore, we have estimates on the rate of  the convergence in terms of $\mu$.

\noindent $\bullet$  We explain a linear-driving decay mechanism for viscous Alfv\'en waves with arbitrarily small diffusion. More precisely, for a given solution, we exhibit a time $T_{n_0}$ (depending on the profile of the datum rather than its energy norm) so that at time $T_{n_0}$ the $H^2$-norm of the solution is  small compared to $\mu$ (therefore   the standard perturbation approach can be applied to obtain the convergence  to the steady state afterwards).

\medskip

The results and proofs have the following main features and innovations:

\medskip

\noindent $\bullet$  We do not assume any symmetry condition on initial data. The size of initial data (and the \emph{a priori} estimates) does not depend on viscosity $\mu$. The entire proof is built upon the basic energy identity.

\noindent $\bullet$  The Alfv\'en waves do not decay in time: the stable mechanism is the separation (geometrically in space) of left- and right-traveling Alfv\'en waves. The analysis of the nonlinear terms are analogous to the \emph{null conditions} for non-linear wave equations.

\noindent $\bullet$  We use the (hyperbolic) energy method. In particular, in addition to the use of usual energies, the proof relies heavily on the energy flux through characteristic hypersurfaces.

\noindent $\bullet$  The viscous terms are the most difficult terms since they are not compatible with the hyperbolic approach.  We obtain a new class of space-time weighted energy estimates for (weighted) viscous terms. The design of weights is one of the main innovations and it unifies the hyperbolic energy method and the parabolic estimates.

\noindent $\bullet$  The approach is \emph{`quasi-linear'} in nature rather than a linear perturbation approach:  the choices of the coordinate systems, characteristic hypersurfaces, weights and multiplier vector fields depend on the solution itself. Our approach is inspired by Christodoulou-Klainerman's proof of the nonlinear stability of Minkowski space-time in general relativity.
\end{abstract}

\maketitle

\tableofcontents

\section{Introduction}

Magnetohydrodynamics (MHD) studies the dynamics of magnetic fields in electrically conducting fluids. It has wide and profound applications to plasma physics, geophysics, astrophysics, cosmology and engineering. In most interesting physical applications, one uses low frequency/velocity approximations so that one may focus on the mutual interaction of magnetic fields and the fluid (or plasma) velocity field. As the name indicates, MHD is in the scope of fluid theories so that it has many similar wave phenomena as usual fluids do.  Roughly speaking, the most common restoring forces for perturbations in fluid theory is the gradient of the fluid pressure and the sound waves are the corresponding wave phenomena. In addition to the fluid pressure, the magnetic field in MHD provides two forces: the magnetic tension force and the magnetic pressure force. The magnetic pressure plays a similar role as the fluid pressure and it generates (fast and slow) magnetoacoustic waves (similar to sound waves). The magnetic
tension force is a restoring force that acts to straighten bent magnetic field lines and it leads to a new wave phenomenon, to which there is no analogue in the ordinary fluid theory. The new waves are called Alfv\'{e}n waves, named after the Swedish plasma physist Hannes Olof G\"{o}sta Alfv\'{e}n.  On 1970, H. Alfv\'en was awarded the Nobel prize for his \emph{`fundamental work and discoveries in magnetohydrodynamics with fruitful applications in different parts of plasma physics'}, in particular his discovery of Alfv\'en waves \cite{Alfven} in 1942.

\bigskip

We discuss briefly the physical origin of Alfv\'en waves. For detailed descriptions, the reader may consult the original paper \cite{Alfven} or text books on MHD, e.g., \cite{Davidson}. One can think of Alfv\'{e}n waves as vibrating strings or more precisely transverse inertial waves. In a electrically conducting fluid, if the conductivity is sufficiently high, one will observe that the magnetic field lines tend to be \emph{frozen} into the fluid. In other words, the fluid particles tend to move along the magnetic field lines. Therefore, we may suppose that the fluid lies along a steady constant magnetic field $B_0$ and we perturb the fluid by a small velocity field $v$ which is perpendicular to $B_0$. The magnetic field line will be swept along with the fluid and the resulting curvature of the lines provides a restoring force (magnetic tension force) on the fluid. The fluid will eventually go back to the rest state and then the Faraday tensions will reverse the flow. The waves developed by the oscillations
are precisely the Alfv\'en waves. According to this description, Alfv\'en wave is different from sound waves and electromagnetic waves. It is driven by the Lorentz force.

\medskip

We now give a heuristic description for Alfv\'en waves. Let $B_0=(0,0,1)$ be a constant magnetic field along the $x_3$-axis. We assume that the fluids are frozen along the magnetic lines. Let $v=(0,\Delta v,0)$ be an infinitesimal velocity perturbation (perpendicular to $B_0$) for a fluid particle. Therefore, the Lorentz force on the particle is proportional to $v\times B = (\Delta v, 0, 0)$. After a small time $\Delta t$, the Lorentz force leads to a velocity change proportional to $v_1=(v\times B)\Delta t = (\Delta v\Delta t, 0, 0)$ in $x_1$ direction. Likewise, the velocity component $v_1$ provides the Lorentz force $v_1\times B =(0,-\Delta v\Delta t,0)$, which is opposite to the initial velocity perturbation. Thus, it acts as a restoring force to push the particle back to the original position; hence, waves develop.

An Alfv\'{e}n wave is a transverse wave. It propagates anisotropically in the direction of the magnetic field. In other words, the motion of the fluid particles (such as ions) and the perturbation of the magnetic field are in the same direction and transverse to the direction of propagation. It also propagates the incompressibility, involving no changes in plasma density or pressure. We remark that, in contrast, the magnetoacoustic waves reflect the compressibility of the plasma.

The theory of Alfv\'en waves supports the existing explanations for the origin of the earth's magnetic field. The magnetic fields have an ability to support two inertial waves, the Alfv\'{e}n waves and the magnetostropic waves (involving the Coriolis force). Both of the inertial waves are of considerable importance in the geodynamo-theory and they are useful in explaining the maintenance of the earth's magnetic fields in terms of a self-excited fluid dynamo. Alfv\'{e}n waves are also fundamental in the astrophysics, particularly topics such as star formation, magnetic field oscillation of the sun, sunspots, solar flares and so on.

\medskip

In \cite{Alfven}, when Alfv\'{e}n first discovered the waves named after him, he also provided a formal linear analysis. He considered the following situation: the conductivity is set to be infinite, the permeability is $1$ and the background constant magnetic field $B_0$ is homogenous and parallel to the $x_3$-axis of the space. He then took the plane waves ansatz by assuming all the physical quantities depending only on the time $t$ and the variable $x_3$. The MHD equations (see also \eqref{MHD general} below) become
\begin{equation*}
-\f{4\pi\rho}{B_0^2}\f{\partial^2b}{\partial t^2}+\f{\partial^2b}{\partial x_3^2}=0,
\end{equation*}
where $b$ is the magnetic field and $\rho$ is the plasma density. This is a $1+1$ dimensional wave equation and it implies immediately that the Alfv\'en waves move along the $x_3$-axis (in both directions) with the velocity (so called the Alfv\'en velocity) $V_A=\f{B_0}{\sqrt{4\pi\rho}}$. The linear analysis also indicates that the Alfv\'en waves are dispersionless. In the real world, the MHD waves obey the nonlinear dynamics and many of them detected sofar seem to be stable, such as the solar wind and waves generated by a solar flare rapidly propagating out across the solar disk. It is surprising that Alfv\'{e}n's linear analysis provides a rather good approximation for nonlinear evolutions. The nonlinear terms may pose serious difficulties in the mathematical studies of the propagation of the Alfv\'en waves in the MHD system, especially in the dispersionless situation.  One of the main objects of the paper is to analyze the relationship between the genuine nonlinear evolution and the linearized analysis.

\bigskip

The phenomena for the Alfv\'{e}n waves are ubiquitous and complex. The existing mathematical theories on Alfv\'en  waves are mostly concerning the linearized equations and are far from being complete. In the present work, we study the incompressible fluids and consider the nonlinear stability of the Alfv\'{e}n waves. The word `stability' roughly means the following two things:  1) the asymptotics of the waves as $t\rightarrow \infty$  for the ideal case (no viscosity);  2) the asymptotics for the viscous waves as the viscosity $\mu\rightarrow 0$ and as $t\rightarrow \infty$. In particular, our work will provide a way to justify why the linearized Alfv\'en waves provide a good approximation for the nonlinear evolution and how the viscosity damps the Alfv\'en waves--two interesting phenomena commonly described in text books on MHD, e.g., \cite{Davidson}, but there is no rigorous mathematical explanation for the phenomena.

\bigskip

Next, we write down the incompressible MHD equations. For simplification, we assume that both the fluid (plasma) density  and the permeability equal $1$. Then the incompressible MHD equations read
\begin{equation}\label{MHD general}
\begin{split}
\partial_t  \v+ \v\cdot \nabla \v &= -\nabla p + (\na\times\b)\times\b + {\mu \triangle \v}, \\
\partial_t \b + \v\cdot \nabla \b &=  \b \cdot \nabla \v + {\mu \triangle \b},\\
\div \v &=0,\\
\div \b &=0,
\end{split}
\end{equation}
where $\b$ is the magnetic field, $\v,\ p$ are the velocity and scalar pressure of the fluid respectively, $\mu$ is the viscosity coefficient or equivalently the dissipation coefficient.

We can write the Lorentz force term $(\na\times\b)\times\b$ in the momentum equation in a more convenient form. Indeed, we have
$$(\nabla\times\b)\times\b=-\na(\f12|\b|^2)+\b \cdot \nabla \b.$$
The first term $\na(\f12|\b|^2)$ is called the magnetic pressure force since it is in the gradient form just as the fluid pressure does. The second term $\b\cdot\na\b=\na\cdot(\b\otimes\b)$ is the magnetic tension force, which produces Alfv\'{e}n waves. Therefore, we can use $p$ again in the place of $p+\f12|\b|^2$. The momentum equation then reads
$$
\partial_t  \v+ \v\cdot \nabla \v = -\nabla p + \b \cdot \nabla \b + {\mu \Delta \v}.
$$

We study the most interesting situation when a strong back ground magnetic field $B_0$ presents (to generate Alfv\'en waves). Heuristically, if $\mu$ is large, the influence of $\v$ on $\b$ is negligible, the magnetic field is dissipative in nature (so that the magnetic disturbance $b-B_0$ tends to decay very fast). If $\mu$ is small, the velocity $\v$ will strongly influence $\b$ so that the situation is similar to ideal Alfv\'en waves and the damping from the dissipations is so weak that it takes a long time to see the effect. We will rigorously justify these facts later on. The heuristics can be depicted as follows:

\medskip

\includegraphics[width = 6.1 in]{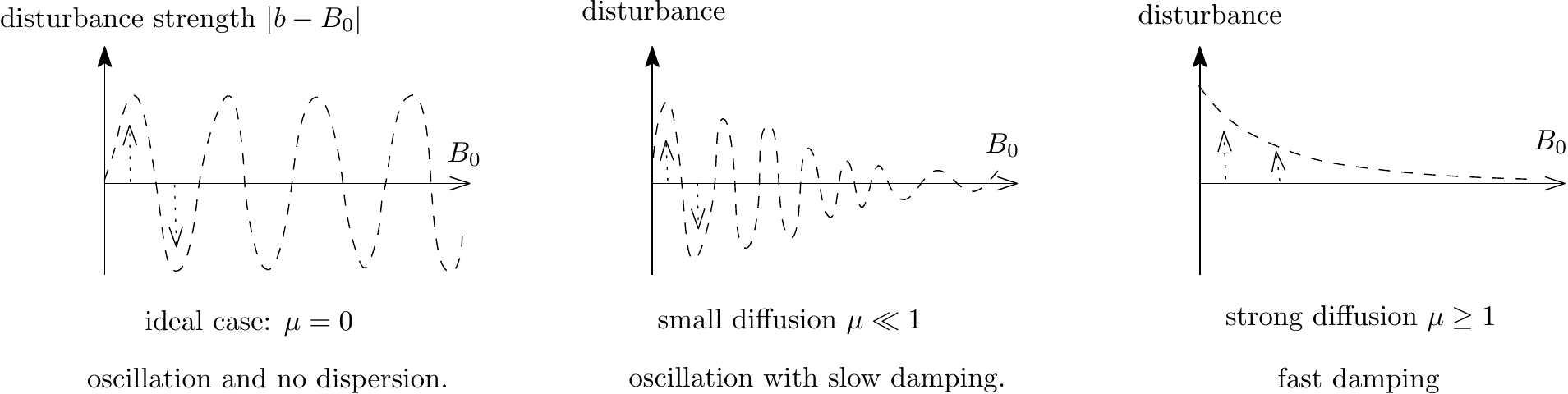}

\medskip

We now give a formal (or linear analysis) discussion about the properties showed in the above figures. Let $B_0=|B_0| \,\vv e_3$ be a uniform constant (non-vanishing) background magnetic field. The vector $\vv e_3$ is the unit vector parallel to $x_3$-axis. We remark that the pair $(0,B_0)$ solves the incompressible MHD system. We consider an infinitesimal perturbation $(v,b-B_0)$ of  $(0,B_0)$.  We take $v$ to be perpendicular to $B_0$. The leading order terms of the MHD system satisfy the following system of equations:
\beno\begin{aligned}
&\p_t\v-B_0\cdot\na b=-\na p+\mu\Delta\v,\\
&\p_t\b-B_0\cdot\na\v= \mu\Delta\b.
\end{aligned}\eeno
We remark that for convenience we do not distinguish $b$ from $b-B_0$ because they have the same derivatives. Taking the $\curl$ of the above equations, we obtain the vorticity equations, namely, for $\omega=\curl\v$ and $j=\curl\b$, we have
\beq\label{linear curl equation}\begin{aligned}
&\p_t\omega-B_0\cdot\na j=\mu\Delta\omega,\\
&\p_tj-B_0\cdot\na\omega=\mu\Delta j.
\end{aligned}\eeq
Alternatively, since $\nabla p$ is a quadratic term, we can ignore it for linear analysis.

We study the dispersion relation $f(\xi)$ of the above linearized equations  \eqref{linear curl equation}. Considering the plane wave solutions
\beno
\omega=\omega_0\exp{[i(\xi \cdot x-f(\xi)t)]},\quad j=j_0\exp{[i(\xi \cdot x-f(\xi)t)]},
\eeno
we obtain
\beno
f(\xi)^2+2i\mu|\xi|^2f(\xi)-(|B_0|^2 \xi_3^2+\mu^2|\xi|^4)=0,
\eeno
or equivalently,
\beq\label{dispersion relation}
f(\xi)=-i \mu |\xi|^2 \pm|B_0|\xi_3.
\eeq
We remark that according to the physics literatures, the plane waves with dispersive relation
$$ f^2(\xi)-|B_0|^2 \xi^2_3=0$$
are called Alfv\'{e}n waves, i.e., $\mu=0$. We study the following three cases for \eqref{dispersion relation} and this analysis can also be found in \cite{Davidson}.
\medskip
\begin{enumerate}
\item[Case-1]
The ideal case $\mu=0$. We have $$f(\xi)=\pm|B_0|\xi_3.$$
Both the phase velocity $\frac{f(\xi)}{|\xi|}$ and group velocity $\nabla_\xi f(\xi)$ are $v_A=|B_0|$, i.e., the Alfv\'{e}n velocity. It represents two families of plane waves propagating in the direction (or the opposite direction) of the magnetic field with velocity $v_A$. There is no dispersion. This corresponds to the first situation in the previous figure.

\item[Case-2] The case when $1>>\mu>0$ is small. We have a closed form for $f(\xi)$. In fact, we have
\beno
f(\xi)=-i\mu|\xi|^2\pm v_A \xi_3.
\eeno
It represents plane waves propagating in the direction (or the opposite direction) of the magnetic field with velocity $v_A$ and damped by a weak dissipation ($\mu<<1$).

\item[Case-3] The case $\mu >> 1$. We have
\beno
f(\xi)\sim -i\mu|\xi|^2.
\eeno
It represents the situation that the disturbance damped rapidly by the dissipations. This corresponds to the third drawing in the previous figure.
\end{enumerate}

\medskip

The third case corresponds to systems with strong diffusion. The mathematical analysis of such systems is analogous to the small data problem for the classical Navier-Stokes equations in three dimensional space. Since the theory is rather classical and well-understood, we will not consider the case in the paper. In the first two cases, the plane waves can travel across a vast distance before we see a significant effect of damping caused by the dissipation. The wave patterns can survive for a long time, which is approximately at least of time scale $O(\f{1}{\mu})$. We will provide a rigorous justification for Case-1 and Case-2 in the nonlinear setting.

\subsection{Main theorem (first version) and previous works}

We recall that, by incorporating the magnetic pressure into the fluid pressure, we can rewrite the incompressible MHD equations as
\begin{equation}\label{MHD original}
\begin{split}
\partial_t  \v+ \v\cdot \nabla \v &= -\nabla p + \b \cdot \nabla \b + {\mu \triangle \v}, \\
\partial_t \b + \v\cdot \nabla \b &=  \b \cdot \nabla \v + {\mu \triangle \b},\\
\div \v &=0,\\
\div \b &=0,
\end{split}
\end{equation}
where the viscosity $\mu$ is either $0$ or a small positive number. We introduce the Els\"{a}sser variables:
\begin{equation*}
\Zp = \v +\b, \ \ \Zm = \v-\b.
\end{equation*}
Then the MHD equations \eqref{MHD original} read
\begin{equation}\label{MHD in Elsasser}
\begin{split}
\partial_t  \Zp +\Zm \cdot \nabla \Zp  - {\mu \triangle \Zp}&= -\nabla p, \\
\partial_t  \Zm +\Zp \cdot \nabla \Zm - {\mu \triangle \Zm}&= -\nabla p,\\
\div \Zp &=0,\\
\div \Zm &=0.
\end{split}
\end{equation}
We use $B_0 =|B_0|(0,0,1)$ to denote a uniform background magnetic field and  we define
\begin{equation*}
\zp = \Zp-B_0, \ \ \zm = \Zm + B_0.
\end{equation*}
The MHD equations can then be reformulated as
\begin{equation}\label{MHD}
\begin{split}
\partial_t  \zp +\Zm \cdot \nabla \zp - {\mu \triangle \zp} &= -\nabla p, \\
\partial_t  \zm +\Zp \cdot \nabla \zm - {\mu \triangle \zm}&= -\nabla p,\\
\div \zp &=0,\\
\div \zm &=0.
\end{split}
\end{equation}

For a vector field $X$ on $\mathbb{R}^3$, its curl is defined by $\curl X=(\p_2X^3-\p_3X^2,\p_3X^1-\p_1X^3,\p_1X^2-\p_2X^1)$ or $\curl X = \varepsilon_{ijk}\partial_iX^j \partial_k$. We use the Einstein's convention: if an index appears once up and once down, it is understood to be summing over $\{1,2,3\}$.

By taking curl of \eqref{MHD}, we derive the following system of equations for $(\jp, \jm)$:
\begin{equation}\label{main equations for j}
\begin{split}
\partial_t  \jp +\Zm \cdot \nabla \jp - {\mu \triangle \zp}&= -\nabla \zm \wedge \nabla \zp, \\
\partial_t  \jm +\Zp \cdot \nabla \jm - {\mu \triangle \zm}&= -\nabla \zp \wedge \nabla \zm,
\end{split}
\end{equation}
where
\begin{equation*}
\jp = \curl \zp, \ \ \jm = \curl \zm.
\end{equation*}
We remark both $\jp$ and $\jm$ are divergence free vector fields. The explicit expressions of the nonlinearities on the righthand side are
\begin{equation}
\nabla \zm \wedge \nabla \zp =\varepsilon_{ijk}\partial_i \zm^l\partial_l \zp^j\partial_k, \ \ \nabla \zp \wedge \nabla \zm =\varepsilon_{ijk}\partial_i \zp^l\partial_l \zm^j\partial_k.
\end{equation}

\medskip

Before introducing more notations, we now provide a first version of our main theorem. It is a rough version in the sense that it only states the global existence part of the result. We will give more precise versions of the main theorem later on. The main result can be stated as follows:

\begin{thm}[First version]\label{the first version of the main theorem}
Let $B_0 =(0,0,1)$ be a given background magnetic field. Given constants $R\geq 100$ and $N_* \in \mathbb{Z}_{\geq 5}$, there exists a constant $\varepsilon_0$ so that  for all given smooth vector fields $(\v_0(x),\widetilde{b}_0)(x)$ on $\mathbb{R}^3$ with the following bound
\begin{align*}
\|(\log(R^2+|x|^2)^{\f12})^2(\v_0,\widetilde{b}_0)\|_{L^2(\mathbb{R}^3)}^2
&+\sum_{k=0}^{N_*}\|(R^2+|x|^2)^{\f12}(\log(R^2+|x|^2)^{\f12})^2\na^{k+1}(\v_0,\widetilde{b}_0)\|_{L^2(\mathbb{R}^3)}^2\\
&+\mu\|(R^2+|x|^2)^{\f12}(\log(R^2+|x|^2)^{\f12})^2\na^{N_*+2}(\v_0,\widetilde{b}_0)\|_{L^2(\mathbb{R}^3)}^2 \leq \varepsilon_0^2,
\end{align*}
for the initial data (to the MHD system \eqref{MHD original}) of the form
$$v(0,x)=v_0(x), \ \  b(0,x)= B_0 + \widetilde{b_0}(x),$$
the MHD system  \eqref{MHD original} admits a unique global smooth solution. In particular, the constant $\varepsilon_0$ is independent of the viscosity coefficient $\mu$.
\end{thm}
\begin{remark}
The proof for the viscous case when $\mu>0$ is in fact considerably harder than the ideal case $\mu=0$. This seems to contradict the intuition that diffusions help the system to stabilize (This intuition will be proved and justified towards the end of the paper). In the statement of the theorem, the weight functions for $(v,b)$ are different from those for the higher order terms. If $\mu=0$, we can choose the weights in a uniform way and in a much simpler form. However, if $\mu>0$, the choice of different weights plays an essential role in the proof and it unifies the hyperbolic estimates (for waves) and the parabolic estimates (for diffusive systems). This is one of the main innovations of the paper and we will explain this point when we discuss the ideas of the proof.
\end{remark}
\begin{remark} From now on, we will only consider the case where $|B_0|=1$. We can also use $B_0=|B_0|(0,0,1)$ to model the constant background magnetic field. The choice of the constant $\varepsilon_0$ will depend on $|B_0|$ but not on the viscosity $\mu$.
\end{remark}

\bigskip

We end this subsection by a quick review of the results on three dimensional incompressible MHD systems with strong magnetic backgrounds. Bardos, Sulem and Sulem \cite{Bardos} first obtained the global existence in the H\"{o}lder space $C^{1,\al}$ (not in the energy space) for the ideal case $(\mu=0)$. They do not treat the case with small diffusion, which we believe is fundamentally different from the ideal case. For the case with strong fluid viscosity but without Ohmic dissipation,   \cite{Xu-Zhang}(see also \cite{Lin-Xu-Zhang}) studies the small-data-global-existence with very special choice of data. We remark that the smallness of the data depends on the viscosity, while the data in the current work are independent of the viscosity coefficient $\mu$. Technically speaking, the work \cite{Bardos} treats the system as one dimensional wave equations and it relies on the convolution with fundamental solutions; the work \cite{Xu-Zhang} observes that the system can be roughly regarded as a damped wave equation
in
Lagrangian coordinates $\p_t^2Y-\mu\Delta\p_tY-\p_3^2Y\approx 0$ and the proof is based on Fourier analysis (more precisely on Littlewood-Paley decomposition).

\medskip

The proof, which will be presented in the sequel, is different from the aforementioned approaches. We will regard the MHD system as a system of $1+1$ dimensional wave equations and the proof makes essential use of the fact that the system is defined on three dimensional space. We derive energy estimates purely in physical space. The characteristic geometry (see next subsection) defined by two families of characteristic hypersurfaces of nonlinear solutions underlies the entire proof. The approach is in nature \emph{quasi-linear} and is similar in spirit to the proof of the nonlinear stability of Minkowski spacetime \cite{Ch-K}. In order to make this remark transparent, we first introduce the underlying geometric structure defined by a solution of \eqref{MHD in Elsasser}.

\subsection{The characteristic geometries}
We study the spacetime $[0,t^*] \times \mathbb{R}^3_{x_1,x_2,x_3}$ associated to a solution $(\v,\b)$ of the MHD equations or equivalently \eqref{MHD in Elsasser}. More precisely, we assume a smooth solution $(\v,\b)$ exists on $[0,t^*] \times \mathbb{R}^3$ and we study the foliation of the characteristic hypersurfaces associated to $(\v,\b)$. We recall that  $[0,t^*] \times \mathbb{R}^3$ admits a natural time foliation  $\bigcup_{0\leq t \leq t^*} \Sigma_t$, where $\Sigma_t$ is the constant time slice (in particular, $\Sigma_0$ is the initial time slice where the initial data are given).

\medskip

We first define two characteristic (spacetime) vector fields $\Lp$ and $\Lm$ as follows
\begin{equation}
\Lp = T + \Zp, \ \ \Lm = T + \Zm,
\end{equation}
where the time vector field $T$ is the usual $\partial_t$ defined in the Cartesian coordinates (we also use the same notations to denote the partial differential operators $\Lp=\p_t+\Zp\cdot\na$ and $\Lm=\p_t+\Zm\cdot\na$).

\medskip

Given a constant $c$, we use $S_{0,c}$ to denote the 2-plane $x_3 = c$ in $\Sigma_0$. Therefore, $\bigcup_{ x_3 \in \mathbb{R}} S_{0,x_3}$ is a foliation of the initial hypersurface $\Sigma_0$.
We define the characteristic hypersurfaces $\Cp_{x_3}$ and $\Cm_{x_3}$ to be the hypersurfaces emanated from $S_{0,x_3}$ along the vector fields $\Lp$ and $\Lm$ respectively. A better way to define $\Cpm$ is to understand the hypersurface as the level set of a certain function. We define the optical function $\up=\up(t,x)$ as follows
\begin{equation}\label{definition for up}
\begin{split}
\Lp \up = 0,\ \ \ \
\up\big|_{\Sigma_0} = x_3.
\end{split}
\end{equation}
Similarly, we define the optical function $\um$ by
\begin{equation}\label{definition for um}
\begin{split}
\Lm \um = 0,\ \ \ \ \um\big|_{\Sigma_0} = x_3.
\end{split}
\end{equation}
Therefore, the characteristic hypersurfaces $\Cp_{x_3}$ and $\Cm_{x_3}$ are the level sets $\{\up = x_3\}$ and $\{\um = x_3\}$ respectively. We will use the notations $\Cp_{\up}$ and $\Cm_{\um}$ to denote them. By construction, $\Lp$ is tangential to $\Cp_{\up}$ and $\Lm$ is tangential to $\Cm_{\um}$.

We remark that the spacetime $[0,t^*] \times \mathbb{R}^3$ admits two characteristic foliations: $\bigcup_{\up \in \mathbb{R}}\Cp_{\up}$ and $\bigcup_{\um \in \mathbb{R}}\Cm_{\um}$. The intersection $\Cp_{\up} \bigcap \Sigma_t$ is a two-plane, denoted by $\Sp_{t,\up}$. Similarly, we denote  $\Cm_{\um} \bigcap \Sigma_t$ by $\Sm_{t,\um}$. Therefore, for each $t$, we obtain two foliations $\bigcup_{\up \in\mathbb{R}}\Sp_{t,\up}$ and $\bigcup_{\um \in\mathbb{R}}\Sm_{t,\um}$ of $\Sigma_t$. In general, they may differ from each other.

Similar to the definitions of $u_\pm$, we also define $x_1^{\pm}=x_1^{\pm}(t,x)$ and $x_2^{\pm}=x_2^\pm(t,x)$. For $i=1$ or $2$, we require
\begin{equation}\label{definition for xipm}
\begin{split}
\Lp x_i^+ &= 0,\ \ \ \Lm x_i^- = 0\\
x_i^+\big|_{\Sigma_0} &= x_i, \ \ \ x_i^-\big|_{\Sigma_0} = x_i.
\end{split}
\end{equation}
We remark that if we let $i=3$ in the above defining formulas, we obtain $x_3^\pm =u^\pm$.

\bigskip

We use the following pictures to illustrate the above geometric constructions:

\medskip

\ \ \ \ \ \ \ \ \ \ \ \ \ \ \ \includegraphics[width = 4.5 in]{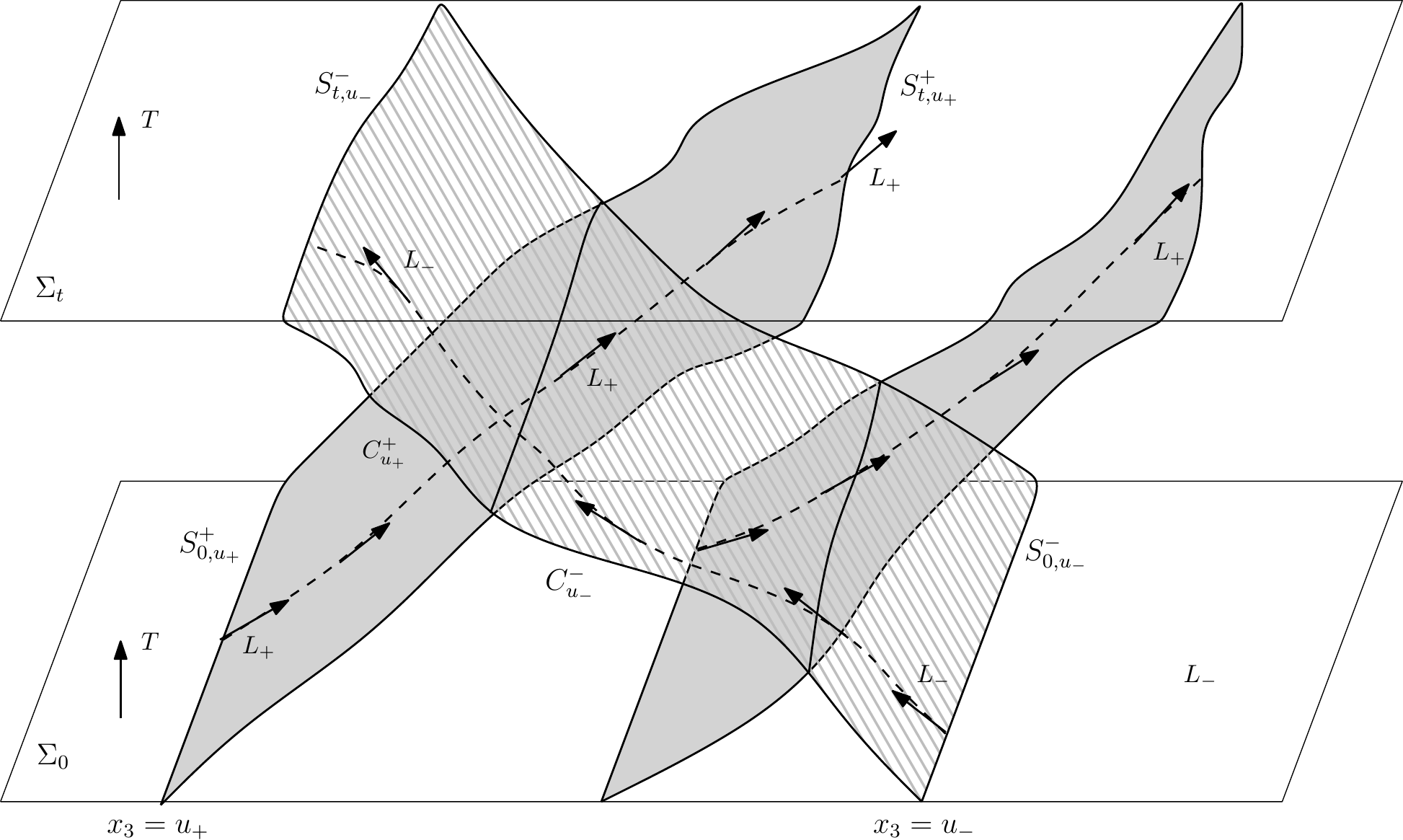}

The right-traveling hypersurfaces $\Cp_{\up}$ are painted grey; the left-traveling hypersurfaces $\Cm_{\um}$ are tiled with grey lines. The dashed lines are integral curves of either $\Lp$ or $\Lm$.

In order to specify the region where the energy estimates are taken place, for $t$, $\up^1$, $\up^2$, $\um^1$ and $\um^2$ given with $\up^1 < \up^2$ and $\um^1 < \um^2$, we define the following hypersurfaces / regions:
\begin{equation*}
\begin{split}
\Sigma_{t}^{[\up^1,\up^2]} = \bigcup_{\up \in [\up^1,\up^2]}\Sp_{t,\up}, \ \ W_{t}^{[\up^1,\up^2]} = \bigcup_{\tau \in [0,t]}\Sigma_{\tau}^{[\up^1,\up^2]},\\
\Sigma_{t}^{[\um^1,\um^2]} = \bigcup_{\um \in [\um^1,\um^2]}\Sm_{t,\um}, \ \ W_{t}^{[\um^1,\um^2]} = \bigcup_{\tau \in [0,t]}\Sigma_{\tau}^{[\um^1,\um^2]}.
\end{split}
\end{equation*}
Roughly speaking, $W_{t}^{[\up^1,\up^2]} = \bigcup_{\tau \in [0,t]}\Sigma_{\tau}^{[\up^1,\up^2]}$ is the spacetime region bounded by the two grey hypersurfaces in the above picture.

\medskip

As a subset of $\mathbb{R}^4$, the domain $W_{t}^{[\up^1,\up^2]}$ or $W_{t}^{[\um^1,\um^2]}$  admits a standard Euclidean metric. By forgetting the $x_1$ and $x_2$ axes, the outwards normals of the boundaries of the above domains are depicted schematically as follows:

\includegraphics[width = 5 in]{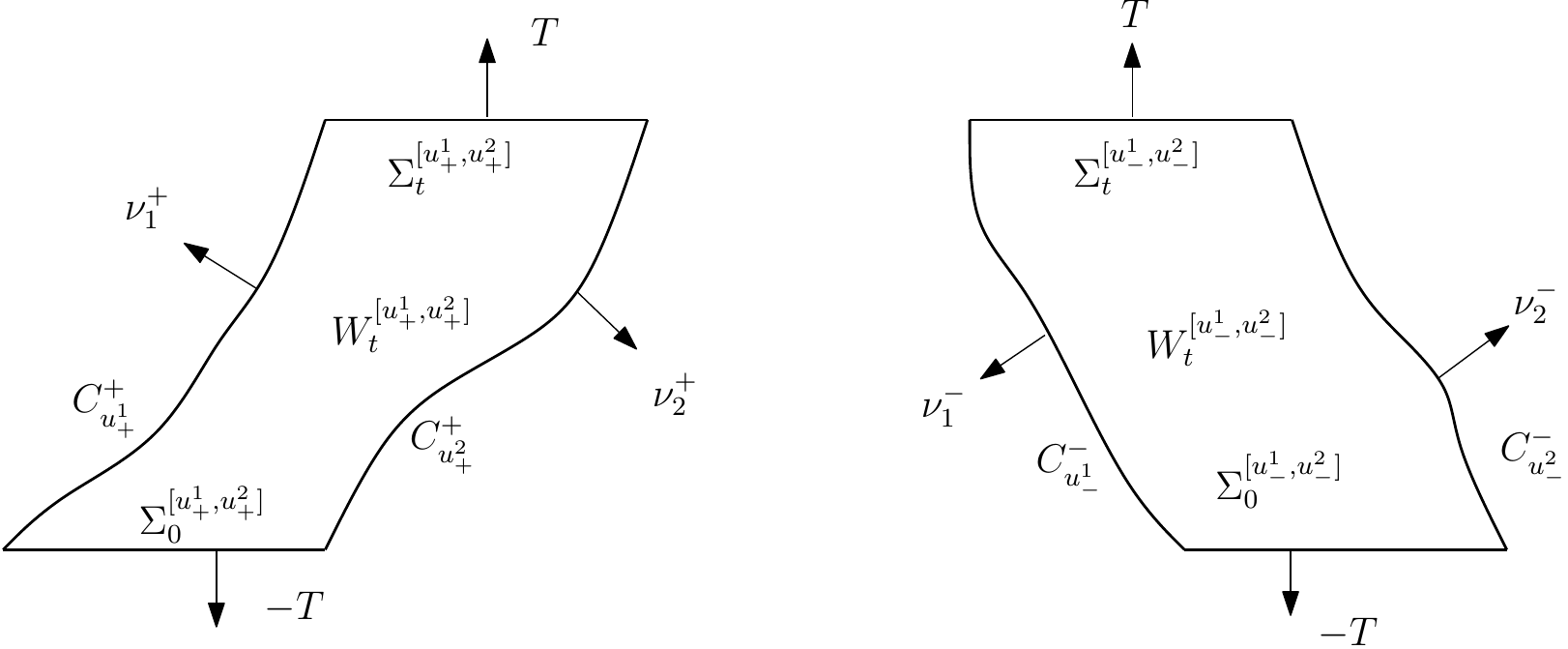}

The outward unit normal of $\Sigma_{0}$ and $\Sigma_{t}$ are $-T$ and $T$ respectively. We use $\nup_1$ to denote the outward unit normal of $\Cp_{\up^1}$. Since $\Cp_{\up^1}$ is the level set of $\up$, we have
\begin{equation*}
\nup_1 = -\frac{(\partial_t \up,\nabla \up)}{\sqrt{(\partial_t \up)^2+ |\nabla \up|^2}}
\end{equation*}
Similarly, for the outward unit normals $\nup_2$, $\num_1$ and $\num_2$ of $\Cp_{\up^2}$,$\Cm_{\um^1}$ and $\Cm_{\um^2}$ respectively, we have
\begin{equation*}
\nup_2 = \frac{(\partial_t \up,\nabla \up)}{\sqrt{(\partial_t \up)^2+ |\nabla \up|^2}}, \ \ \num_1 = -\frac{(\partial_t \um,\nabla \um)}{\sqrt{(\partial_t \um)^2+ |\nabla \um|^2}}, \ \ \num_2 = \frac{(\partial_t \um,\nabla \um)}{\sqrt{(\partial_t \um)^2+ |\nabla \um|^2}}.
\end{equation*}

\subsection{Main theorems (second version)} The notation $a\lesssim b$ means that there exists a universal constant $C$ such that $a\le Cb$. We use the notation $C_{\omega_1,\omega_2, \cdots}$ to represent the constant  that depends on the parameters $\omega_1, \omega_2, \cdots$.
For a multi-index $\alpha=(\alpha_1,\alpha_2,\alpha_3)$ with $\alpha_i\in\mathbb{Z}_{\geq 0}$, we define
$\zpm^{(\alpha)}=\big(\frac{\partial{}}{\partial {x_1}}\big)^{\alpha_1}\big(\frac{\partial{}}{\partial {x_2}}\big)^{\alpha_2}\big(\frac{\partial{}}{\partial {x_3}}\big)^{\alpha_3}\zpm$; for a positive integer $k$, we define $|\zpm^{(k)}|=(\sum_{|\alpha|=k}|\zpm^{(\alpha)}|^2)^{\frac12}$. One can also define $\jpm^{(\alpha)}$ and $|\jpm^{(k)}|$ in a similar way. Let $R$ and $\varepsilon_0$ be two positive numbers. They will be determined later on. In principle, $R$ is large and $\varepsilon_0$ is small.

We introduce two weight functions $\wp$ and $\wm$ as follows
\begin{equation*}
 \wp=\big(R^2 + |x_1^+|^2+|x_2^+|^2+|\up|^2\big)^\frac{1}{2}, \ \ \ \wm=\big(R^2 + |x_1^-|^2+|x_2^-|^2+|\um|^2\big)^\frac{1}{2}.
\end{equation*}
We remark that $L_+ \wp =0$ and $L_- \wm =0$.

For a given multi-index $\alpha$, we define the energy $E_{\mp}^{(\alpha)}$ and flux $F_{\mp}^{(\alpha)}$ (associated to characteristic hypersurfaces) of the solution $\zpm$ as follows:
\begin{align*}
&E_{\mp}^{(\alpha)}(t) = \int_{\Sigma_t} \wpm^2 \big(\log\wpm\big)^4 |\nabla z_{\mp}^{(\alpha)}|^2 dx,  \ \ \ |\alpha|\ge0,\\
&F_{\mp}^{(0)}(\na z_{\mp})=\int_{C_{u_{\mp}}^{\mp}} \wpm^2 \big(\log\wpm\big)^4 |\na z_{\mp}|^2d\sigma_{\mp},\\
&F_{\mp}^{(\alpha)}(j_{\mp})=\int_{C_{u_{\mp}}^{\mp}} \wpm^2 \big(\log\wpm\big)^4 |j_{\mp}^{(\alpha)}|^2d\sigma_{\mp},\ \ \ |\alpha|\geq1,
\end{align*}
where $d\sigma_{\pm}$ is the surface measure of the characteristic hypersurface $C_{u_{\pm}}^\pm$. We define the diffusion $D_{\mp}^{(\alpha)}$ as follows
$$
D_{\mp}^{(\alpha)}(t)=\mu\int_0^t\int_{\Sigma_\tau}\wpm^2 \big(\log\wpm\big)^4 |\nabla^2 z_{\mp}^{(\alpha)}|^2 dxd\tau, \
\ \ |\alpha|\geq0.
$$
We remark that, for $|\al|\geq1$, the flux parts contain only the vorticity component rather than the full derivatives of $\nabla z_\pm^{(\alpha)}$. This is a technical choice that makes it easier to deal with the nonlinear contribution from the pressure term. If we consider the energy identities \eqref{energy estimates fm in Wp}-\eqref{energy estimates fm in Wm}(see below), the corresponding weight functions $\lambdap$ and $\lambdam$ will be $\wm^2 \big(\log\wm\big)^4$ and $\wp^2 \big(\log\wp\big)^4$. In particular, we have $\Lp \lambdam=0$ and $\Lm \lambdap =0$.

The lowest order energy and flux are defined as
\begin{equation*}
E_{\mp}(t) = \int_{\Sigma_t} \big(\log\wpm\big)^4 | z_{\mp}|^2 dx, \ \ \ F_{\mp}(z_{\mp})=\int_{C_{u_{\mp}}^{\mp}}\big(\log\wpm\big)^4 |z_{\mp}|^2d\sigma_{\mp}.
\end{equation*}
The lowest order diffusion is defined as
\begin{equation*}
 D_{\mp}(t)=\mu\int_0^t\int_{\Sigma_\tau}(\log \wpm)^4 |\nabla z_{\mp}|^2 dxd\tau.
\end{equation*}
In view of the energy identities \eqref{energy estimates fm in Wp}-\eqref{energy estimates fm in Wm}, the corresponding weight functions $\lambdap$ and $\lambdam$ will be $\big(\log\wm\big)^4$ and $\big(\log\wp\big)^4$. The constraints $\Lp \lambdam=0$ and $\Lm \lambdap =0$ still hold.
\begin{remark}
Unlike the usual choice, the weight functions $\wpm$ indeed depend on the solutions $z_\pm$. This reflects the quasilinear nature of the problem.
\end{remark}
\begin{remark}
The weight functions for the lowest order energy and flux are different from those for higher order energy and flux. The difference is exactly $\wpm^2$. These special weights are designed to control the diffusion terms $\mu \triangle \zpm$. Indeed, for the ideal MHD system ($\mu=0$), we can choose the weight functions in a much simpler and uniform manner, say $\wpm = (R^2+|u_{\pm}|^2)^{\f{1+\delta}{2}}$ for some small $\delta >0$. The choice of different weights is essential to the proof and it incorporates the hyperbolic and parabolic estimates at the same time. Since we use different weights and consider a hyperbolic-parabolic mixed situation, we also say the energy estimates are hybrid.
\end{remark}

To make the statement of the energy estimates simpler, we introduce the total energy norms, total flux norms and total  diffusions as follows:
\begin{align*}
&E_{\mp} = \sup_{0\leq t\leq t^*} E_{\mp}(t),\ \ E_{\mp}^k = \sup_{0\leq t\leq t^*} \sum_{|\alpha|=k}E_{\mp}^{(\alpha)}(t), \\
&F_{\mp}= \sup_{u_{\mp} \in \mathbb{R}} F_{\mp}(z_{\mp}), \ \ F_{\mp}^{0}= \sup_{u_{\mp} \in \mathbb{R}} F_{\mp}^{0}(\na z_{\mp}), \\
&F_{\mp}^{k}= \sup_{u_{\mp} \in \mathbb{R}} \sum_{|\alpha|=k} F_{\mp}^{(\alpha)}(j_{\mp}),\ \ D_{\mp}^k=\sum_{|\alpha|=k}D_{\mp}^{(\alpha)}(t^*).
\end{align*}

The first theorem is about the global existence to the MHD system \eqref{MHD in Elsasser} with small $\mu\geq 0$.

\begin{thm}[Second version with \emph{a priori} estimates]\label{global existence for MHDmu}
Let $B_0 =(0,0,1)$, $R= 100$ and $N^* \in \mathbb{Z}_{\geq 5}$. There exists a constant $\varepsilon_0$, which is independent of the viscosity coefficient $\mu$, such that if the initial data of \eqref{MHD original} or equivalently \eqref{MHD in Elsasser} satisfy
\begin{align*}
\mathcal{E}^\mu(0)&=\sum_{+,-}\Bigl(\|(\log(R^2+|x|^2)^{\f12})^2\zpm(0,x)\|_{L_x^2 }^2
+\sum_{k=0}^{N_*}\|(R^2+|x|^2)^{\f12}(\log(R^2+|x|^2)^{\f12})^2\na^{k+1}\zpm(0,x)\|_{L_x^2 }^2\\
&
\ \ \ \ \ \ \ \ \ +\mu\|(R^2+|x|^2)^{\f12}(\log(R^2+|x|^2)^{\f12})^2\na^{N_*+2}\zpm(0,x)\|_{L^2_x}^2\Bigr)\leq\varepsilon_0^2,
\end{align*}
then \eqref{MHD in Elsasser} admits a unique global solution $\zpm(t,x)$. Moreover, there exists a constant $C$ independent of $\mathcal{E}^\mu(0)$ and $\mu$, such that the solution $\zpm(t,x)$ enjoys the following energy estimate:
\beq\label{global energy estimate for MHDmu}\begin{aligned}
&\sup_{t\geq 0}\bigl(E_{\pm}(t)+\sum_{k=0}^{N_*}E_{\pm}^k(t)+\mu E_{\pm}^{N_*+1}(t)\bigr)+\sup_{u_{\pm}\in \mathbb{R}}\bigl(F_{\pm}(\zpm)+F_{\pm}^0(\na\zpm)+\sum_{k=1}^{N_*}F_{\pm}^k(j_{\pm})\bigr)\\
&\qquad+\bigl(D_{\pm}+\sum_{k=0}^{N^*}D_{\mp}^k+\mu D_{\mp}^{N^*+1}\bigr)\big|_{t^*=\infty} \leq C \mathcal{E}^\mu(0).
\end{aligned}\eeq
\end{thm}

As a direct consequence of the above theorem, we obtain the global existence result for ideal MHD for the data with the following  bound
\begin{align*}
\sum_{+,-}\Bigl(\|(\log(R^2+|x|^2)^{\f12})^2\zpm(0,x)\|_{L^2(\mathbb{R}^3)}^2
+\sum_{k=0}^{N_*}\|(R^2+|x|^2)^{\f12}(\log(R^2+|x|^2)^{\f12})^2\na^{k+1}\zpm(0,x)\|_{L^2(\mathbb{R}^3)}^2\Bigr)\leq\varepsilon_0^2.
\end{align*}
Due to the absence of the viscous terms, we can actually do much better. As we mentioned above, the different weights on $\zpm$ and higher derivatives of $\zpm$ are designed to deal with the small diffusions. Roughly speaking, when we derive the usual (hyperbolic or wave) energy estimates, the procedure of integrations by parts acting on the viscosity term will generate a linear term. This term is extremely difficult to control. It mirrors the fact that the hyperbolic type of energy estimates is not entirely compatible with the small diffusions. This is one of the main difficulties of the problem. When $\mu=0$, we are free of the above contraint and we can use much simpler choices of weights, such as $(R^2+|u_\pm|^2)^{\frac{1+\delta}{2}}$ or  $(R^2+|x_1^\pm|^2+|x_2^\pm|^2+|u_\pm|^2)^{\frac{1+\delta}{2}}$. This leads to the following theorem:

\begin{thm}[Global existence for ideal MHD]\label{global existence for ideal MHD with small amplitude}
Let $\mu=0$, $B_0 =(0,0,1)$, $\delta \in(0,1)$, $R= 100$ and $N^* \in \mathbb{Z}_{\geq 5}$. There exists a constant $\varepsilon_0$, such that if the initial data of \eqref{MHD original} or equivalently \eqref{MHD in Elsasser} satisfy
$$
\mathcal{E}^{\mu=0}(0) = \sum_{+,-}\sum_{k=0}^{N_*+1}\|(R^2+|x_3|^2)^{\f{1+\delta}{2}}\nabla^{k}\zpm(0,x)\|_{L^2(\mathbb{R}^3)}^2\leq\varepsilon_0^2,
$$
the ideal MHD system \eqref{MHD in Elsasser} ($\mu=0$) admits a unique global solution $\zpm(t,x)$. Moreover, there is a universal constant $C$, so that, for all $k\leq N_*$, we have
\begin{align*}
\sup_{t\geq 0}\|(R^2+|u_\mp|^2)^{\f{1+\delta}{2}}\na^{k+1}\zpm(t,x)\|_{L^2(\mathbb{R}^3)}^2
+\sup_{u_{\pm}}\int_{C_{u_\pm}^\pm}(R^2+|u_\mp|^2)^{1+\delta}|\zpm|^2d\sigma_\pm &\\
+\sup_{u_{\pm}}\int_{C_{u_\pm}^\pm}(R^2+|u_\mp|^2)^{1+\delta}|j_{\pm}^{(k)}|^2d\sigma_\pm
&\leq C \mathcal{E}^{\mu=0}(0).
\end{align*}
\end{thm}

\begin{remark}
For the ideal MHD system, we could prove more stronger existence results in the sense that the  weighted $L^2$ condition on $\zpm$  can be removed  in  Theorem \ref{global existence for MHDmu} and Theorem \ref{global existence for ideal MHD with small amplitude}. The key point lies in the  proof to Theorem \ref{global existence for MHDmu}  and Theorem \ref{global existence for ideal MHD with small amplitude}. In fact, the lowest order energy estimates of $\zpm$ are not needed for the ideal MHD under the assumption
$\|\zpm(t,\cdot)\|_{L^\infty}\leq\f12$. Thanks to the Gagliardo-Nirenberg interpolation inequality, $\|\zpm(t,\cdot)\|_{L^\infty}$ can be bounded by $C\|\na\zpm(t,\cdot)\|_{L^2}^{\f12}\|\na^2\zpm(t,\cdot)\|_{L^2}^{\f12}(\lesssim\varepsilon)$ which is enough to close the argument by the continuity method.  This is merely a technical improvement and we will not pursue this direction in the paper.
\end{remark}

As applications of the above theorems, we are now ready to study the nonlinear asymptotic stability of Alfv\'en waves.

\subsection{Nonlinear stability of ideal Alfv\'en waves: a scattering picture}
We now focus on the ideal incompressible MHD system. The goal is to understand the global dynamics of the Alfv\'{e}n waves, or equivalently the asymptotics of $\zpm$ for $t\rightarrow \infty$. For this purpose, we introduce a so-called \emph{scattering diagram} for the Alfv\'en waves. The idea is to capture the behavior of waves along each characteristic curves. It is similar to the \emph{Penrose diagram} in general relativity (which keeps record of the null/characteristic geometry of the spacetime).

\medskip

\ \ \ \ \ \ \  \includegraphics[width = 5.5 in]{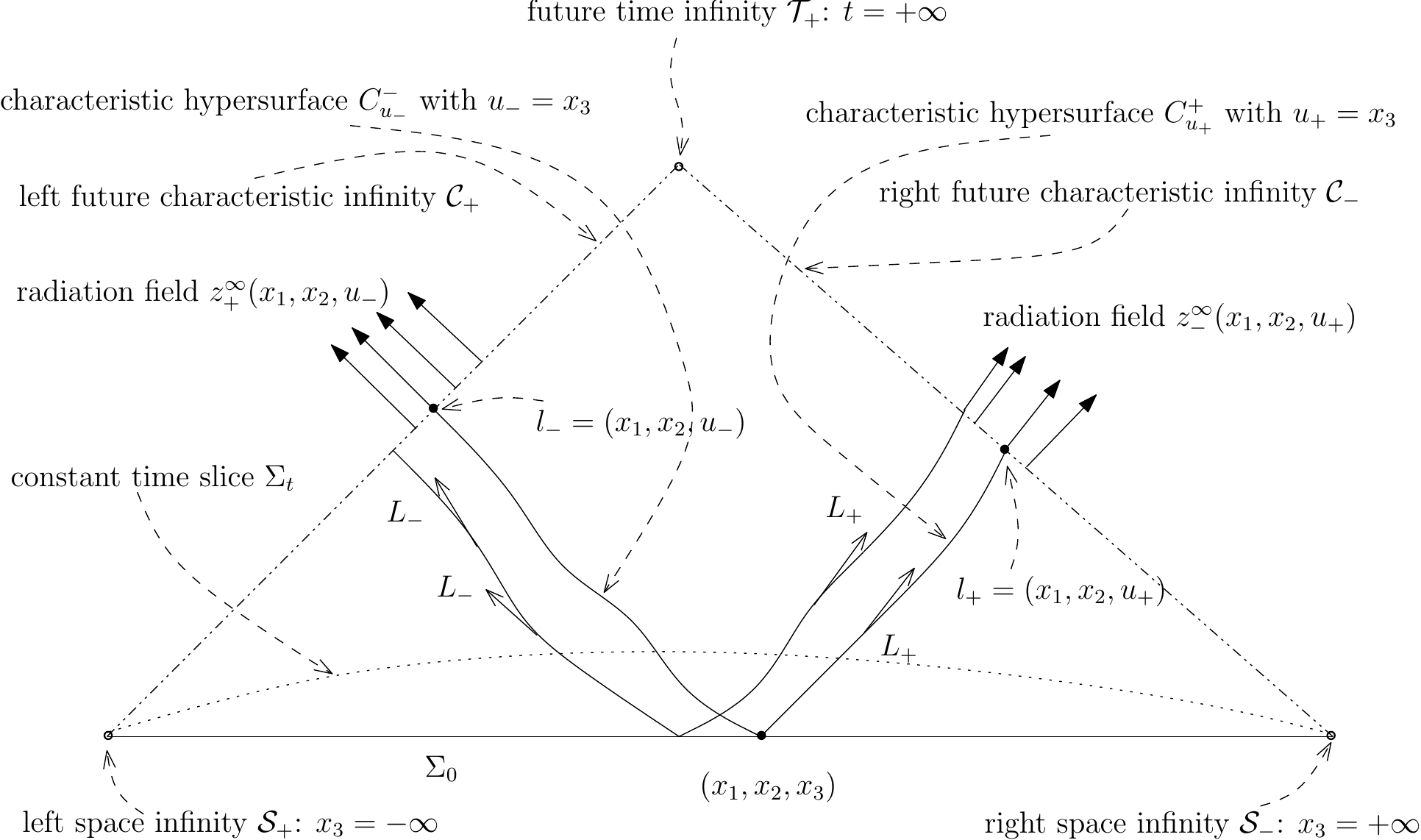}

\medskip

Given a point $(x_1,x_2,x_3) \in \Sigma_0$, it determines uniquely a left-traveling characteristic line: it is parameterized by $(x_1,x_2,u_-,t)$, where $u_-= x_3$ and $t \in [0,+\infty)$. This line is denoted by $l_-(x_1,x_2,u_-)$ (with $u_-=x_3$) or simply $l_-$. We use $\mathcal{C}_+$ to denote the collection of all the characteristic lines and we call it the left future characteristic infinity. We use $(x_1,x_2,u_-)$ as a global coordinate system on $\mathcal{C}_+$ so that $\mathcal{C}_+$ can be regarded as a differentiable manifold. In the picture, $\mathcal{C}_+$ is depicted as the double-dotted dashed line on the left hand side. The picture shows that $l_-$ starts from $(x_1,x_2,x_3) \in \Sigma_0$ and hits $\mathcal{C}_+$ at $(x_1,x_2,u_-)$ with $u_- = x_3$. The tangent vector field of the line $l_-$ is exactly $L_-$. We remark that a line $l_-(x_1,x_2,u_-)$ lies on the characteristic hypersurface $C_{u_-}^-$. The intersection of $C_{u_-}^-$ with $\mathcal{C}_+$ should be understood as the collection
of all the $l_-(x_1,x_2,u_-)$'s, where $u_- \in \mathbb{R}$.

Similarly, we can also define the right future characteristic infinity $\mathcal{C}_-$ as the collection of all the right-traveling characteristic lines.

We use $\mathcal{T}_+$ to denote the virtual intersection of $\mathcal{C}_+$ and $\mathcal{C}_-$ in the picture. We call it the future time infinity since it represents morally $t\rightarrow +\infty$. Besides $\mathcal{T}_+$, $\mathcal{C}_{+}$ has another endpoint $\mathcal{S}_+$ in the picture. It represents the left space infinity, i.e., $x_3 \rightarrow -\infty$. Similarly, we can define the right space infinity $\mathcal{S}_-$. For an arbitrary time slice $\Sigma_t$, it is depicted by the horizontal dotted line in the picture. We remark that each $\Sigma_t$ ends at $\mathcal{S}_-$ and $\mathcal{S}_+$.

We can now define the scattering fields $z_+^{\text{(scatter)}}(x_1,x_2,u_-)$ on $\mathcal{C}_+$ and $z_-^{\text{(scatter)}}(x_1,x_2,u_+)$ on $\mathcal{C}_-$:

\begin{definition}Given points $l_\mp \in \mathcal{C}_\pm$ with coordinates $(x_1,x_2, u_\mp)$, the corresponding \emph{scattering field of the ideal Alfv\'en waves} for the solutions $z_\pm$  are defined by the following formulas
\begin{equation}\label{scattering fields}
\begin{split}
z_+^{\text{(scatter)}}(x_1,x_2,u_-)= \lim_{t\rightarrow \infty} z_+(x_1,x_2,u_-, t),\\
z_-^{\text{(scatter)}}(x_1,x_2,u_+)= \lim_{t\rightarrow \infty} z_-(x_1,x_2,u_+, t).
\end{split}
\end{equation}
\end{definition}
Similarly, we also introduce the scattering vorticities (and their derivatives) as limits of the corresponding objects along the characteristics:
\begin{equation}\label{vorticity of scattering fields}
\begin{split}
(\curl z_+^{\text{(scatter)}})(x_1,x_2,u_-)= \lim_{t\rightarrow \infty} (\curl z_+)(x_1,x_2,u_-, t),\\
(\curl z_-^{\text{(scatter)}})(x_1,x_2,u_+)= \lim_{t\rightarrow \infty}(\curl z_-)(x_1,x_2,u_+, t).
\end{split}
\end{equation}
\begin{remark}[Notation Convention] We would like to avoid confusions when we switch between coordinates. Given a vector field $f$ on $\mathbb{R}_t\times \mathbb{R}^3$, $\nabla f$, $\div f$ or $\curl f$ are defined on each time slice with respect to the standard coordinates $(x_1,x_2,x_3)$. Geometrically, they are defined with respect to the standard Euclidean metric on $\Sigma_t$. It is in this sense that they are globally defined, in particular, are independent of the choices of coordinates. On the other hand, for the quantities defined as scattering limit (e.g. $\curl z_+^{\text{(scatter)}}$), the corresponding $\nabla$, $\div$ and $\curl$ are merely symbols rather than having any geometric meanings.

To better illustrate the idea,  we consider some examples.

1)  $\na p$ are understood as vector field in $\R^4$ and it is coordinate independent. More precisely, we can write $(\nabla p)(t, x_1^+, x_2^+, x_3^+)$. It simply means the vector field $\nabla p$ evaluated at the point $(t, x_1^+, x_2^+, x_3^+)$ rather than $(\partial_t p, \partial_{x_1^+}p, \partial_{x_2^+}p, \partial_{x_3^+}p)$.

2)  $\zp$ are obviously global defined as the real physical objects. If we change coordinates according to $\Phi:\, (y_0,y_1,y_2,y_3)\mapsto (t,x_1,x_2,x_3)$, then $\zp(y_0,y_1,y_2,y_3)=\zp|_{(t,x_1,x_2,x_3)=\Phi(y_0,y_1,y_2,y_3)}$ represents the same vector field on the same space-time point.
\end{remark}

In physics, the scattering fields have more pratical/physical meaning than the original fields. They are the fields received and measured by a far-away observer. Based on Theorem \ref{global existence for ideal MHD with small amplitude}, we will prove that the scattering fields are \emph{well-defined}. In fact, we will prove that $\nabla p$ is integrable over each $l_\pm$ and the scattering fields are given by the following explicit formulas:
\begin{equation}\label{formula for scatter +}
z_+^{\text{(scatter)}}(x_1,x_2, u_-) = z_+(x_1,x_2, u_-, 0) -\int_{0}^\infty (\nabla p) (x_1,x_2, u_-,\tau) d\tau,
\end{equation}
and
\begin{equation}\label{formula for scatter -}
z_-^{\text{(scatter)}}(x_1,x_2, u_+) = z_-(x_1,x_2, u_+, 0) -\int_{0}^\infty (\nabla p) (x_1,x_2, u_+,\tau) d\tau.
\end{equation}

The vorticities of the scattering fields can be written down explicitly:
\beq\label{expression for vorticity of scatter +}
(\curl z_+^{(\text{scatter})})(x_1,x_2, u_-)=(\curl z_+)(x_1,x_2, u_-,0) -\int_{0}^\infty (\na z_-\wedge\na z_+) (x_1,x_2, u_-,\tau) d\tau.
\eeq
and
\beq\label{expression for vorticity of scatter -}
(\curl z_-^{(\text{scatter})})(x_1,x_2, u_+)=(\curl z_+)(x_1,x_2, u_+,0) -\int_{0}^\infty (\na z_+\wedge\na z_-) (x_1,x_2, u_+,\tau) d\tau.
\eeq

\medskip

The above analysis also provides a framework, via the scattering fields, to compare the nonlinear Alfv\'en waves with the linearized theory of Alfv\'en waves (\emph{\`a la Alfv\'en}). For the linearized theory, one assumes that $\v\cdot \nabla \v \sim 0$, $\nabla p \sim 0$ and $\b \cdot \nabla \sim B_0\cdot \nabla$ (they are of order $O(\varepsilon_0^2)$ in the nonlinear evolution). The linearized ideal MHD system reduces to
\begin{equation*}
\begin{split}
\partial_t  \v-& B_0 \cdot \nabla \b  = 0, \ \ \partial_t \b -B_0 \cdot \nabla \v = 0,
\end{split}
\end{equation*}
or equivalently,
\begin{equation*}
\begin{split}
\partial_t  \zp -&B_0 \cdot \nabla \zp  = 0, \ \ \partial_t  \zm +B_0 \cdot \nabla \zm = 0.
\end{split}
\end{equation*}

Given initial data $z_\pm(x_1,x_2,x_3,0)$, the linearized system can be solved directly by the method of characteristics. Therefore, the solutions of the linearized system can also define a similar scattering diagram as above. To give a precise description,  we first fix a measure $d\tilde{\sigma}_\pm$ on $\mathcal{C}_\pm$. By virtue of the coordinates $(x_1,x_2,u_\mp)$ on $\mathcal{C}_\pm$, we require that $d\tilde{\sigma}_\pm = dx_1\wedge dx_2 \wedge d u_\mp$. Intuitively, if we regard $\mathcal{C}_\pm$ as the limits of $C^\pm_{u_\pm}$, we would like to define the measure as limiting objects of $d\sigma_\pm$ on $C^\pm_{u_\pm}$ as $u_\pm\rightarrow\mp\infty$. Our definition may be different from the limiting measures by universal constants (thanks to the proof to Theorem \ref{global existence for ideal MHD with small amplitude}) and this will not effect any statement in this subsection. Then we introduce the following weighted Sobolev spaces:
\begin{align*}
H^{N_*+1,\delta}(\Sigma_0) &= \ \text{the completion of compactly supported smooth vector fields on }\  \mathbb{R}^3 \  \\
&\ \ \ \ \text{with respect to the norm} \sum_{k=0}^{N_*+1}\|(R^2+|x_3|^2)^{\f{1+\delta}{2}}\nabla^{k} f(x)\|_{L^2(\mathbb{R}^3)}^2,
\end{align*}
\begin{align*}
H&^{N_*+1,\delta}(\mathcal{C}_{\pm}) = \ \text{the completion of compactly supported smooth vector fields $f$ on }\  \mathbb{R}^3 \  \\
&\text{with respect to the norm} \int_{\mathcal{C}_{\pm}}(R^2+|u_\mp|^2)^{1+\delta}|f|^2d\tilde{\sigma}_\pm+ \sum_{k=0}^{N_*}\int_{\mathcal{C}_{\pm}}(R^2+|u_\mp|^2)^{1+\delta}|\nabla^k (\curl f)|^2d\tilde{\sigma}_\pm.
\end{align*}

We now define the following linear solution operator or linear scattering operator:
\begin{equation}\label{linear scattering operator}
\begin{split}
\mathbf{S}^{\text{linear}}:H^{N_*+1,\delta}(\Sigma_0)\times H^{N_*+1,\delta}(\Sigma_0)&\rightarrow H^{N_*+1,\delta}(\mathcal{C}_-) \times  H^{N_*+1,\delta}(\mathcal{C}_+),\\
(z^{(0)}_-,z^{(0)}_+)&\mapsto \big(z^{(0)}_-,z^{(0)}_+\big),
\end{split}
\end{equation}
where we identify $\Sigma_0$ with $\mathcal{C}_\pm$ by the coordinates $(x_1,x_2,x_3)\mapsto (x_1,x_2,u_\mp)$$(u_\mp=x_3)$.

For the nonlinear scattering theory, we can similarly define the nonlinear scattering operator as follows:
\begin{equation}\label{nonlinear scattering operator}
\begin{split}
\mathbf{S}:H^{N_*+1,\delta}(\Sigma_0)\times H^{N_*+1,\delta}(\Sigma_0)&\rightarrow H^{N_*+1,\delta}(\mathcal{C}_-) \times  H^{N_*+1,\delta}(\mathcal{C}_+),\\
(z^{(0)}_-,z^{(0)}_+)&\mapsto \big(z_-^{\text{(scatter)}},z_+^{\text{(scatter)}}\big),
\end{split}
\end{equation}
where $(z_-^{\text{(scatter)}}, z_+^{\text{(scatter)}})$ are the scattering fields associated to the initial data  $(z^{(0)}_-,z^{(0)}_+)$. By the \emph{a priori} estimates in Theorem \ref{global existence for ideal MHD with small amplitude}, $\mathbf{S}$ is an continous operator.

\medskip

\includegraphics[width = 6 in]{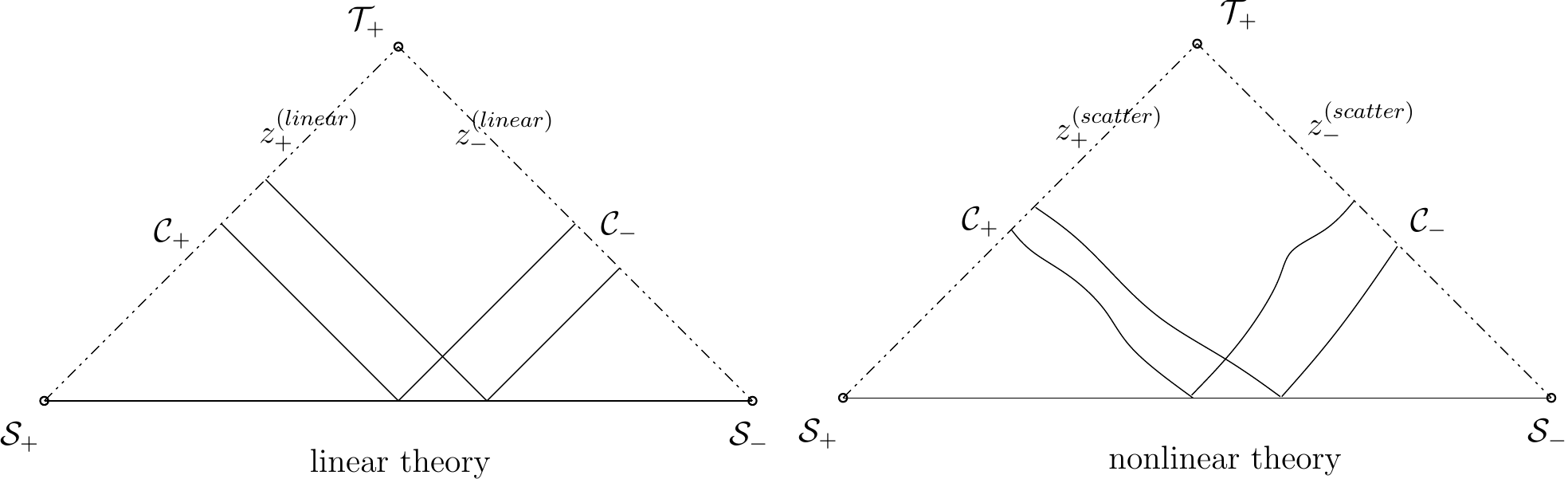}

\medskip

We compare the linear (scattering) theory and the nonlinear scattering theory. In the linear theory, we use $z_\pm^{\text{(linear)}}$ to denote the scattering fields. In the above pictures, the characteristic curves of the linearized equations are straight lines;  the characteristic curves of the nonlinear equations are curved lines. Since in both theories we can use $(x_1,x_2,u_\mp)$ as common coordinate systems for $\mathcal{C}_\pm$, we can compute the differences  $z_\pm^{\text{(scatter)}}-z_\pm^{\text{(linear)}}$ to quantify the difference between the linear theory and the nonlinear theory:
\begin{equation*}
\begin{split}
(z_\pm^{\text{(scatter)}}-z_\pm^{\text{(linear)}})(x_1,x_2, u_\mp,\tau)&= -\int_{0}^\infty (\nabla p) (x_1,x_2, u_\mp,\tau) d\tau\\
&= \int_{0}^\infty \big(\nabla\triangle^{-1}\p_i\p_j(z_-^iz_+^j)\big)(x_1,x_2, u_\mp,\tau) d\tau.
\end{split}
\end{equation*}
Therefore, the deviation of the nonlinear theory from the linearized theory reflects the nonlinear interactions between the nonlinear left-traveling wave $z_+$ and the nonlinear right-traveling wave $z_-$. Based on this formula, we show that the linearization of the nonlinear scattering operator is the linear scattering operator:
\begin{thm}\label{theorem scattering theory for ideal Alfven waves}
Assume the initial data of the ideal MHD system satisfy $\|\zpm\|_{H^{N_*+1,\delta}(\Sigma_0)} \leq \varepsilon_0$ with $N_*\geq 5$ and $\varepsilon_0$ being determined in Theorem \ref{global existence for ideal MHD with small amplitude}. Therefore, the scattering fields given in \eqref{scattering fields} is well defined. Similarly,  the scattering vorticities fields  given in
\eqref{vorticity of scattering fields} is also well-defined.
Moreover, regarded as operators between Hilbert spaces:
 $$H^{N_*+1,\delta}(\Sigma_0)\times H^{N_*+1,\delta}(\Sigma_0) \rightarrow H^{0,\delta}(\mathcal{C}_-) \times H^{0,\delta}(\mathcal{C}_+),$$
the differential of $\mathbf{S}$ at $\textbf{0} \in H^{N_*+1,\omega}(\Sigma_0)\times H^{N_*+1,\omega}(\Sigma_0)$ is equal to $\mathbf{S}^{\text{linear}}$, i.e.,
\begin{equation}
d \,\mathbf{S}	\big|_{\textbf{0}}=\mathbf{S}^{\text{linear}}.
\end{equation}
\end{thm}
\begin{remark}
The map $\mathbf{S}: H^{N_*+1,\delta}(\Sigma_0)\times H^{N_*+1,\delta}(\Sigma_0) \rightarrow H^{0,\delta}(\mathcal{C}_-) \times H^{0,\delta}(\mathcal{C}_+)$ considered in the theorem only addresses the $L^2$ norm of the scattering fields. Indeed, to recover all the derivatives at infinity, this motivates the study of the inverse scattering problem for the ideal Alfv\'en waves. Since the problem is of great independent interests and difficulties (in particular because this would be a quasi-linear type inverse scattering theory), we will discuss this issue in a forthcoming paper.
\end{remark}

\subsection{Nonlinear stability of viscous Alfv\'en waves}
The main application of the estimates given in Theorem \ref{global existence for MHDmu} is the study of global dynamics of  viscous Alfv\'en waves. The analysis of the Alfv\'en waves in the previous subsection is subject to the constraint that the MHD system is ideal. In reality, all the physical systems have diffusion phenomena and the corresponding wave phenomena will be damped by the diffusion.
\medskip

For the presentation of our main result, for a fixed $\mu$, we first introduce the so called \emph{the classical $\mu$-small-data parabolic regime for \eqref{MHD}}. Once the viscosity $\mu$ is given, as one usually does for the Navier-Stokes equations, one can regard \eqref{MHD} as \emph{semi-linear} heat equations rather than a quasi-linear system. Therefore, the classical approach for the Navier-Stokes equations shows that, there exists a constant $\varepsilon_\mu$, such that if the $H^2$-norm of the initial data are bounded above by $\varepsilon_\mu$, then we can construct global solutions of \eqref{MHD} by regarding the system as a small perturbation of the linearized equation. We remark that usually $\varepsilon_\mu =O(\mu)$. Intuitively, in the small-data parabolic regime, the diffusion is so strong (compared to the convection) so that the solution will stay in this regime and converge to the steady state of the system.

In Theorem \ref{global existence for MHDmu}, the size of initial data is of order $\varepsilon$. We emphasize that $\varepsilon$ is independent of $\mu$. Since $\mu$ can be arbitrarily small, we can think of the size of the data as being very large compared to $\varepsilon_\mu$. It is in this sense that the initial data given in Theorem \ref{global existence for MHDmu}  is \emph{far away} from the classical $\mu$-small-data parabolic regime.  Now the problem on global dynamics of  viscous Alfv\'en waves can be formulated that, given a $\mu$, how   and when the solution of the MHD system from a far away position will enter the small-data parabolic regime.

\medskip

To understand the mechanism of the small dissipation for the energy,  we begin with two families of special data to see the dissipative properties of the corresponding viscous solutions. Their behaviors are very different near the time $T_c=O(\frac{1}{\mu})$. This on one hand shows the rich dynamical phenomena of the viscous Alfv\'en waves; on the other hand, this shows that it is more natural to consider the small diffusion problem via the hyperbolic method rather than the parabolic method, since the dissipative property of a solution is sensitive to the initial data. In the following discussion, we assume that $\mu$ is given.

\medskip

{\bf Example 1:} The first family of data is  so-called the low frequency data or data with very small oscillations. We may take $$(v(0,x),b(0,x))=(\varepsilon^{5/2} f_1(\varepsilon^{} x),\varepsilon^{5/2} f_2(\varepsilon^{} x)),$$
where $f_1(x)$ and $f_2(x)$ are two compactly supported smooth divergence-free vector fields and $\varepsilon\leq \varepsilon_0$ measures the smallness of the data. According to the energy estimates in Theorem \ref{global existence for MHDmu}, one can show that
$$\int_{\mathbb{R}^3}(|\nabla v|^2+|\nabla b|^2)dx \lesssim\varepsilon^3.$$
Roughly speaking, the main reason of having $\varepsilon^3$ instead of $\varepsilon^2$ in the energy is that the initial data on one derivative of $(v,b)$ is one-order-in-$\varepsilon$ smaller than $(v,b)$ itself. According to the basic energy identity, we have
\begin{align*}
 \int_{\mathbb{R}^3}\bigl(|v(T_1,x)|^2+|(b-B_0)(T_1,x)|^2\bigr) dx &= \int_{\mathbb{R}^3}\bigl(|v(0,x)|^2+|(b-B_0)(0,x)|^2\bigr) dx\\
 & \ \ \ \ -2\mu\int_{0}^{T_1}\int_{\mathbb{R}^3}\bigl(|\nabla v(\tau,x)|^2+|\nabla b(\tau,x)|^2\bigr)dx d\tau\\
 &\geq \int_{\mathbb{R}^3}\bigl(|v(0,x)|^2+|(b-B_0)(0,x)|^2\bigr) dx-\mu T_1 \varepsilon^3.
\end{align*}
Since the initial energy is proportional to  $\varepsilon^2$, for the  time $T_1= \frac{1}{\mu}$, the dissipation of energy is approximately $\varepsilon^3$. Therefore, almost no energy has been consumed due to viscosity all the way up to the   time  $T_1$. In other words, for the data with very small oscillations, the dissipation on the waves is very weak within the  time  $T_1$ and the viscous waves resemble the ideal Alfv\'en waves.

\medskip

{\bf Example 2:} The second family of data contains considerable oscillations. We measure the oscillations by looking at the energies. Let $E_k (t)= \int_{\mathbb{R}^3}|\nabla^k v(t,x)|^2+|\nabla^k (b(t,x)-B_0)|^2 dx$. We assume that
$$E_0(0) \sim E_1 (0) \sim E_2(0).$$
We recall that for the low frequency data, we have $E_0(0) >> E_1 (0) >>E_2(0)$. To avoid dealing with too many constants, we further assume that $E_0(0) = E_1 (0) = E_2(0)=\varepsilon^2$ and the analysis for the general case is the same.
Similar to the analysis in the low frequency data case, since $E_2(t)\leq C\varepsilon^2$, we have
\begin{align*}
 E_1(t) &= E_1(0)  -2\mu\int_{0}^{t}E_2(\tau) d\tau \geq \varepsilon^2 -2C\mu \varepsilon^2 t.
\end{align*}
In fact, we have neglected the contribution of the nonlinear terms since they are all of order $\varepsilon^3$. Therefore,  we have
\begin{align*}
 E_1(t)  \geq \frac{1}{2}\varepsilon^2,\ \ \ \text{ for } t\leq\frac{1}{4C\mu}.
\end{align*}
This implies that, for $T_2=\frac{1}{4C\mu}$, we have
\begin{align*}
 E_0(T_2) &= E_0(0)  -2\mu\int_{0}^{T_2}E_1(\tau) d\tau \leq \varepsilon^2 -2\mu \frac{1}{2}\varepsilon^2 T_2=(1-\frac{1}{4C})\varepsilon^2.
\end{align*}
We remark that $C$ is the universal constant in the energy estimates in Theorem \ref{global existence for MHDmu}. This analysis shows that for oscillating data, within the  time $T_2$, a considerable amount of energy has been dissipated. Indeed, by suffering a loss of derivatives (since the viscous terms require one more derivative), we can further iterate the above analysis to amplify the dissipation. This shows that the highly oscillating solutions damp much faster than the low frequency data.

\medskip

These two examples show that on one hand, the viscous Alfv\'en waves (for small $\mu$) preserve the wave profile for a long time (approximately $\frac{1}{\mu}$) and the behavior of the waves in this regime is very similar to that of the ideal Alfv\'en waves; on the other hand, after a sufficiently long time ($> \frac{1}{\mu}$), the dissipation accumulates and the wave amplitude begins to dissipate and will eventually vanish. The time scale $T_c =O( \frac{1}{\mu})$ is called the characteristic time for the system which is also suggested by the physics (see \cite{Davidson}). It is roughly the time for the transition from non-dissipative wave like solutions to solutions of the heat equation (with fast decay in time).  It also indicates on when solutions decay to the $\mu$-small-data parabolic regime.

\medskip

The main theorem of the subsection is as follows:

\begin{thm}[Nonlinear stability of viscous Alfv\'en waves]\label{thm nonlinear stability of viscous waves}
Let $B_0 =(0,0,1)$, $\mu_0>0$, $R\ge100$ and $N_* \in \mathbb{Z}_{\geq 5}$. For all $\mu\leq \mu_0$, there exists a constant $\varepsilon_0$, which is independent of the viscosity coefficient $\mu$, so that if the initial data of \eqref{MHD original} or equivalently \eqref{MHD in Elsasser} satisfy
\begin{equation}\label{mu 00}
 \begin{split}
&\sum_{+,-}\Bigl(\|(\log(R^2+|x|^2)^{\f12})^2\zpm(0,x)\|_{L^2(\mathbb{R}^3)}^2
+\sum_{k=0}^{N_*}\|(R^2+|x|^2)^{\f12}(\log(R^2+|x|^2)^{\f12})^2\na^{k+1}\zpm(0,x)\|_{L^2(\mathbb{R}^3)}^2\\
&
\ \ \ \ \ \ \ \ \ +\mu\|(R^2+|x|^2)^{\f12}(\log(R^2+|x|^2)^{\f12})^2\na^{N_*+2}\zpm(0,x)\|_{L^2(\mathbb{R}^3)}^2\Bigr)\eqdefa \mathcal{E}^\mu(0)=\varepsilon^2\leq\varepsilon_0^2,
\end{split}
\end{equation}
then \eqref{MHD in Elsasser} admits a unique global solution $\zpm^\mu(t,x)$ or $(v^\mu,b^\mu)$. We remark that the solutions $\zpm^\mu(t,x)$ have the same initial data and we use $\zpm(t,x)$ or $(v,b)$to denote the solution corresponding to the ideal system. The solutions $\zpm^\mu(t,x)$ satisfy the following properties:
\begin{itemize}
 \item[1)]{\bf(Convergence to the ideal solution)} For any given $T>0$, we have
 \begin{equation}\label{vanishing viscosity limit}
 \|\zpm^\mu(t,x)-\zpm(t,x)\|^2_{L_t^\infty L^2_x\big([0,T]\times\mathbb{R}^3\big)} \lesssim \mu\varepsilon e^{\varepsilon T}.
 \end{equation}
 \item[2)]{\bf(Decay to the small-data parabolic regime)} We fix $\varepsilon_0$ (determined by Theorem \ref{global existence for MHDmu}) and fix the initial data $(z_+(x,0),z_-(x,0))$ so that it satisfies \eqref{mu 00}. We define the total energy $\mathcal{E}^\mu(t)$ as
$$
\mathcal{E}^\mu(t) = \sum_{+,-}\Bigl(E_\pm(t)+\!\!\!\!\!\!\sum_{|\alpha|\le N^*}\!\!\!\!E_{\pm}^{(\alpha)}(t)  +\mu\!\!\!\!\sum_{|\alpha|=N^*+1}\!\!\!\!E_{\pm}^{(\alpha)}(t)\Bigr),
$$
For arbitrary small $\mu>0$, there exist a universal constant $C$ and a sequence of time $T_1<T_2<\cdots<T_{n_0}$ in such way that, for any $k\leq n_0$, we have
$$\mathcal{E}^\mu(T_k) \leq  (C\mathcal{E}(0))^{\f{k}2+1}.$$
Moreover, $$\mathcal{E}^\mu(T_{n_0}) \leq \varepsilon_\mu.$$
In other words, at time $T_{n_0}$ the solution enters the $\mu$-small-data parabolic regime.
\end{itemize}
\end{thm}

The next figure shows the intuitive idea of the decay:

\medskip

\ \ \ \ \ \ \ \ \ \ \ \   \includegraphics[width = 3.5 in]{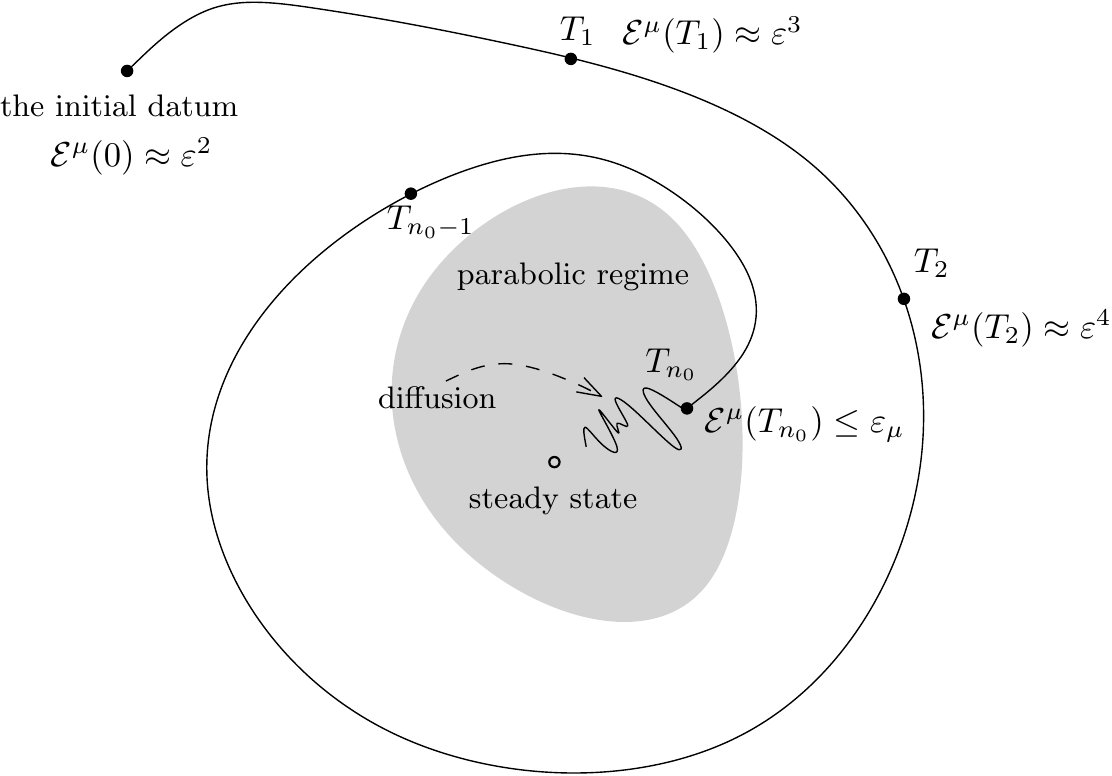}

\medskip
The gray region is the classical $\mu$-small-data parabolic regime in the energy space (roughly $H^2(\mathbb{R}^3)$) and the curve is the evolution curve of the solution.  The solution initially is far-way from the grey region. In the course of the evolution, the viscosity damps the total energy. The total energy may decay very slowly (the rate depends on the profile of the data) before the solution enters the parabolic regime. Once it enters the grey region at time $T_{n_0}$, the diffusion takes over and we see that the solution converges to the steady state (denoted by a circle in the figure) very fast.
\begin{remark}
It is routine to repeat the proof of \eqref{vanishing viscosity limit} to show that, for $k\leq 4$, we have
 $$\|\zpm^\mu(t,x)-\zpm(t,x)\|^2_{L_t^\infty H^k_x\big([0,T]\times\mathbb{R}^3\big)} \lesssim \mu\varepsilon e^{\varepsilon T}.$$
In particular, for any fixed time interval $[0,T]$, in the classical sense (with respect to the topology of ${L^\infty_t C_x^2([0,T]\times \mathbb{R}^3)}$) , we have
 $$\lim_{\mu\rightarrow 0} (v^\mu,b^\mu){\longrightarrow}(v,b).$$
 Moreover, for fixed $\mu$, it shows that viscous Alfv\'en waves are very close to the ideal Alfv\'en waves at least for $t\le |\log \mu-2\log \varepsilon|/\varepsilon$.
\end{remark}

\begin{remark}[Choice of $T_k$] We emphasize that the choice of $T_k$ depends not only on the size of energy norms of the initial data but also on the \emph{profile} of the data.

In the course of the proof of the above theorem (which will be at the end of the paper), we have to iterate the following decay estimates:
\ben\label{roughdecay}
 \mathcal{E}^\mu(t)\lesssim \mathcal{I}^\mu(t;0)+\f{\log\bigl(\log(\mu t+e)+e\bigr)}{\log(\mu t+e)}\mathcal{E}^\mu(0)
+\bigl(\mathcal{E}^\mu(0)\bigr)^{\f32}.
\een
The function $\mathcal{I}^\mu(t;0)$ is \emph{completely and explicitly determined} by the initial data in a straightforward manner. It has the property that  $\mathcal{I}^\mu(t;0)\rightarrow 0$ for $t\rightarrow \infty$.	 Roughly speaking, it measures the distribution of the data in the low frequency (in Fourier space) region. The exact form of the function is not enlightening so that we only give the expression in the proof.
\end{remark}

\begin{remark}[Comparison to known decay estimates which are based on the parabolic method] If we assume the initial data are in $L^1(\mathbb{R}^3)$, classical results such as  \cite{schonbeck} or \cite{Wiegner} suggest that the energy of the system should have the following decay estimates:
\beno
&&\|z_+(t)\|_{L^2}^2+\|z_-(t)\|_{L^2}^2\\&&\lesssim(1+\mu t)^{-\f32}\bigl( \|z_+(0)\|_{L^1\cap L^2}^2+\|z_-(0)\|_{L^1\cap L^2}^2\bigr)+\frac{(\|z_+(0)\|_{L^2}^2+\|z_-(0)\|_{L^2}^2)^2}{\mu^2} (1+\mu t)^{-\f12}.
\eeno
In the case where $E_0(0)=\varepsilon^2\gg  \mu^2$, for $t\ll \varepsilon^4/\mu^3$, we see that the upper bound of the energy from the above inequality is extremely large compared to $\varepsilon^2$. This cannot help to justify the characteristic time $T_c=O(\f1{\mu})$ as the physics suggested. Therefore, in a large time scale (up to time $\varepsilon^4/\mu^3$) the classical estimates do not capture the decay mechanism for the small diffusion.

In our approach, with the additional assumption that the datum is in $L^1(\mathbb{R}^3)$,  we can improve \eqref{roughdecay} to
$$
 \mathcal{E}^\mu(t)\lesssim \f{\log\bigl(\log(\mu t+e)+e\bigr)}{\log(\mu t+e)}\mathcal{E}^\mu(0)
+\bigl(\mathcal{E}^\mu(0)\bigr)^{\f32}.
$$
It is straightforward to see that, if $t\gg T_c=O(\f1{\mu})$, the total energy becomes $o(\varepsilon^2)$. Therefore, much energy has been
dissipated at $T_c$. This provides a theoretical support to the characteristic time and it is also consistent with the previous two examples.

We also want to point out that, in the classical estimates, the factor $\mu^{-2}$ makes the estimates rougher. It comes from the estimates for convection terms in the equations since they are treated as nonlinear terms.
Our approach is quasi-linear hyperbolic energy method and the convection terms do not contribute extra negative power of $\mu$.
\end{remark}

\begin{remark}
Once the solution enters the classical $\mu$-small-data regime, classical approach yields immediately the final decay rate for $t\rightarrow \infty$:
 $$\|(v^\mu,b^\mu-B_0)\|_{L^\infty(\mathbb{R}^3)}\lesssim  \frac{\mu}{(1+\mu t)^\frac{3}{4}}.$$
 In particular, due to the diffusion, $(v^\mu,b^\mu)$ converges to the steady state $(0,B_0)$.
\end{remark}

\medskip

\subsection{Comments on the proof} We would like to address the motivations for difficulties in the proof.

\noindent $\bullet$ {\bf Separation of Alfv\'en waves and null structures}

\medskip

\noindent A main difficulty in understanding the three dimensional Euler equations is the accumulation of the vorticity. In fact, the vorticity $\omega$ for incompressible Euler equations satisfies the following equation:
 $$\partial_t \omega + v \cdot \nabla \omega = \nabla v \wedge \nabla v.$$
 This is a transport type of equation and in general we do not expect decay in time for $\omega$. The righthand side can be roughly regarded as $|\omega|^2$ and this nonlinearity of Ricatti type is hard to control.

 In the current work, the strong magnetic background provides a cancellation structure for the nonlinear terms. It resembles the null structure (\`a la Klainerman) in many nonlinear wave equations. First of all, it is crucial to realize that the solutions are indeed waves (Alfv\'en waves) and we have two families of waves $z_+$ and $z_-$. The vorticity equations now read as (up to a sign)
 $$\partial_t j_\pm + Z_\mp \cdot \nabla j_\pm =-\nabla z_\pm \wedge \nabla z_\mp.$$
 Here, $z_+$ and $z_-$ are $1+1$ dimensional waves and we do not expect any decay in time for each of them (just as for Euler equations!). The remarkable fact is that $z_+$  and $z_-$ travel in opposite directions. Therefore, after a long time, $z_+$ and $z_-$ are far apart from each other and their distance can be measured by the time $t$. Therefore, the quadratic nonlinearity $\nabla z_+ \wedge \nabla z_-$ must be small (in sharp contrast to the Euler equations!) since $z_+$ and $z_-$ are basically supported in different regions. This observation provides the decay mechanism to control the nonlinear terms. We remark that, in the context of the standard null structure for wave equations, $z_+$ and $z_-$ can be regarded as incoming and outgoing waves and the null structure says that incoming waves can only couple with outgoing waves.

 More precisely, we have the following schematic equations:
$$\partial_t  \zp +\Zm \cdot \nabla \zp = \cdots,  \ \ \partial_t  \zm +\Zp \cdot \nabla \zm = \cdots,\ \ \Zm \sim -B_0, \ \ \Zp \sim B_0.$$
Therefore, we can roughly think of the waves $\zp$ and $\zm$ as follows:

a) $\zp$ travels along the $-B_0$ direction (we say that it is \emph{left-traveling}) with speed approximately $1$. It is centered around $(0,0,-t)$.

b) $\zm$ travels along the $B_0$ direction (we say that it is \emph{right-traveling}) with speed approximately $1$. It is centered around $(0,0,t)$.

The centers of $\zp$ and $\zm$ are moving away from each other. We will later on say that \emph{$\zpm$ separate from each other} to refer to this phenomenon. This picture indeed underlies each step in the proof. For instance, although $\nabla z_+$ and $\nabla z_-$ are not decaying in $L^\infty$ norm, but their product satisfies the following decay estimate:	
$$|\nabla\zp(t,x) \nabla\zm(t,x)|\lesssim \frac{1}{1+t}\big(\frac{1}{\log(2+t)}\big)^2.$$
Moreover, the decay is fast enough so that righthand side is integrable in $t$.
  \medskip

\noindent $\bullet$ {\bf Weighted estimates and $(1+1)-$dimension wave equations}

 \medskip

\noindent  As we have noted before, at least on the linearization level, the Alfv\'en waves $z_\pm$ satisfy $1+1$ dimensional wave equations.  It is well-known that $(1+1)$-dimension waves are conformally invariant (the energy-momentum tensor of a linear wave is trace-free!). We briefly recall the conformal structure of $(1+1)$-dimensional Minkowski space $(\mathbb{R}^{1+1}, m=-dt\otimes dt+dx\otimes dx)$. If we let $u=-t+x$ and $\underline{u}=t+x$, we have $m=\frac{1}{2}(du\otimes d\underline{u}+d\underline{u}\otimes d {u})$. The optical functions $u$ and $\ub$ are analogues of the Riemann invariants for the $2\times 2$ conservation laws and the defining functions for the characteristic surfaces $u_+$ and $u_-$ in the current paper are also similar. We define $L =\partial_t + \partial_x$ and $\underline{L} =\partial_t -\partial_x$ (analogues of $L_\pm$ in this paper). Therefore, all the conformal Killing vector fields on $\mathbb{R}^{1+1}$ are linear combinations of $f(u)L$ and $g(\underline{u})\underline{L}
$. The
associated energy current will provide conservation quantities (energies) for $(1+1)$-waves. In a more analytical way, the above analysis shows that, if one wants to define a good conserved energy, we can systematically multiply the equations by $f(u)\varphi$ or $g(\underline{u})\varphi$ ($\varphi$ is a solution of the wave equations) and then integrate by parts.

 This idea underlies all the energy estimates in the sequel. We will multiply the MHD equations by $f(u_-)z_+$ or $g(u_+)z_-$ to derive energy estimates. This leaves another important issue: the choices of the weights $f(u_-)$ and $g(u_+)$. This question is far from being trivial since we also have to take the viscosity terms into account. This term indeed prevents (via damping) the solution from behaving like $(1+1)$-waves (dispersionless). We will discuss this issue later on.

\medskip

\noindent $\bullet$ {\bf Energy flux through characteristic hypersurface}

\medskip

\noindent In the study of fluid problems, it is common to use energy associated to each slice $\Sigma_t$, e.g., the standard energy such as $\int_{\mathbb{R}^3}|v(t,x)|^2dx$. However, given the facts that the solutions are waves (Alfv\'en wave), there are other more natural energy type quantities, called the energy flux. In our work, the flux comes into play merely as auxiliary (except for the scattering picture for ideal Alfv\'en waves) quantities, but it is indeed indispensable for each step of the proof. The use of the flux is indeed one of the main innovations in our approach.

To make the meaning of flux more transparent, we consider the left-traveling characteristic hypersurface $C^-_{u_-}$. The associated energy flux for $z_-$ is defined as
$$F(z_-)=\int_{C^-_{u_-}} |z_-|^2 d\sigma_-,$$
where $d\sigma_-$ is the surface measure for $C^-_{u_-}$. Since $z_-$ is right-traveling and is transversal to $C^-_{u_-}$, the flux $F(z_-)$ measures exactly the amount of energy carried by $z_-$ through $C^-_{u_-}$.

Besides its clear physical meaning, the flux is a robust technical tool to explore the "decay" of $(1+1)$-waves. Indeed, the weighted fluxes provide decays such as $(1+|u_-|^2)^{-\f{1+\delta}{2}}$. We may think of $|u_-|$ as $|x_3+t|$. This factor is not integrable in $t$ but is integrable in $u_-$! This on one hand indicates that the usual quantities associated to $\Sigma_t$ may be inadequate in the proof and on the other hand shows the importance of the quantities (such as flux!) associated to $u_-$ or $C^-_{u_-}$. This will be clear in the course of the proof.

\medskip

\noindent $\bullet$ {\bf The quasi-linear approach versus linear perturbation}

 \medskip

\noindent One of the main innovations of the current work is to use the \emph{`quasi-linear'} approach to attack the problem. It consists of two main ingredients: first of all, we use the characteristic surfaces defined by the solution itself rather than the `linear' solution (or equivalently the background solution $(0,B_0)$); secondly, the multiplier vector fields and the weight functions that we use to derive energy estimates are also constructed from the solutions. Roughly speaking, each step in the course of obtaining the main estimates depends completely on the solution and we believe that a less `non-linear' approach may not work.

As we mentioned, this shares many main features with the proof of nonlinear stability of Minkowski spacetime \cite{Ch-K} in general relativity. In fact, in \cite{Ch-K}, the authors use the solution ($\approx$ spacetime) itself to construct the outgoing and incoming light cones and they are defined as the level sets of two optical functions $u$ and $\underline{u}$. In the current paper, we have constructed the functions $u_+$ and $u_-$ as analogues of optical functions and the left-traveling and right-traveling characteristic hypersurfaces $C^{\pm}_{u_\pm}$ play a similar role as light cones. In \cite{Ch-K}, the authors also use the solution to construct the multiplier vector fields, such as $\partial_t$ and Morawetz vector field $K$. We point out that in the situation of relativity the time function $t$ is not \emph{a priori} defined (since the spacetime is not defined yet and it is the solution that one is looking for) and one has to define it by knowing the solution. In our approach, the weight functions $\wpm$ or the multiplier vector fields $L_\pm$ are also defined by the solutions.

We would like to point out that, if one uses a more `linear' approach, it may not work (even for the global existence part of the main theorems). This is in contrast to the proof of stability of Minkowski spacetime. Indeed, Lindblad and Rodnianski in \cite{Lindblad-Rodnianski} proved a weaker version of the stability of Minkowski spacetime based on the multiplier vector fields and light cones of the Minkowski spacetime (near infinity, the spacetime should be more like a Schwarzschild solution rather than a flat solution). The main reason is that free waves in three dimensions decay fast (of order $\frac{1}{t}$ while $z_\pm$ behave more like a $1$-dimension waves which have no decay!) and the decoupling structure of Einstein equations in harmonic coordinates still allows one to use the null structure. In the current work, if we use the linear characteristic hypersufaces defined by $u^{(\text{linear})}_\pm=x_3\mp t$ and the corresponding $w^{(\text{linear})}_\pm$, when we derive the energy estimates, since $L_
\pm (u^{(\text{linear})}_\pm)\neq 0$, we obtain linear terms like
$$\int_0^t\int_{\Sigma_\tau} L_\pm \left(\log^2 \left< w^{(\text{linear})}_\pm\right>\right)|z_\mp|^2dxd\tau.$$
We can show that $L_\pm \left(\log^2 \left< w^{(\text{linear})}_\pm\right>\right) \geq \frac{\varepsilon}{1+|x_3\mp t|}$ and the decay is too weak to close the energy estimates. We remark that using the characteristic hypersurfaces of a real solution one can avoid this linear term.

\medskip

\noindent $\bullet$ {\bf The hybrid energy estimates}

\medskip

\noindent The most difficult part of the proof is to deal with the viscosity terms (small diffusion) since one seeks for estimates independent of the viscosity $\mu$. To make this clear, we first consider the ideal MHD system which is free of diffusion. As we mentioned before, we can use weight functions $(1+|u_\mp|^2)^\frac{1+\delta}{2}=\langle u_\mp\rangle^{1+\delta}$ for $z_\pm$ and the derivatives of $z_\pm$. The uniform choice of the weights reflects the fact that the solution $z_\pm$ behaves in all the scale like waves. When viscosity presents, we may also attempt to use the same weight. In the course of deriving energy estimates, we use integration by parts for the viscous term and the derivative will hit the weights to generate linear terms such as
\begin{equation}\label{weights example}
\mu\int_{0}^t\int_{\Sigma_\tau}\nabla^2\big(\langle u_\mp\rangle^{1+\delta}\big)|\zpm|^2dxd\tau, \ \mu\int_{0}^t\int_{\Sigma_\tau}\nabla^2\big(\langle u_\mp\rangle^{1+\delta}\big)|\nabla\zpm|^2dxd\tau,\ \cdots.
\end{equation}
Since we do not have decay estimates for terms like $\int_{\Sigma_\tau}\nabla^2\big(\langle u_\mp\rangle^{1+\delta}\big)|\zpm|^2 dx$, we can not use usual energy type estimates for wave equations to close the argument. This difficulty is indeed natural since the diffusion terms are not a wave phenomenon one does not expect to bound those terms by usual energy estimates (unless there is a new idea).

One possible approach is to lower the weight to $\langle u_\mp\rangle$ instead of $\langle u_\mp\rangle^{1+\delta}$. We can show that $|\nabla\bigl(\langle u_\mp\rangle\bigr)| \lesssim 1$ and the second term $\mu\int_{0}^t\int_{\Sigma_\tau}\nabla^2\bigl(\langle u_\mp\rangle\bigr)|\nabla\zpm|^2dxd\tau$ in \eqref{weights example} can be bounded by $\mu\int_{0}^t\int_{\Sigma_\tau} |\nabla\zpm|^2dxd\tau$. Hence, it is bounded by the basic energy estimates. It implies that $\mu\int_{0}^t\int_{\Sigma_\tau} |\nabla\zpm|^2dxd\tau$ is bounded by the initial energy. However, the first term in \eqref{weights example} cannot be bounded in this way since there is no estimates at the moment to control terms like $\mu\int_{0}^t\int_{\Sigma_\tau} |\zpm|^2dxd\tau$. We remark that, although this approach does not work, we can actually use this idea to show that the lifespan of the solution is at least $\min(\frac{1}{\mu}, e^{\frac{1}{\varepsilon}})$. Combined with the iteration method mentioned at the end of the last subsection, we can show that for $\mu \approx \varepsilon$, the solution is global. This is
in fact much better than most of the small-global-existence results in three dimensional fluids whose smallness on energy is relative to the size of $\mu$.

The new idea in our approach is to use hybrid weights to combine the hyperbolic and parabolic estimates at the same time. In fact, by lowering the weights of $z_\pm$ to $\log$-level, the first term in \eqref{weights example} will be bounded by a term that looks like
$$\mu\int_{0}^t\int_{\Sigma_\tau}\frac{\big(\log\wpm\big)^4}{\wpm^2}|z_\mp|^2dxd\tau.$$
By Hardy inequalities with respect to a right coordinates system defined by the solutions, the above terms will be bounded by
$$\mu\int_0^t\int_{\Sigma_\tau}\big(\log \wmp \big)^4 |\nabla z_\pm|^2dxd\tau.$$
This quantity will be bounded in \eqref{estimates on viscosity} and we believe that this is a new estimate to deal with small diffusion terms. This new estimate plays a central role in the proof and makes use of the full strength of the basic energy identity for the viscous MHD system.

Finally, we emphasize again that the estimate on $$\mu\int_{0}^t\int_{\Sigma_\tau}\frac{\big(\log\wpm\big)^4}{\wpm^2}|z_\mp|^2dxd\tau$$ will make an essential use of the basic energy identity. In some sense, the basic energy identity is cornerstone of the entire proof.

\medskip

\noindent $\bullet$ {\bf Three dimensional feature of the problem}

\medskip

\noindent Although the viscous Alfv\'en waves $z_\pm$ behave very similar to $(1+1)$-dimension waves on a large time scale ($\approx \frac{1}{\mu}$), the analysis indeed relies heavily on the fact that the problem is over the three dimensional space. This is another indication why the viscous case is more difficult than the ideal case (where we can only use weights function in $u_\pm$ so that it is very similar to $1$-dimension theory). A key step in the proof is to bound the weighted spacetime viscous energy  $\mu\int_0^t\int_{\Sigma_\tau}\big(\log \wmp \big)^4 |\nabla z_\pm|^2dxd\tau$ by the initial energy. We use 3-dimensional Hardy's inequality in the moving coordinate systems $(x_1^\pm,x_2^\pm,x_3^\pm)$ for $\Sigma_t$ to obtain desired estimates. It forces the weight functions involving the three dimensional radius functions $r^\pm =\sqrt{(x_1^\pm)^2 + (x_2^\pm)^2+(x_3^\pm)^2}$ (rather than $x_3^\pm = u_\pm$ as in the ideal case).

The physical picture is clear: the weight functions defined by $\wpm$ indicate that the Alfv\'en waves can be thought of as localized in all the directions in a small region of space with support moving along the characteristics.

\medskip

\noindent $\bullet$ {\bf Linear-driving decay mechanism for Alfv\'en waves with very small viscosity}

\medskip

We would like to discuss the intuition for the decay (second statement) in Theorem \ref{thm nonlinear stability of viscous waves}. We treat the MHD system as $(1+1)$-dimensional wave equations and regard the small diffusion term more or less as an error term, therefore the estimates obtained do not provide any information on the decay, just like the usual $(1+1)$-dimensional waves. In order to explore the possible decay mechanism, the new idea is now to treat the system as heat equations (with small diffusion). In a schematic manner, we can simplify the system to the following model equation:
$$\partial_t f - \mu \triangle f = \underbrace{\cdots}_{\text{error terms}}.$$
By the \emph{a priori} energy estimates, we can show that the error terms are of order $\varepsilon^2$ (say, according to $L^\infty$ norm). Therefore, by inverting the heat operator, we can think of $f$ as
$$f(t)=e^{t\mu\triangle} f(0)+\underbrace{\cdots}_{\text{error terms of order $\varepsilon^2$}}.$$
We remark that the best estimate for the error terms at the moment is a bound of order $\varepsilon^2$ and there is no decay so far for the errors.

We make the following key observation: the linear part $e^{t\mu\triangle} f(0)$ decays! Therefore, after a long time $T_1$, although initially $f(0) \sim \varepsilon$, the linear decay forces $f(T_1)$ to be of order $\varepsilon^2$. We then use the \emph{a priori} energy estimates again but set $T_1$ as the initial time for the system, this shows that after $T_1$, the solution is already of order $\varepsilon^2$ so that we have
$$f(t)=e^{(t-T_1)\mu\triangle} f(T_1)+\underbrace{\cdots}_{\text{error terms of order $\varepsilon^3$}}.$$
It is clear how to repeat the above linear-driving decay mechanism to improve the order of $\varepsilon$ by $1$ each time. This eventually pushes the solution into the $\mu$-small-data parabolic regime.

In reality, we explore the decay of $L^2$-norms of the semi-group $e^{t\mu\triangle}$. The reason is that we can only prove $L^2$-type estimates are propagated (via the hyperbolic method) and the iteration requires the estimates must be propagated in evolution. It is well-know that $\lim_{t\rightarrow \infty}\|e^{t\mu\triangle}f(0)\|_{L^2}=0$ without an explicit decay rate. Indeed, the decay behavior of $\|e^{t\mu\triangle}f(0)\|_{L^2}$ depends on the distribution of $\widehat{f}(\xi)$ around zero frequency. This is exactly the reason  why the decay behavior in Theorem \ref{thm nonlinear stability of viscous waves} depends not only on the energy norm but also on the profile of the initial data.
\bigskip

The rest of the paper consists of two sections. The next section is the technical heart of the paper and it proves the main \emph{a priori} energy estimates (and Theorem \ref{global existence for MHDmu}). The last section proves Theorem \ref{global existence for ideal MHD with small amplitude}, Theorem \ref{theorem scattering theory for ideal Alfven waves} and Theorem \ref{thm nonlinear stability of viscous waves}.

\subsection{Further remarks}

From both of the mathematical and physical perspectives, it is certainly of great interests to study the global dynamics of the incompressible MHD systems \eqref{MHD general} in domains other than the three dimensional Euclidean space. We would like propose three typical open problems in this direction:

\medskip

\textbf{(Problem 1)} Study Alfv\'en waves on a slab bounded by two hyperplane.  More precisely, let $\Omega = [-1,1] \times \mathbb{R}^2 \subset \mathbb{R}^3$ be slab in $\mathbb{R}^3$ defined as follows:
 $$\Omega = \{(x_1,x_2,x_3) \in \mathbb{R}^3 | -1 \leq  x_1 \leq 1 \}.$$
 In addition to the incompressible MHD system \eqref{MHD general}, on $\partial \Omega$ which consists of two parallel hyperplanes, we pose the following boundary condition:
 \begin{equation*}
 v\cdot \mathbf{n} =0, \ \ b\cdot \mathbf{n}=0,
 \end{equation*}
 where $\mathbf{n}$ is the unit normal of $\partial \Omega$. The background magnetic field $B_0$ is still taken to be $(0,0,1)$. One would like to prove a theorem similar to Theorem \ref{global existence for ideal MHD with small amplitude} and Theorem \ref{thm nonlinear stability of viscous waves}.

\smallskip

\textbf{(Problem 2)} Study the nonlinear stability of Alfv\'en waves on $\mathbb{R}^2$.

\smallskip

\textbf{(Problem 3)} Study the nonlinear stability of Alfv\'en waves on periodic domains, i.e., two or three dimensional tori, $\mathbb{R}^2 \slash \mathbb{Z}^2$ or  $\mathbb{R}^3\slash \mathbb{Z}^3$.

\medskip

 From the technical point of view, the second and third problems are indeed relatively easy to attack since the stabilizing mechanism is still in the scope of this work. One of the key ingredients in the current work is the separation of two families of Alfv\'en waves. Indeed, as we have explained, we may heuristically regard the Alfv\'en waves travel along the $B_0$ direction. In particular, in the above slab $\Omega$ in \textbf{(Problem 1)}, the $B_0$ direction, i.e., $x_3$-axis direction, is unbounded. Therefore, the Alf\'evn waves can escape to infinity in a similar way as for the whole space case. Therefore, the energy will still be transmitted to infinity and this is stabilizing mechanism underlying the proof. Similarly, the idea may also be adapted to the two dimensional situation. In contrast to this discussion, the space in the periodic case is compact so that the Alf\'en waves cannot go too far away. So \textbf{(Problem 3)} may require a completely new approach.

In a forthcoming paper, we will provide an affirmative answer to \textbf{(Problem 1)} and  \textbf{(Problem 2)}.

\section{Main a priori estimates}

\subsection{Ansatz for the method of continuity}\label{continuity subsection}
To use the method of continuity, we have three sets of assumptions concerning the underlying geometry and the energy of the waves.

The first set describes the geometry defined by the solution. Recall that, $(x_1^+, x_2^+, x_3^+(=u_+))$ are the $L_+$-transported functions which coincide with the Cartesian coordinates $(x_1,x_2,x_3)$ on $\Sigma_0$. For a given time $t\in [0,t^*]$, the restrictions of $(x_1^+,x_2^+,x_3^+)$ on $\Sigma_t$ yield a new coordinate system. We consider the change of coordinates $(x_1,x_2,x_3) \rightarrow (x_1^+,x_2^+,x_3^+)$ on $\Sigma_t$ and we use $\big({\partial x_i^+} / {\partial x_j}\big)_{1\leq i,j \leq 3}$ to denote the corresponding Jacobian matrix. Similarly, we have another change of coordinates  $(x_1,x_2,x_3) \rightarrow (x_1^-,x_2^-,x_3^-)$ on $\Sigma_t$ and the corresponding Jacobian matrix $\big({\partial x_i^-} / {\partial x_j}\big)_{1\leq i,j \leq 3}$.

We make the following ansatz on the underlying geometry:
\begin{equation}\label{Bootstrap on geometry}
\boxed{ \big|\big(\frac{\partial x_i^\pm}{\partial x_j}\big)-\I\big|\leq 2 C_0 \varepsilon\leq \frac{1}{10},\ \big|\na\big(\frac{\partial x_i^\pm}{\partial x_j}\big)\big|\leq 2 C_0 \varepsilon\leq \frac{1}{10}, \ \ \text{for all $(t,x)\in [0,t^*]\times \mathbb{R}^3$,}}
\end{equation}
where $\I$ is the $3\times 3$ identity matrix and $C_0$ is a universal constant which will be  determined towards the end of the proof.

The second ansatz is about the amplitude of $\zpm$. We assume that
\begin{equation}\label{Bootstrap on amplitude}
\boxed{\|\zpm\|_{L^\infty}\leq\f12.}
\end{equation}

The third set of ansatz is designed for the energy and flux. We fix a positive integer $N_* \geq 5$. For all $k\leq N_*$, we assume that
\begin{equation}\label{Bootstrap on energy}
\boxed{E_{\pm} \leq 2C_1\varepsilon^2,\ \  F_{\pm} \leq  2C_1 \varepsilon^2,\ \  \mu E^{N_*+1}_{\pm}+E^k_{\pm} \leq  2C_1\varepsilon^2,\ \  F^k_{\pm} \leq  2C_1 \varepsilon^2, \  \ D^k_{\pm} \leq  2C_1 \varepsilon^2, \ \ k \leq N_*.}
\end{equation}
Here $C_1$ will be determined by the energy estimate.

 We will use the standard continuity argument: since \eqref{Bootstrap on geometry} and \eqref{Bootstrap on energy} hold for the initial data, they remain correct for a short time, say $[0,t_{\max}]$ where $t_{\max}$ is the maximal possible time so that the three sets of ansatz remain valid. Without loss of generality, we can assume $t_{\max}=t^*$. We need two steps to close the continuity argument:
 \begin{itemize}
 \item[{\bf{Step 1}}] There exists a $\varepsilon_0$, for all $\varepsilon<\varepsilon_0$, we can improve the constant $2$ in \eqref{Bootstrap on energy} to $1$, i.e.,
 \begin{equation*}
E_{\pm} \leq C_1 \varepsilon^2,\ \  F_{\pm} \leq  C_1 \varepsilon^2,\ \  \mu E^{N^*+1}_{\pm}+E^k_{\pm} \leq  C_1\varepsilon^2,\ \  F^k_{\pm} \leq  C_1 \varepsilon^2, \ \ k \leq N_*.
\end{equation*}
 \item[{\bf{Step 2}}]There exists a $\varepsilon_0$, for all $\varepsilon<\varepsilon_0$, we can improve the constant $2C_0$ to $C_0$ in \eqref{Bootstrap on geometry}, i.e., we have
 \begin{equation*}
 \big|\big(\frac{\partial x_i^\pm}{\partial x_j}\big)-\I\big|\leq  C_0 \varepsilon,\ \big|\na\big(\frac{\partial x_i^\pm}{\partial x_j}\big)\big|\leq  C_0 \varepsilon, \ \ \text{for all $(t,x)\in [0,t^*]\times \mathbb{R}^3$,}
\end{equation*}
\end{itemize}
Once we complete the above two steps, the method of continuity implies global solutions for the MHD system. We emphasize that the smallness of $\varepsilon_0$ in the above two steps does not depend on the size of viscosity $\mu$ and does not depend on the lifespan $[0,t^*]$. It indeed depends only on the background stationary magnetic field $B_0$.

\subsection{Preliminary estimates}\label{subsection Preliminary estimates}

In this subsection, we assume that the geometric ansatz \eqref{Bootstrap on geometry} and the amplitude ansatz \eqref{Bootstrap on amplitude} hold.

Let $\psi_{\pm}(t,y)=(\psi^1_{\pm}(t,y), \psi^2_{\pm}(t,y), \psi^3_{\pm}(t,y))$ (the mapping from $\Sigma_0$ to $\Sigma_t$) be the flow generated by $Z_{\pm}$, i.e.,
\begin{equation}
\frac{d}{dt}\psi_{\pm}(t,y)=Z_{\pm}(t,\psi_{\pm}(t,y)), \ \ \psi_{\pm}(0,y)=y,
\end{equation}
where $y\in \mathbb{R}^3$. Here and in what follows, if we use the flow map, we  use $y$ as the initial label(or the Lagrangian coordinates), and $x$ as the present label (or the Eulerian coordinates). Since $\zpm = \Zpm \mp B_0$ (recall that $B_0=(0,0,1)$), after integration, we obtain
\begin{equation}\label{flow in integration form}
\psi_{\pm}(t,y)=y+\int_0^tZ_{\pm}(\tau,\psi_{\pm}(\tau,y))d\tau =y \pm t B_0+\int_0^tz_{\pm}(\tau,\psi_{\pm}(\tau,y))d\tau.
\end{equation}
We remark that the flows $\psi_{\pm}$ are the analogues of the Lagrangian coordinates in the ordinary fluid theory.

Let $\frac{\partial\psi_{\pm}(t,y)}{\partial y}$ be the differential of $\psi(t,y)$ at $y$. Thanks to the privileged Cartesian coordinates on $\mathbb{R}^3$, we regard $\frac{\partial\psi_{\pm}(t,y)}{\partial y}$ as a $3\times 3$ matrix. By definition,
we know that $\psi_{\pm}(t,\cdot)^* x^\pm_i = x_i$, i.e., $x^\pm(t,\psi_\pm(t,y))=y$. Therefore, we indeed have
\begin{equation*}
\frac{\partial x^\pm}{\partial x}|_{x=\psi_\pm(t,y)}=\Bigl(\frac{\partial\psi_{\pm}(t,y)}{\partial y}\Bigr)^{-1}, \ \ \na_x\big(\frac{\partial x^\pm}{\partial x}\big)|_{x=\psi_\pm(t,y)}=\na_y\bigr\{\Bigl(\frac{\partial\psi_{\pm}(t,y)}{\partial y}\Bigr)^{-1}\bigr\}\Bigl(\frac{\partial\psi_{\pm}(t,y)}{\partial y}\Bigr)^{-1}.
\end{equation*}
Therefore, we can rephrase the geometric ansatz \eqref{Bootstrap on geometry} as
\begin{equation}\label{Bootstrap on the flow}
\boxed{ \big|\frac{\partial\psi_{\pm}(t,y)}{\partial y}-\I\big|\leq C_0' \varepsilon\leq \frac{1}{2},\ \big|\na_y\Bigl(\frac{\partial\psi_{\pm}(t,y)}{\partial y}\Bigr)\big|\leq C_0' \varepsilon\leq \frac{1}{2}, \ \ \text{for all $t\in [0,t^*],\ \ y\in\mathbb{R}^3$,}}
\end{equation}

This ansatz gives the following bounds on the weight functions:
\begin{lemma}[\bf{Differentiate Weights}]\label{lemma differentiate weights}
We have
\begin{equation}\label{differentiate weights}
 |\nabla^i \wpm| \leq 2, \quad\text{for}\quad i=1,2.
\end{equation}
In particular, for all $\omega_1,\omega_2 \in \R$, we have for $i=1,2$
\begin{equation*}
\begin{split}
&\big|\na^i \wp^{\omega_1} \big| \le C_{\omega_1}\wp^{\omega_1-1}, \ \  \big|\na^i \wm^{\omega_2} \big|\le C_{\omega_2}\wm^{\omega_2-1},\ \ \big|\na^i\big(\wp^{\omega_1}\wm^{\omega_2}\big)\big|\le C_{\omega_{1},\omega_2}\f{\wp^{\omega_1}\wm^{\omega_2}}{R},\\
&\big|\na^i \big(\log\wpm\big)^{\omega_1} \big| \le C_{\omega_1} \frac{\big(\log\wpm\big)^{\omega_1-1}}{\wpm}, \ \ \big|\na^i \Big(\wpm^{\omega_1}\big(\log\wpm\big)^{\omega_2}\Big) \big| \le C_{\omega_{1},\omega_2}\wpm^{\omega_1-1}\big(\log\wpm\big)^{\omega_2}.
\end{split}
\end{equation*}
\end{lemma}
\begin{proof}It suffices to show \eqref{differentiate weights} and the rest inequalities are immediate consequences of this inequality. It suffices to bound $\na^i\wp$ for $i=1,2$. The inequalities for $\wm$ will be similar to derive.

In view of the the definition of $\wp$  and the chain rule for differentiation, letting $k \in \{1,2,3\}$, we have
\begin{equation*}
 \p_k \wp = \frac{x_1^+\frac{\partial x_1^+}{\partial x_k} + x_2^+\frac{\partial x_2^+}{\partial x_k}+\up\frac{\partial \up}{\partial x_k} }{\big(R^2 + |x_1^+|^2+|x_2^+|^2+|\up|^2\big)^\frac{1}{2}}.
\end{equation*}
By the geometric ansatz \eqref{Bootstrap on geometry}, we have $\big|\frac{\partial x_l^+}{\partial x_k}\big| \leq 2$ for all $l$ (recall that $x_3^+ = \up$). Then we obtain $|\na\wp|\leq 2$. Similarly, by the chain rule and the ansatz \eqref{Bootstrap on geometry}, we could obtain that $|\na^2\wp|\leq2$. Therefore, \eqref{differentiate weights} is proved. This completes the proof of lemma.
\end{proof}

As an application of this lemma, we claim the following weighted Sobolev inequalities hold:
\begin{lemma}[{\bf{Sobolev inequalities}}]For all $k \leq N_*-2$ and multi-indices $\alpha$ with $|\alpha|=k$, we have
\begin{equation}\label{Sobolev}
\begin{split}
|z_{\mp}|&\lesssim \frac{1}{\big(\log \wpm\big)^2}\big(E_{\mp} + E^0_{\mp}+E^1_{\mp}\big)^\frac{1}{2},\\
|\na z_{\mp}^{(\alpha)}|&\lesssim   \frac{1}{\wpm \big(\log \wpm\big)^2} \big(E_{\mp}^{k}+E_{\mp}^{k+1}+E_{\mp}^{k+2}\big)^\frac{1}{2}.
\end{split}
\end{equation}
\end{lemma}
\begin{proof} We only give the proof concerning the right-traveling Alfv\'en wave $z_-$. The estimates for $z_+$ can be derived in the same manner.

By the standard Sobolev inequality, we have
\begin{align*}
\big| \big(\log \wp\big)^2 z_- \big|^2 &\lesssim  \|\big(\log \wp\big)^2 z_- \|_{H^2(\mathbb{R}^3)}^2= \sum_{|\beta|\leq 2}\|\p^\beta \Big(\big(\log \wp\big)^2 z_-\Big)\|^2_{L^2}.
\end{align*}
According to Lemma \ref{lemma differentiate weights}, we have
\begin{align*}
\Big|\p^\beta \Big(\big(\log \wp\big)^2 z_-\Big)\Big|&\leq \sum_{\gamma \leq \beta}\Big|\nabla^{\gamma} \big(\log \wp\big)^2 z_-^{(\beta-\gamma)}\Big|\\
&\lesssim \sum_{\gamma \leq \beta}\Big| \big(\log \wp\big)^2 z_-^{(\beta-\gamma)}\Big|.
\end{align*}
Hence,
\begin{align*}
\big| \big(\log \wp\big)^2 z_- \big|^2 &\lesssim  \sum_{|\beta|\leq 2}\|\big(\log \wp\big)^2 z^{(\beta)}_-\|^2_{L^2}\\
&\lesssim E_{-} + E^0_{-}+E^1_{-}.
\end{align*}
This gives the $L^\infty$ bound on $z_-$.

For higher order derivatives, we have
\begin{align*}
\big| \wp \big(\log \wp\big)^2 \nabla z_-^{(\alpha)}\big|^2 &\lesssim  \sum_{|\beta|\leq 2}\|\p^\beta \Big(\wp\big(\log \wp\big)^2 \na z_-^{(\al)}\Big)\|^2_{L^2}\\
&\stackrel{Lemma\, \ref{lemma differentiate weights}}{\lesssim}\sum_{k\leq |\beta|\leq k+2}\|\wp\big(\log \wp\big)^2 \na z_-^{(\beta)}\|^2_{L^2}.
\end{align*}
The last line is obviously bounded by $E_{-}^{k}+E_{-}^{k+1}+E_{-}^{k+2}$. This completes the proof of the lemma.
\end{proof}

We present the lemma about the separation property of the left- and right-traveling waves.
\begin{lemma}
Assume that $\|\zpm\|_{L^\infty}\leq\f12$, $R>10$, we have
\beq\label{separate property}
t\leq|u_+-u_-|\leq3t.
\eeq
Moreover, there hold
\beq\label{separate weight}\begin{aligned}
&\wp\wm\geq(R^2+|u_+|^2)^{\f12}(R^2+|u_-|^2)^{\f12}\geq\f{R}{2}(R^2+t^2)^{\f12},\\
&\log\wp\log\wm\geq\log(R^2+|u_+|^2)^{\f12}\log(R^2+|u_-|^2)^{\f12}\geq\f{\log R}{2}\log(R^2+t^2)^{\f12}.
\end{aligned}\eeq
\end{lemma}
\begin{proof}By  virtue of $\psi_\pm(t,y)$, we solve $u_\pm$ from $L_\pm u_\pm=0$ as follows
\beno
u_\pm(t,\psi_\pm(t,y))=y_3.
\eeno
Thanks to \eqref{flow in integration form}, we have
\beno
u_\pm(t,\psi_\pm(t,y))=\psi_{\pm}^3(t,y)\mp t-\int_0^tz_{\pm}^3(\tau,\psi_{\pm}(\tau,y))d\tau.
\eeno
Then
\beno
u_\pm(t,x)=x_3\mp t-\int_0^tz_{\pm}^3(\tau,\psi_{\pm}(\tau,\psi^{-1}_{\pm}(t,x)))d\tau,
\eeno
which gives rise to
\beno
|(u_--u_+)-2t|\leq\int_0^t(\|\zp^3\|_{L^\infty}+\|\zm^3\|_{L^\infty})dt\leq t,
\eeno
where we used the assumption $\|\zpm^3\|_{L^\infty}\leq\f12$.
This yields the estimate \eqref{separate property}. And \eqref{separate property} gives rise to $|u_+|+|u_-|\geq t$ which shows that either $|u_+|\geq\f{t}{2}$ or $|u_-|\geq\f{t}{2}$. Then there holds \eqref{separate weight}. The lemma is proved.
\end{proof}

\begin{remark}
The estimate \eqref{separate property} shows that if $\|\zpm^3\|_{L^\infty}$ is small than the background magnetic field, the left-traveling hypersurface $C_{u_+}^+$ and the right-traveling hypersurface $C_{u_-}^-$ will separate from each other after the initial time. And at time $t$, the distance between them is of order $O(t)$.
\end{remark}

We now state a lemma to control the normal derivatives of the characteristic hypersurfaces in $[0,t^*]\times \R^3$:
\begin{lemma}\label{lemma normal derivative control} Assume that $\|\zpm\|_{L^\infty}\leq\f12$. Then for all $u_+$ and $u_-$, we have
\begin{equation}
\f{7}{16}\leq\langle L_-,\nu_+\rangle|_{C_{u_+}^+}\leq 4, \ \ \f{7}{16}\leq\langle L_+,\nu_-\rangle|_{C_{u_-}^-}\leq 4,
\end{equation}
where $\nu_{\pm}$ is the normal vector field of $C_{u_\pm}^\pm$.
\end{lemma}
\begin{proof} We prove the first inequality and the second can be derived exactly in the same manner.

Since $L_-=(1,Z_-^1,Z_-^2,Z_-^3)$ and $\nu_+=-\f{\wt\na_{t,x}u_+}{|\wt\na_{t,x}u_+|}=-\f{(\p_tu_+,\na u_+)}{\sqrt{|\p_tu_+|^2+|\na u_+|^2}}$, we have
\begin{equation}\label{inner product}
\langle L_-,\nu_+\rangle=\f{1}{|\wt\na_{t,x}u_+|}\bigl(-\p_tu_+-Z_-\cdot\na u_+\bigr).
\end{equation}
Let $\vv e_3 =(0,0,1)$. Since $\p_tu_++Z_+\cdot\na u_+=0$, we have
\begin{align*}
\langle L_-,\nu_+\rangle&=\f{1}{|\wt\na_{t,x}u_+|}(Z_+-Z_-)\cdot\na u_+\\
&=\f{1}{|\wt\na_{t,x}u_+|}\bigl(2\vv e_3\cdot\na u_++(z_+-z_-)\cdot\na u_+\bigr)
\end{align*}
and
\begin{align*}
|\wt\na_{t,x}u_+|&=\sqrt{|Z_+\cdot\na u_+|^2+|\na u_+|^2}\\
&=\sqrt{|\partial_3 u_+|^2+|z_+\cdot\na u_+|^2+|\na u_+|^2+2\partial_3 u_+(z_+\cdot\na u_+)}.
\end{align*}
In view of \eqref{Bootstrap on geometry}, we obtain
$$|\na u_+-\vv e_3|\leq \sqrt{2C_0}\varepsilon.$$
By virtue of \eqref{Bootstrap on amplitude} i.e.,  $\|\zpm\|_{L^\infty}\leq\f12$, we have
$$|z_\pm\cdot\na u_+|\leq \f12+ \sqrt{2C_0}\varepsilon.$$
It is straightforward to see that the numerator in \eqref{inner product} is in $[\frac{7}{8},\frac{25}{8}]$; the denominator in \eqref{inner product} is
in $[\frac{7}{8},2]$, provided $\varepsilon$ is sufficiently small. This completes the proof.
\end{proof}

We will also need a weighted version of div-curl lemma:
\begin{lemma}[{\bf{div-curl lemma}}]\label{div-curl lemma} Let $\lam(x)$ be a smooth positive function on $\R^3$. For all smooth vector field $\vv v(x)\in H^1(\R^3)$ with the following properties
$$\dv\,\vv v=0,\ \ \  \sqrt\lam\na\vv v\in L^2(\R^3), \ \ \  \f{|\na\lam|}{\sqrt\lam}\vv v\in L^2(\R^3),$$
we have \beq\label{mu21}
\|\sqrt\lam\na\vv v\|_{L^2}^2\lesssim \|\sqrt\lam\curl\vv v\|_{L^2}^2+\|\f{|\na\lam|}{\sqrt\lam}\vv v\|_{L^2}^2.
\eeq
\end{lemma}
\begin{proof} Since $\dv\,\vv v=0$, we have $-\D\vv v=\curl\curl\,\vv v$. We now multiply this identity by $\lam\vv v$ and then integrate over $\R^3$. We obtain
\begin{align*}
\int_{\R^3}\lam|\na\vv v|^2dx&=-\int_{\R^3}\sum_{i=1}^3\p_i\lam\p_i\vv v\cdot\vv vdx+\int_{\R^3}\curl\,\vv v\cdot\curl(\lam\vv v)dx\\
&\leq\int_{\R^3}\lam|\curl\,\vv v|^2dx+2\int_{\R^3}|\na\lam||\vv v||\na\vv v|dx\\
&\leq\int_{\R^3}\lam|\curl\,\vv v|^2dx+2\int_{\R^3}\f{|\na\lam|^2}{\lam}|\vv v|^2dx+\frac{1}{2}\int_{\R^3}\lam|\na\vv v|^2dx.
\end{align*}
To complete the proof, it suffices to move the last term to the left hand side.
\end{proof}
\begin{remark}
Because of $\div z_\pm =0$, this lemma allows us to switch the term $\nabla z_{\pm}^{(\gamma)}$ in energy to the vorticity term $j_\pm^{(\gamma)}$.  This enables us to use the vorticity formulation \eqref{main equations for j} of the MHD system. And we will show that it is difficult for us to avoid investigating the vorticity formulation \eqref{main equations for j}, especially for the highest order energy estimates.
\end{remark}

\begin{remark}
In applications, we will take weight function $\lambda$ satisfying the following property:
\begin{equation}\label{a property of lambda}
|\na\lam|\lesssim \lam.
\end{equation}
Therefore, \eqref{mu21} becomes
\begin{equation*}
\|\sqrt\lam\na\vv v\|_{L^2(\mathbb{R}^3)}^2 \lesssim \|\sqrt\lam\curl\vv v\|_{L^2(\mathbb{R}^3)}^2+ \|\sqrt\lam\vv v\|_{L^2(\mathbb{R}^3)}^2.
\end{equation*}
In particular, for $v=\na z_+^{(\gamma)}$ which is divergence free, we have
\begin{equation}\label{mu18}
\|\sqrt\lam\na z_+^{(\gamma)}\|_{L^2(\Sigma_\tau)}^2\lesssim \|\sqrt\lam j_+^{(\gamma|)}\|_{L^2(\Sigma_\tau)}^2+\|\sqrt\lam
\nabla z_+^{(|\gamma|-1)}\|_{L^2(\Sigma_\tau)}^2.
\end{equation}
For $1\leq|\gamma| \leq N_*$, we can iterate \eqref{mu18} to derive
\begin{equation}\label{z to j inequality}
\|\sqrt\lam\na z_+^{(\gamma)}\|_{L^2(\Sigma_\tau)}^2 \lesssim \|\sqrt\lam
\na z_+\|_{L^2(\Sigma_\tau)}^2+\sum_{k=1}^{|\gamma|} \|\sqrt\lam
j_+^{( k)}\|_{L^2(\Sigma_\tau)}^2.
\end{equation}

We remark that in \eqref{z to j inequality}, we do not iterate $\|\sqrt\lam
\na z_+\|_{L^2(\Sigma_\tau)}^2$  by $\|\sqrt\lam
j_+\|_{L^2(\Sigma_\tau)}^2+\|\f{|\na\lam|}{\sqrt\lam}
z_+\|_{L^2(\Sigma_\tau)}^2$. We will see that it is difficult to control $\|\f{|\na\lam|}{\sqrt\lam}
z_+\|_{L^2(\Sigma_\tau)}^2$ by taking $\lam=\lam(u_+,u_-)$.
\end{remark}

The geometric ansatz \eqref{Bootstrap on geometry} also provides a trace theorem for restrictions of functions to the characteristic hypersurfaces $C_{u_\pm}^\pm$:
\begin{lemma}[{\bf{Trace}}]\label{Trace Estimates}
For all $f(t,x)\in L^2([0,t^*];H^1(\R^3))$,  the restriction of $f$ to $C_{u_\pm}^\pm$ belongs to $L^2(C_{u_\pm}^\pm)$. In fact, we have
\begin{equation*}
\|f\|_{L^2(C_{u_\pm}^\pm)}\lesssim \|f\|_{L^2([0,t^*];H^1(\R^3))}.
\end{equation*}
\end{lemma}
\begin{proof}
Let $a_+$ be a fixed real number and we will prove the trace estimates for $C^+_{a_+}$. By definition, we have $S_{t,u_+}^+ =\p\,\Sigma_t^{[u_+,+\infty)}$ and $C_{u_+}^+=\bigcup_{0\leq\tau\leq t^*}S_{\tau,u_+}^+$. On each $\Sigma_t$, we will write $S_{t,a_+}^+$ as a graph over $(x_1,x_2)$ plane. We emphasize that $(x_1,x_2,x_3)$ is the standard Cartesian coordinates system on $\Sigma_t$.

We claim that $S_{t,a_+}^+ \subset \Sigma_t$ is the following graph
\begin{equation}
S_{t,a_+}^+=\{(x_1,x_2,x_3)\,|\,x_3=\eta_+(t,x_h),\ x_h=(x_1,x_2)\}
\end{equation}
where $\eta_+$ is defined by $\p_t\eta_++z_+^h\cdot\na_{x_h}\eta_+=1+z_+^3$ with $\eta_+|_{t=0}=a_+$ and $z_+^h=(z_+^1,z_+^2)$. In fact, the equation for $\eta_+$ is equivalent to $\p_t(x_3-\eta_+)+Z_+\cdot\na(x_3-\eta_+)=0$. Therefore, it is easy to see that
$u_+(t,x)=x_3-\eta_+(t,x_h)+a_+$ and
$$C_{a_+}^+=\{(t,x)\,|\,x_3=\eta_+(t,x_h),\,\,\text{with}\,\,\eta_+(0,x_h)=a_+\}.$$

To prove the lemma, we will first of all control the hypersurface measure on  $C_{a_+}^+$:
\begin{align*}
d\sigma_+&=\sqrt{1+|\na_{t,x_h}\eta_+|^2}dx_1dx_2dt=\sqrt{1+|\na_{t,x_h}u_+|^2}dx_1dx_2dt.
\end{align*}
Since $\p_tu_+ + Z_+\cdot\na u_+=0$, we have
\beq\label{surface measure}
d\sigma_+=\sqrt{1+|Z_+\cdot\na u_+|^2+|\na_{x_h}u_+|^2}dx_1dx_2dt.
\eeq
By \eqref{Bootstrap on amplitude}, for sufficiently small $\varepsilon$, we have $|z_+| \leq C\varepsilon$. By \eqref{Bootstrap on geometry}, we have $|\nabla u_+-\vv e_3| \leq C\varepsilon$. Therefore, we obtain that
$$\sqrt{2}-C\varepsilon\leq\sqrt{1+|Z_+\cdot\na u_+|^2+|\na_{x_h}u_+|^2}\leq\sqrt{2}+C\varepsilon.$$
As a consequence, we have
\begin{equation}\label{mu11}
\int_{C_{u_+}^+}|f|^2 d\sigma_+ \leq 4 \int_0^t\int_{\R^2}\Big(f(\tau,x)\big|_{x_3=\eta_+(\tau,x_h)}\Big)^2dx_1dx_2d\tau.
\end{equation}

We consider a change of coordinates on $\Sigma_t$:
$$(x_1,x_2,x_3)\rightarrow(\tilde{x}_1,\tilde{x}_2,\tilde{x}_3)$$
where the new coordinate $\tilde{x}_1 =x_1$, $\tilde{x}_2 = x_2$ and $\tilde{x}_3=x_3-\eta_+(t,x_h)$. We define
$$\tilde{f}(t,\tilde{x})= f(t,\tilde{x}_1,\tilde{x}_2,\tilde{x}_3+\eta_+(t,\tilde{x}_1,\tilde{x}_2)).$$
Hence,
$$f(t,x)|_{x_3=\eta_+(t,x_h)}=\tilde{f}(t,\tilde{x}_h,\tilde{x}_3)|_{\tilde{x}_3=0}.$$
By the standard trace theorem, we have
$$\|\tilde{f}(t,\tilde{x}_h,0)\|_{L^2(\R^2)} \lesssim \|\tilde{f}(t,\tilde{x})\|_{H^1(\R^3)}.$$
In view of \eqref{mu11}, we have
\begin{equation}\label{mu12}
\int_{C_{u_+}^+}|f|^2 d\sigma_+\lesssim \int_0^t\big(\|\tilde{f}(\tau,\tilde{x})\|_{L^2(\R^3)}^2+\|\na_{\tilde{x}}\tilde{f}(\tau,\tilde{x})\|_{L^2(\R^3)}^2\big)d\tau.
\end{equation}
We now change the $(\tilde{x}_1,\tilde{x}_2,\tilde{x}_3)$ coordinates back to $(x_1,x_2,x_3)$. Since $
\f{\p\tilde{x}(x)}{\p x}=\begin{pmatrix}1&0&0\\
0&1&0\\
-\p_1\eta_+&-\p_2\eta_+&1\end{pmatrix}$, the inverse Jacobian matrix reads as $
\bigl(\f{\p\tilde{x}(x)}{\p x}\bigr)^{-1}=\begin{pmatrix}1&0&0\\
0&1&0\\
\p_1\eta_+&\p_2\eta_+&1\end{pmatrix}.
$
As a result, we have
$
\det\bigl(\f{\p\tilde{x}(x)}{\p x}\bigr)=1$. We also remark that $\na_{\tilde{x}}=\bigl(\f{\p\tilde{x}(x)}{\p x}\bigr)^{-T}\na_x
$.

Because $\det\bigl(\f{\p\tilde{x}(x)}{\p x}\bigr)=1$, we have
$$
\|\tilde{f}(\tau,\tilde{x})\|_{L^2(\R^3)}^2=\|f(\tau,x)\|_{L^2(\R^3)}^2.
$$
Furthermore, we have
\begin{align*}
\|\na_{\tilde{x}}\tilde{f}(\tau,\tilde{x})\|_{L^2(\R^3)}^2&=\|\bigl(\f{\p\tilde{x}(x)}{\p x}\bigr)^{-T}\na_xf(\tau,x)\|_{L^2(\R^3)}^2 \\
&\leq\|\na_xf(\tau,x)\|_{L^2(\R^3)}^2+\|\bigl(\bigl(\f{\p\tilde{x}(x)}{\p x}\bigr)^{-T}-\I\bigr)\na_xf(\tau,x)\|_{L^2(\R^3)}^2\\
&\leq(1+\|\na_h\eta_+\|_{L^\infty(\R^2)}^2)\|\na_xf(\tau,x)\|_{L^2(\R^3)}^2\\
&=(1+\|\na_hu_+\|_{L^\infty(\R^2)}^2)\|\na_xf(\tau,x)\|_{L^2(\R^3)}^2.
\end{align*}
By \eqref{Bootstrap on geometry}, we have
$$
\|\na_{\tilde{x}}\tilde{f}(\tau,\tilde{x})\|_{L^2(\R^3)}^2\lesssim \|\na_xf(\tau,x)\|_{L^2(\R^3)}^2.
$$
Combining all the estimates with \eqref{mu12}, this completes the proof of lemma.
\end{proof}

\subsection{Energy estimates for linear equations}
We start by deriving energy identities for the following linear system of equations:
\begin{equation}\label{linear equation}
\begin{split}
\partial_t  \fp +\Zm \cdot \nabla \fp &= \rhop, \\
\partial_t  \fm +\Zp \cdot \nabla \fm &= \rhom.
\end{split}
\end{equation}
We emphasize that $\Zp$ and $\Zm$ are divergence-free vector fields.

We consider two weight functions $\lambdap$ and $\lambdam$ defined on $[0,t^*]\times \mathbb{R}^3$. They will be determined later on in the paper. We require that
\begin{equation*}
 \Lp \lambdam =0, \ \ \Lm \lambdap =0.
\end{equation*}

We start with the estimates on $\fm$ which corresponds to the right-traveling Alfv\'{e}n waves. By multiplying (or taking inner product with) $\lambdam\fm$ to the second equation in \eqref{linear equation}, we have
\begin{equation}\label{e1}
\frac{1}{2}\lambdam \partial_t \big( |\fm|^2\big) +\frac{1}{2}\lambdam  (\Zp \cdot \nabla ) \big(|\fm|^2\big) = \lambdam \fm\cdot\rhom.
\end{equation}
By the definition of $\Lp$, the left hand side can be rewritten as $\frac{1}{2}\lambdam \Lp   \big(|\fm|^2\big)$. In view of the fact that $\Lp \lambdam=0$, it again can be reformulated as
$\frac{1}{2}\Lp  \big( \lambdam|\fm|^2\big)$.

We use $\widetilde{\div}$ to denote the divergence of $\mathbb{R}^4$ with respect to the standard Euclidean metric. Since $\div \Zp =0$, therefore, $\widetilde{\div}\Lp =0$. We integrate equation \eqref{e1} on $W_{t}^{[\up^1,\up^2]}$.  According to the Stokes formula, the left hand side of the resulting equation yields

\includegraphics[width = 6 in]{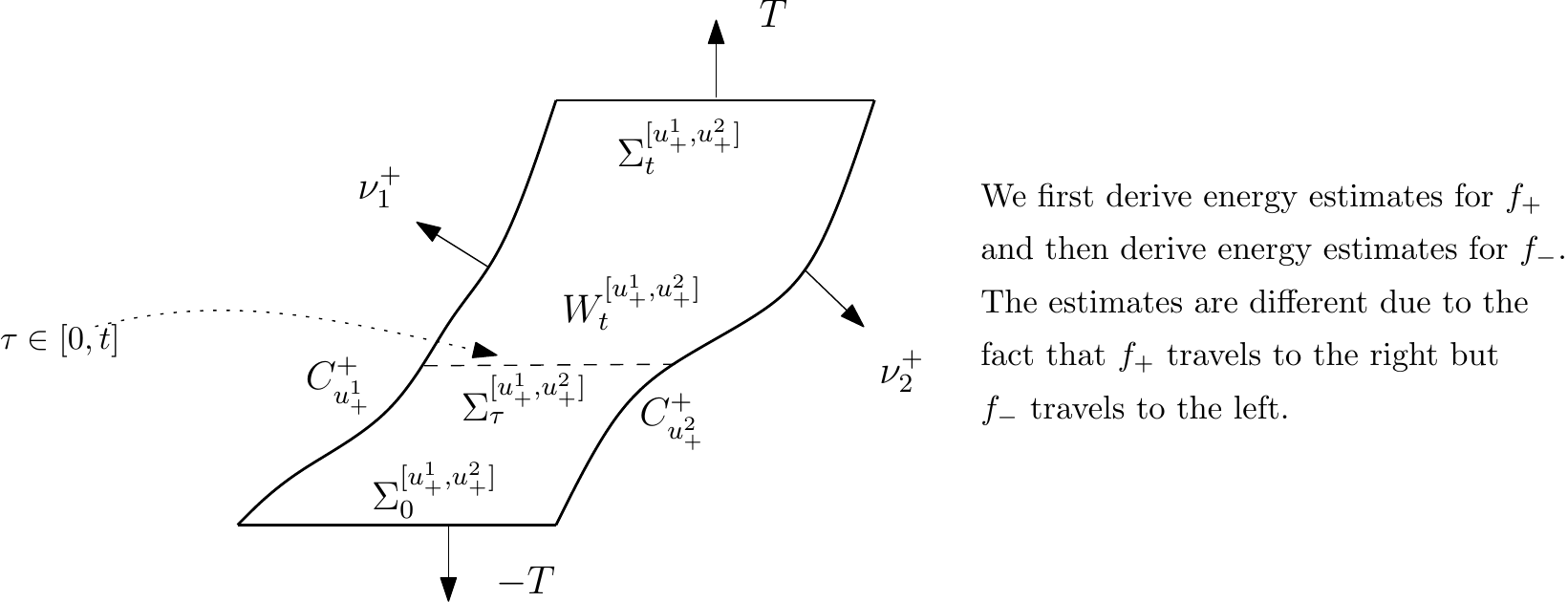}

\begin{equation}\label{mu0}
\begin{split}
& \ \ \ \frac{1}{2}\doubleint_{W_{t}^{[\up^1,\up^2]}} \Lp  \big( \lambdam|\fm|^2\big) dxd\tau \\
&=\frac{1}{2}\doubleint_{W_{t}^{[\up^1,\up^2]}}  \widetilde{\div} \big(\lambdam|\fm|^2 \Lp\big) dxd\tau-\underbrace{\frac{1}{2}\doubleint_{W_{t}^{[\up^1,\up^2]}}  \lambdam|\fm|^2 \widetilde{\div}  \Lp dxd\tau}_{\widetilde{\div}\Lp=0 \ \Rightarrow \  \text{This term is $0$.}}\\
&\overset{\text{Stokes}}{=} \frac{1}{2}\int_{\Sigma_{t}^{[\up^1,\up^2]}}\lambdam|\fm|^2 \langle \Lp, T \rangle dx -\frac{1}{2}\int_{\Sigma_{0}^{[\up^1,\up^2]}}\lambdam|\fm|^2 \langle \Lp, T \rangle dx\\
 &\ \ \ \ +\frac{1}{2}\sum_{k=1,2}\int_{\Cp_{\up^k}}\lambdam|\fm|^2   \!\!\!\!\!\! \!\!\!\!\!\! \!\!\!\!\!\! \!\!\!\!\!\! \!\!\!\!\!\! \underbrace{\langle \Lp,\nup_k \rangle}_{\ \ \ \ \ \ \Lp \ \text{is tangential to}\ \Cp_{\up^k} \Rightarrow \text{This term is} \ 0}d\sigma_+.
 \end{split}
\end{equation}

Finally, we obtain by using $\langle \Lp, T \rangle =1$ that
\begin{equation}\label{energy estimates fm in Wp}
 {
 \int_{\Sigma_{t}^{[\up^1,\up^2]}}\lambdam|\fm|^2 dx= \int_{\Sigma_{0}^{[\up^1,\up^2]}}\lambdam|\fm|^2dx + 2\int_{0}^t \int_{\Sigma_\tau^{[\up^1,\up^2]}} \lambdam \fm\cdot\rhom \ dx d\tau.
 }
\end{equation}

\bigskip

We now derive  the estimates for $\fp$ in $W_t^{[\up^1,\up^2]}$. In view of the facts that $\Lm = T + \Zm$ and  $\Lm \lambdap=0$, by taking inner product with $\lambdap \fp$ for the first equation in \eqref{linear equation}, we obtain
\begin{equation}\label{e2}
\frac{1}{2} \Lm  \big( \lambdap|\fp|^2\big)= \lambdap \fp\cdot\rhop.
\end{equation}

We integrate equation \eqref{e2} on $W_{t}^{[\up^1,\up^2]}$. Similar to the previous calculation, by virtue of Stokes formula and the fact that $\div\Zm =0$, the left hand side of \eqref{e2} gives
\begin{align*}
& \ \ \ \frac{1}{2}\doubleint_{W_{t}^{[\up^1,\up^2]}} \Lm  \big( \lambdap|\fp|^2\big) dxd\tau\\
&=\frac{1}{2}\doubleint_{W_{t}^{[\up^1,\up^2]}}  \widetilde{\div} \big(\lambdap|\fp|^2 \Lm\big) dxd\tau-\underbrace{\frac{1}{2}\doubleint_{W_{t}^{[\up^1,\up^2]}}  \lambdap|\fp|^2 \widetilde{\div}  \Lm \ dxd\tau}_{\widetilde{\div}\Lm=0 \ \Rightarrow \  \text{This term is $0$.}}\\
&\overset{\text{Stokes}}{=} \frac{1}{2}\int_{\Sigma_{t}^{[\up^1,\up^2]}}\lambdap|\fp|^2 dx-\frac{1}{2}\int_{\Sigma_{0}^{[\up^1,\up^2]}}\lambdap|\fp|^2 dx +\frac{1}{2}\sum_{k=1,2}\int_{\Cp_{\up^k}}\lambdap|\fp|^2  \langle \Lm,\nup_k \rangle d\sigma_+.
\end{align*}
Finally, we obtain
\begin{equation}\label{energy estimates fp in Wp}
 {
\begin{aligned}
& \ \ \ \int_{\Sigma_{t}^{[\up^1,\up^2]}}\lambdap|\fp|^2dx + \int_{\Cp_{\up^1}}\lambdap|\fp|^2  \langle \Lm,\nup_1 \rangle d\sigma_+\\
 & = \int_{\Sigma_{0}^{[\up^1,\up^2]}}\lambdap|\fp|^2dx + \int_{\Cp_{\up^2}}\lambdap|\fp|^2  \langle \Lm,-\nup_2 \rangle d\sigma_+ + 2\int_{0}^t \int_{\Sigma_\tau^{[\up^1,\up^2]}} \lambdap \fp\cdot\rhop \ dx d\tau.
 \end{aligned}
 }
\end{equation}

Similarly, on $W_{t}^{[\um^1,\um^2]}$, we have

\begin{equation}\label{energy estimates fp in Wm}
 {
 \int_{\Sigma_{t}^{[\um^1,\um^2]}}\lambdap|\fp|^2dx= \int_{\Sigma_{0}^{[\um^1,\um^2]}}\lambdap|\fp|^2dx + 2\int_{0}^t \int_{\Sigma_\tau^{[\um^1,\um^2]}} \lambdap \fp\cdot\rhop \ dx d\tau.
 }
\end{equation}
and
\begin{equation}\label{energy estimates fm in Wm}
 {
\begin{aligned}
& \ \ \ \int_{\Sigma_{t}^{[\um^1,\um^2]}}\lambdam|\fm|^2dx + \int_{\Cm_{\um^2}}\lambdam|\fm|^2  \langle \Lp,\num_2 \rangle d\sigma_-\\
 & = \int_{\Sigma_{0}^{[\um^1,\um^2]}}\lambdam|\fm|^2dx + \int_{\Cm_{\um^1}}\lambdam|\fm|^2  \langle \Lp,-\num_1 \rangle d\sigma_-  + 2\int_{0}^t \int_{\Sigma_\tau^{[\um^1,\um^2]}} \lambdam \fm\cdot\rhom \ dx d\tau.
 \end{aligned}
 }
\end{equation}

Under the bootstrap ansatz \eqref{Bootstrap on geometry} and \eqref{Bootstrap on amplitude}, we study the energy estimates for the following viscous linear system:
\begin{equation}\label{model viscous equation}
\begin{split}
\p_tf_++Z_-\cdot\na f_+-\mu\D f_+&=\r_+,\\
\p_tf_-+Z_+\cdot\na f_--\mu\D f_-&=\r_-,
\end{split}
\end{equation}
where $Z_+$ and $Z_-$ are divergence free.

\begin{prop}\label{prop energy estimates for the viscous linear system}
For all weight functions $\lambda_{\pm}$ with the properties $L_{\pm}\lambda_{\mp}=0$, we have
\begin{equation}\label{energy estimates for the viscous linear system}
\begin{split}
&\ \ \ \sup_{0\leq\tau\leq t}\int_{\Sigma_\tau}\lam_\pm |f_\pm|^2dx + \f12\sup_{u_\pm}\int_{C_{u_\pm}^\pm}\lam_\pm|f_\pm|^2d\sigma_\pm
+\mu\int_0^t\int_{\Sigma_\tau}\lam_\pm|\na f_\pm|^2dxd\tau\\
&\leq 2\int_{\Sigma_0}\lam_\pm|f_\pm|^2dx + 4\int_0^t\int_{\Sigma_\tau}\lam_\pm|f_\pm||\r_\pm|dxd\tau  +{\mu}\int_0^t\int_{\Sigma_\tau}\f{|\na\lam_\pm|^2}{\lam_\pm}|f_\pm|^2dxd\tau
+2\mu^2\sup_{u_\pm}\int_{C_{u_\pm}^\pm}\lam_\pm|\na f_\pm|^2d\sigma_\pm.
\end{split}
\end{equation}
\end{prop}
We remark that except for the coefficients of the first terms in the first and second line of \eqref{energy estimates for the viscous linear system}, the exactly numerical constants are irrelevant to the rest of the proof.
\begin{proof}
We only give the estimates for $f_+$. The estimates on $f_-$ can be derived in the same manner.

By setting $u_-^1=-\infty$ and $u_-^2=\infty$ in \eqref{energy estimates fp in Wm}, we have
$$
\f12\int_{\Sigma_t}\lambda_+|f_+|^2dx\underbrace{-\mu\int_0^t\int_{\Sigma_\tau}\D f_+\cdot\lambda_+f_+dxd\tau
}_{\text{the viscosity term}}=\f12\int_{\Sigma_0}\lambda_+|f_+|^2dx + \int_0^t\int_{\Sigma_\tau}\lambda_+f_+\cdot\r_+dxd\tau.
$$
Integrating by parts, we can deal with the viscosity term as follows:
\begin{align*}
 \text{Viscosity term}&=\mu\int_0^t\int_{\Sigma_\tau}\lam_+ |\na f_+|^2dxd\tau
+\underbrace{\mu\int_0^t\int_{\Sigma_\tau}\p_i\lam_+ f_+\cdot\p_i f_+dxd\tau}_{\text{Cauchy-Schwarz}}\\
&\geq \mu\int_0^t\int_{\Sigma_\tau}\lam_+ |\na f_+|^2dxd\tau-\big(\frac{1}{2}\mu\int_0^t\int_{\Sigma_\tau}\lam_+ |\na f_+|^2dxd\tau
+\frac{1}{2}\mu\int_0^t\int_{\Sigma_\tau} \frac{|\nabla\lam_+|^2}{\lambda_+} |f_+|^2dxd\tau\big)\\
&=\frac{1}{2} \mu\int_0^t\int_{\Sigma_\tau}\lam_+ |\na f_+|^2dxd\tau-\frac{1}{2}\mu\int_0^t\int_{\Sigma_\tau} \frac{|\nabla\lam_+|^2}{\lambda_+} |f_+|^2dxd\tau.
\end{align*}
Therefore, we obtain
\begin{equation}\label{energy identity for vicous system 1}
\begin{split}
&\ \ \ \int_{\Sigma_t}\lambda_+|f_+|^2dx+\mu\int_0^t\int_{\Sigma_\tau}\lambda_+|\na f_+|^2dxd\tau\\
&\leq \int_{\Sigma_0}\lambda_+|f_+|^2dx+2\int_0^t\int_{\Sigma_\tau}\lambda_+f_+\cdot\r_+ dxd\tau+\mu\int_0^t\int_{\Sigma_\tau}\frac{|\nabla\lam_+|^2}{\lambda_+}|f_+|^2dxd\tau.
\end{split}
\end{equation}
By setting $u_+^1=u_+$ and $u_+^2=\infty$ in \eqref{energy estimates fm in Wm},  we have
\begin{equation}\label{energy identity for vicous system 2}
\begin{split}
& \ \ \ \int_{\Sigma_t^{[u_+,+\infty)}}\lambda_+|f_+|^2dx +\underbrace{\int_{C_{u_+}^+}\lambda_+|f_+|^2\langle L_-,\nu_+\rangle d\sigma_+}_{II}
\underbrace{-2\mu\doubleint_{W_t^{[u_+,+\infty)}}\D f_+\cdot\lambda_+f_+dxd\tau}_{I} \\
&= \int_{\Sigma_0^{[u_+,+\infty)}}\lambda_+|f_+|^2dx+2\doubleint_{W_t^{[u_+,+\infty)}}\lambda_+f_+\cdot\r_+dxd\tau,
\end{split}
\end{equation}
where $L_-=(1,Z_-^1,Z_-^2,Z_-^3)$ and $\nu_+=-\f{(\p_tu_+,\na u_+)}{\sqrt{|\p_tu_+|^2+|\na u_+|^2}}$. After an integration by parts, the viscosity term $I$ can be written as
\begin{equation*}
I=\underbrace{2\mu\doubleint_{W_t^{\geq u_+}}\lam_+ |\na f_+|^2dxd\tau}_{I_1} + \underbrace{2\mu\doubleint_{W_t^{\geq u_+}}\p_i\lam_+  f_+\cdot\p_i f_+dxd\tau}_{I_2}\underbrace{-2\mu\int_{C_{u_+}^+}\lam_+ f_+\cdot\sum_{i=1}^3\nu_+^i\p_i f_+d\sigma_+}_{I_3},
\end{equation*}
where $\nu_+=(\nu_+^0,\nu_+^1,\nu_+^2,\nu_+^3)$.

We can bound $I_2$ and $I_3$ by Cauchy-Schwarz inequality:
\begin{equation*}
\begin{split}
|I_2| &\leq \mu \doubleint_{W_t^{[u_+,+\infty)}}\lambda_+|\na f_+|^2dxd\tau+\mu\doubleint_{W_t^{[u_+,+\infty)}}\f{|\na\lambda_+|^2}{\lambda_+}|f_+|^2dxd\tau,\\
|I_3| &\leq\f12\int_{C_{u_+}^+}\lambda_+|f_+|^2d\sigma_+ +2\mu^2\int_{C_{u_+}^+}\lambda_+|\na f_+|^2d\sigma_+.
\end{split}
\end{equation*}
Hence,
\begin{equation}\label{mu01}
\begin{split}
I&\geq \mu \doubleint_{W_t^{[u_+,\infty)}}\lambda_+|\na f_+|^2dxd\tau - \f12\int_{C_{u_+}^+}\lambda_+|f_+|^2d\sigma_+  \\
&\quad-\mu\doubleint_{W_t^{[u_+,\infty)}}\f{|\na\lambda_+|^2}{\lambda_+}|f_+|^2dxd\tau -2\mu^2\int_{C_{u_+}^+}\lambda_+|\na f_+|^2d\sigma_+.
\end{split}
\end{equation}
To bound the term $II$ in \eqref{energy identity for vicous system 2}, we use Lemma \ref{lemma normal derivative control}. Indeed, since $\langle L_-,\nu_+\rangle \sim 1$, we have
\begin{equation*}
II=\int_{C_{u_+}^+}\lambda_+|f_+|^2\langle L_-,\nu_+\rangle d\sigma_+ \sim \int_{C_{u_+}^+}\lambda_+|f_+|^2d\sigma_+,
\end{equation*}
Together with \eqref{energy identity for vicous system 1}, \eqref{energy identity for vicous system 2} and \eqref{mu01}, this completes the proof of the proposition.
\end{proof}

A byproduct of the proof is the energy inequality \eqref{energy identity for vicous system 1}. Since it will be used many times to control the viscosity terms, we restate the estimates in the following lemma:
\begin{corollary} For all weight functions $\lambda_{\pm}$ with the properties $L_{\pm}\lambda_{\mp}=0$, we have
\begin{equation}\label{energy estimates on whole space}
\begin{split}
 &\int_{\Sigma_t}\lambda_\pm|f_\pm|^2dx+\mu\int_0^t\int_{\Sigma_\tau}\lambda_\pm|\na f_\pm|^2dxd\tau\\
 &\leq \int_{\Sigma_0}\lambda_\pm|f_\pm|^2dx +2\int_0^t\int_{\Sigma_\tau}\lambda_\pm f_\pm\cdot\rho_\pm dxd\tau+\mu\int_0^t\int_{\Sigma_\tau}\frac{|\nabla\lam_\pm|^2}{\lambda_\pm}|f_\pm|^2dxd\tau.
\end{split}
\end{equation}
\end{corollary}

By the trace estimates in Lemma \ref{Trace Estimates}, we can indeed remove the last flux term in \eqref{energy estimates for the viscous linear system}:
\begin{corollary}\label{coro energy estimates for the viscous linear system}
We make an extra assumption that $\mu<<1$. For all weight functions $\lambda_{\pm}$ with the properties $L_{\pm}\lambda_{\mp}=0$,  $|\nabla \lambda_{\pm}| \leq |\lambda_{\pm}|$ and $|\nabla^2\lambda_{\pm}|\leq |\lambda_{\pm}|$, we have
\begin{equation}\label{energy estimates for the viscous linear system no flux on the right hand side}
\begin{split}
&\ \ \ \sup_{0\leq\tau\leq t}\int_{\Sigma_\tau}\lam_\pm |f_\pm|^2dx + \f12\sup_{u_\pm}\int_{C_{u_\pm}^\pm}\lam_\pm|f_\pm|^2d\sigma_\pm
+\frac{1}{2}\mu\int_0^t\int_{\Sigma_\tau}\lam_\pm|\na f_\pm|^2dxd\tau \\
&\leq 2\int_{\Sigma_0}\lam_\pm|f_\pm|^2dx + 4\int_0^t\int_{\Sigma_\tau}\lam_\pm|f_\pm||\r_\pm| dxd\tau +2{\mu}\int_0^t\int_{\Sigma_\tau}\f{|\na\lam_\pm|^2}{\lam_\pm}|f_\pm|^2dxd\tau
+2{\mu^2}\int_0^t\int_{\Sigma_\tau}\lam_\pm|\na^2 f_\pm|^2dxd\tau.
\end{split}
\end{equation}
\end{corollary}
\begin{proof}
According to Lemma \ref{Trace Estimates}, we have
\begin{align*}
\mu^2\int_{C_{u_+}^+}\lambda_+|\na f_+|^2d\sigma_+  &\lesssim \mu^2\int_0^t\|\sqrt{\lambda_+}\na f_+\|_{H^1(\Sigma_\tau)}^2d\tau \\
&= \underbrace{\mu^2\int_0^t\int_{\Sigma_\tau}\lambda_+ |\na f_+|^2dxd\tau}_{I}
+\mu^2\int_0^t \underbrace{\|\na\bigl(\sqrt{\lambda_+}\na f_+\bigr)\|_{L^2(\Sigma_\tau)}^2}_{II}d\tau.
\end{align*}
We can ignore the term $I$. The reason is as follows: Since $\mu<<1$, the term $I$ will be absorbed by the viscosity term on the left hand side of \eqref{energy estimates for the viscous linear system}.

We bound the term $II$ as follows:
\begin{align*}
II&\leq\|\sqrt{\lambda_+}\nabla^2 f_+\|_{L^2(\Sigma_\tau)}^2
+\underbrace{\|(\na\sqrt{\lambda_+})\na f_+\|_{L^2(\Sigma_\tau)}^2}_{II_1}.
\end{align*}
We can ignore the term $II_1$. The reason is as follows: since $|\nabla \sqrt{\lambda_+}|^2=\frac{|\nabla \lambdap|^2}{\lambdap}$ and $\mu<<1$, the contribution of the $II_1$ term can be absorbed by the viscosity term on the left hand side of \eqref{energy estimates for the viscous linear system}.

Then, the corollary follows immediately from the above analysis.
\end{proof}

\subsection{Energy estimates on the lowest order terms}\label{subsection Energy estimates on the lowest order terms}
In this section, we will apply Proposition \ref{prop energy estimates for the viscous linear system} to the system
\begin{equation}\label{MHDmu}
\begin{split}
\partial_t  \zp +\Zm \cdot \nabla \zp - {\mu \triangle \zp} &= -\nabla p, \\
\partial_t  \zm +\Zp \cdot \nabla \zm - {\mu \triangle \zm}&= -\nabla p.
\end{split}
\end{equation}
The weight functions $\lambda_\pm$ will be chosen as $\big(\log \wmp\big)^4$. We remark that by choosing the constant weights $\lambda_\pm=1$, we have the energy identities:
\begin{equation}\label{basic energy estimates}
\begin{split}
  \int_{\Sigma_t} |z_\pm|^2dx+2\mu\int_0^t\int_{\Sigma_\tau}|\na z_\pm|^2dxd\tau = \int_{\Sigma_0} |z_\pm|^2dx.
\end{split}
\end{equation}
In particular, it implies that
$$
\mu\int_0^t\int_{\Sigma_\tau}|\na z_\pm|^2dxd\tau\leq \frac{1}{2} \int_{\Sigma_0} |z_\pm|^2dx.
$$
This is the cornerstone of all the estimates in this work.

In this section, our task is to prove the following proposition concerning the lowest order energy estimate.
\begin{proposition}\label{lowest proposition}
Under the  bootstrap ansatz \eqref{Bootstrap on geometry} (or \eqref{Bootstrap on the flow}) and
\beno
\sup_{0\leq l\leq 2}E_\mp^l\leq 2C_1\varepsilon^2,
\eeno
for $\varepsilon$ sufficiently small,
 there holds
\begin{equation}\label{lowest order estimates}
 E_\pm(t) + \f14\sup_{u_\pm}F_\pm(z_\pm) +\frac{1}{2}D_\pm(t)
\lesssim  E_\pm(0)+\sup_{0\leq l\leq 2}\bigl(E_\mp^l\bigr)^{\f12}\sup_{u_\pm}F_\pm^0(\na\zpm)+2\mu D_\pm^0(t).
\end{equation}
\end{proposition}

\subsubsection{Estimates on the pressure}
The current subsection is devoted to derive the following estimates concerning the pressure term $\nabla p$:
\begin{proposition}\label{lemma estimates on pressure}
Under the ansatz \eqref{Bootstrap on geometry}, for all $t\in [0,t^*]$, we have
\begin{equation}\label{estimates on pressure}
 \Big|\int_0^t\int_{\Sigma_\tau}\big(\log \wmp \big)^4 |z_\pm||\na p|dxd\tau\Big| \lesssim\sum_{k=0}^2\bigl(E_\mp^k\bigr)^{\f{1}{2}}\bigl(\sup_{u_\pm}F_\pm(z_\pm)+\sup_{u_\pm} F_\pm^0(\na z_\pm)\bigr).
\end{equation}
\end{proposition}
\begin{proof}
We only derive bound on $I=\big|\int_0^t\int_{\Sigma_\tau}\big(\log \wm \big)^4 |z_+||\na p|dxd\tau\big|$. To do this,  we start with a decomposition on $\na p$. Since $\div z_{\pm}=0$, by taking the divergence of the first equation of \eqref{MHDmu}, we obtain
$$-\D p=\p_i\big(z_+^j	\p_jz_-^i\big).$$
Therefore, on each time slice $\Sigma_\tau$, we have
$$
\na p(\tau,x)=-\f{1}{4\pi}\na\int_{\R^3}\f{1}{|x-y|}\p_i(z_+^j \p_jz_-^i)(\tau,y)dy.$$
We choose a smooth cut-off function $\theta(r)$ so that
\beno
\theta(r)=\left\{\begin{aligned}
&1,\quad\text{for}\quad|r|\leq1,\\
&0,\quad\text{for}\quad|r|\geq2.
\end{aligned}\right.\eeno
After a possible integration by parts, we can split $\nabla p$ as
\begin{equation}\label{decomposition of nabla p}
\begin{split}
\na p(\tau,x)=&\underbrace{-\f{1}{4\pi}\int_{\R^3}\na\f{1}{|x-y|} \cdot \theta(|x-y|) \cdot (\p_iz_-^j\p_jz_+^i)(\tau,y)dy}_{A_1(\tau,x)}\\
&+\underbrace{\f{1}{4\pi}\int_{\R^3}\p_i \Bigl(\na\f{1}{|x-y|}\cdot\bigl(1-\theta(|x-y|)\bigr)\Bigr)\cdot(z_+^j \p_jz_-^i)(\tau,y)dy}_{A_2(\tau,x)}.
\end{split}
\end{equation}
According to this decomposition, we split $I$ into two parts:
$$
I=\underbrace{\int_0^t\int_{\Sigma_\tau}\big(\log \wm \big)^4 |z_+||A_1|dxd\tau}_{I_{1}}+\underbrace{\int_0^t\int_{\Sigma_\tau}\big(\log \wm \big)^4 |z_+||A_2|dxd\tau}_{I_{2}}.
$$

We deal with $I_1$ first. In fact, we have
\begin{align*}
I_{1}&=\int_0^t\int_{\Sigma_\tau}\f{\big(\log \wm \big)^2}{\wp^\frac{1}{2}\big(\log \wp \big)}|z_+| \cdot {\big(\log \wm \big)^2}{\wp^\frac{1}{2}\big(\log \wp \big)}|A_1|dxd\tau\\
&\leq\int_0^t\|\f{\big(\log \wm \big)^2}{\wp^\frac{1}{2}\big(\log \wp \big)}z_+\|_{L^2(\Sigma_\tau)}
\|{\big(\log \wm \big)^2}{\wp^\frac{1}{2}\big(\log \wp \big)}A_1\|_{L^2(\Sigma_\tau)}d\tau.
\end{align*}
By the definition of $A_1$, we have
\begin{equation}\label{u002}
{\big(\log \wm \big)^2}{\wp^\frac{1}{2}\log \wp}|A_1(\tau,x)|
\leq\int_{|y-x|\leq2}\!\!\!\!\!\!\!\!\!\!\!\!\!\!\!\!\f{\overbrace{\{\big(\log \wm \big)^2\wp^\frac{1}{2} \log \wp\} (t,x)}^{\text{weight functions with $x$ as variables}}|\na z_-(\tau,y)||\na z_+(\tau,y)|}{|x-y|^2}dy.
\end{equation}

The following auxiliary lemma allows us to switch the $x$ variables in the above functions to $y$ variables.
\begin{lemma} For $|x-y|\leq 2$, $R\geq100$, we have
\begin{equation}\label{swithing weights x-y <2}
 \wpm(\tau,x)\leq\sqrt2 \wpm(\tau,y), \ \  \log \wpm(\tau,x)\leq 2\log \wpm(\tau,y).
\end{equation}
\end{lemma}
\begin{proof} In fact, by the geometric ansatz \eqref{Bootstrap on geometry} and the mean value theorem, we have
\begin{align*}
|x_i^\pm(\tau,x)|&\leq|x_i^\pm(\tau,y)|+|x_i^\pm(\tau,x)-x_i^\pm(\tau,y)|\\
&\leq|x_i^\pm(\tau,y)|+|x-y|\sup |\nabla x_i^\pm|\\
&\leq|x_i^\pm(\tau,y)|+4,
\end{align*}
where $i=1,2,3$ and $x_3^\pm = u_\pm$. Thus, for $R\geq 100$, we have
\begin{equation*}
 \wpm(\tau,x)=\big(R^2 + |x^\pm|^2\big)^\frac{1}{2}(\tau,x)\leq{\sqrt{2}}\big(R^2 + |x^\pm|^2\big)^\frac{1}{2}(\tau,y)
 ={\sqrt2}\wpm(\tau,y)
\end{equation*}
This proves the first inequality in \eqref{swithing weights x-y <2}. For the second one, we have
\begin{equation*}
\log \wpm(\tau,x)\leq \log(\sqrt2)+\log \wpm(\tau,y)\leq 2\log \wpm(\tau,y).
\end{equation*}
This ends the proof of the lemma.
\end{proof}

We return to \eqref{u002} and we now have
\begin{equation}\label{mu25}
\begin{split}
&\ \ \ {\big(\log \wm \big)^2}{\wp^\frac{1}{2}\log \wp}|A_1(\tau,x)|\\
&\leq 8\int_{|y-x|\leq2}\f{{\big(\log \wm \big)^2\wp^\frac{1}{2} \log \wp (t,y)}|\na z_-(\tau,y)||\na z_+(\tau,y)|}{|x-y|^2}dy\\
&\leq 8\|\wp (\log \wp)^2 \na z_-\|_{L^\infty}\int_{|x-y|\leq2}\frac{1}{|x-y|^2}\frac{\big( \log\wm \big)^2|\na z_+(\tau,y)|}{\wp^\frac{1}{2} \log \wp}dy\\
&\stackrel{\eqref{Sobolev}}{\lesssim}\sum_{k=0}^2\bigl(E_-^k(\tau)\bigr)^{\f{1}{2}}\int_{|x-y|\leq2}\frac{1}{|x-y|^2}\frac{\big( \log\wm \big)^2|\na z_+(\tau,y)|}{\wp^\frac{1}{2} \log \wp}dy.
\end{split}
\end{equation}
By Young's inequality, we obtain
\begin{equation}\label{mu16}
\begin{split}
\|{\big(\log \wm \big)^2}{\wp^\frac{1}{2}\log \wp}A_1(\tau,x)\|_{L^2(\Sigma_\tau)}&\lesssim \sum_{k=0}^2\bigl(E_-^k(\tau)\bigr)^{\f{1}{2}}\|\f{1}{|x|^2}\|_{L^1(|x|\leq2)}
\|\frac{\big( \log\wm \big)^2\na z_+}{\wp^\frac{1}{2} \log \wp}\|_{L^2(\Sigma_\tau)}\\
&\lesssim \sum_{k=0}^2\bigl(E_-^k(\tau)\bigr)^{\f{1}{2}} \|\frac{\big( \log\wm \big)^2\na z_+}{\wp^\frac{1}{2} \log \wp}\|_{L^2(\Sigma_\tau)}.
\end{split}
\end{equation}
Therefore, we can bound $I_{1}$ as follows:
\beq\label{mu55}\begin{aligned}
I_1 &\lesssim \sum_{k=0}^2\bigl(E_-^k(\tau)\bigr)^{\f{1}{2}}\int_0^t\|\frac{\big( \log\wm \big)^2 z_+}{\wp^\frac{1}{2} \log \wp}\|_{L^2(\Sigma_\tau)}\|\frac{\big( \log\wm \big)^2\na z_+}{\wp^\frac{1}{2} \log \wp}\|_{L^2(\Sigma_\tau)}d\tau\\
&\lesssim \sum_{k=0}^2\bigl(E_-^k(\tau)\bigr)^{\f{1}{2}}\Bigl(\underbrace{\int_0^t\int_{\Sigma_\tau}\frac{\big( \log\wm \big)^4 |z_+(\tau,x)|^2}{\wp\big( \log \wp\big)^2}dxd\tau}_{I_{11}}
+ \underbrace{\int_0^t \int_{\Sigma_\tau}\frac{\big( \log\wm \big)^4 |\nabla z_+(\tau,x)|^2}{\wp\big( \log \wp\big)^2}dxd\tau}_{I_{12}}\Bigr).
\end{aligned}\eeq
We will use the flux to bound $I_{11}$ and $I_{12}$. For this purpose, we consider the following change of coordinates on $\mathbb{R}^3\times[0,t^*)$:
\begin{align*}
\Phi_+:\mathbb{R}^3\times[0,t^*)&\rightarrow \mathbb{R}^3\times[0,t^*),  \\
(x_1,x_2,x_3,\tau)&\mapsto (x_1,x_2,u_+,t)=(x_1,x_2,u_+(\tau,x),t).
\end{align*}
In view of the geometric ansatz \eqref{Bootstrap on geometry}, it is straightforward to see that the Jacobian $d\Phi_+$ of $\Phi_+$ satisfies
\beq\label{Jacobian}
\det(d\Phi_+)=\p_3u_+=1+O(\varepsilon).
\eeq

Therefore, to compute the integral $I_{11}$, up to the Jacobian factor coming from the change of coordinates, we use $(x_1,x_2,u_+,t)$ as reference coordinates. As a result, by using the obtained result \eqref{surface measure} that $d\sigma_+=(\sqrt{2}+O(\varepsilon))d{x_1}d{x_2}dt$, we have
\begin{equation}\label{mu15}
\begin{split}
I_{11}&\lesssim \int_{u_+}\Big(\int_{C_{u_+}^+}\frac{\big( \log\wm \big)^4 |z_+(\tau,x)|^2}{\wp\big( \log \wp\big)^2}d\sigma_+\Big)du_+\\
&\leq\int_{u_+}\Big(\int_{C_{u_+}^+}\frac{\big( \log\wm \big)^4 |z_+(\tau,x)|^2}{(R^2+|u_+|^2)^\frac{1}{2}\big( \log ((R^2+|u_+|^2)^\frac{1}{2})\big)^2}d\sigma_+\Big)du_+.
\end{split}
\end{equation}
Since $u_+$ is constant along $C_{u_+}^+$, we then have
\begin{equation}\label{mu15p}
\begin{split}
I_{11}&\leq \sup_{u_+}\Big[\int_{C_{u_+}^+}\big( \log\wm \big)^4 |z_+|^2d\sigma_+\Big] \int_{\R} \underbrace{\f{1}{(R^2+|u_+|^2)^\frac{1}{2}\big( \log ((R^2+|u_+|^2)^\frac{1}{2})\big)^2}}_{\text{integrable!}}du_+\\
&\lesssim\sup_{u_+}F_+(z_+)
\end{split}
\end{equation}

For $I_{12}$, proceeding exactly in the same manner as for \eqref{mu15} and \eqref{mu15p}, we obtain
\begin{align*}
I_{12}\lesssim\sup_{u_+}\Big[\int_{C_{u_+}^+}\big( \log\wm \big)^4 |\na z_+|^2d\sigma_+\Big]=\sup_{u_+} F_+^0(\na z_+).
\end{align*}
We then conclude that
\begin{equation}\label{bound on I_1}
I_{1}\lesssim\sum_{k=0}^2\bigl(E_-^k\bigr)^{\f{1}{2}}\bigl(\sup_{u_+}F_+(z_+)+\sup_{u_+} F_+^0(\na z_+)\bigr) 
\end{equation}

We turn to the estimate on $I_{2}$. We first split $A_{2}(t,x)$ as
\begin{equation}\label{mu27}
\begin{split}
|A_2(\tau,y)|&\lesssim \underbrace{\int_{\R^3}\f{1-\theta(|x-y|)}{|x-y|^3}|z_+(\tau,y)||\nabla z_-(\tau,y)|dy}_{A_{21}(t,x)}
+\underbrace{\int_{\R^3} \f{\theta'(|x-y|)}{|x-y|^2}|z_+(\tau,y)||\nabla z_-(\tau,y)|dy}_{A_{22}(t,x)}.
\end{split}
\end{equation}
Since the support of $\theta'$ is in $[1,2]$, the contribution of the $A_{22}(t,x)$ term to $I_2$ is essentially the same as the contribution of $A_1(t,x)$ to $I_{1}$, i.e.,
\begin{align*}
\int_0^t\int_{\Sigma_\tau}(\log\wm)^4|z_+||A_{22}|dxd\tau\lesssim\sum_{k=0}^2\bigl(E_-^k\bigr)^{\f{1}{2}}\sup_{u_+}F_+(z_+).
\end{align*}
 Therefore,
\begin{align*}
I_{2}\lesssim\sum_{k=0}^2\bigl(E_-^k\bigr)^{\f{1}{2}}\sup_{u_+}F_+(z_+)	+ \underbrace{\int_0^t\int_{\Sigma_\tau} \big(\log \wm\big)^4 |z_+||A_{21}|dxd\tau}_{I_{21}}.
\end{align*}

To bound $I_{21}$, we first prove the following lemma concerning the weights:
\begin{lemma}\label{|x-y|>1}
For $|y-x|\geq 1$, $R\geq100$, we have
\begin{equation}\label{swithing weights x-y >2}
 \wpm(\tau,x) \leq2|x-y|\wpm(\tau,y), \ \ \ \log\wpm(\tau,x) \leq 4\log\wpm(\tau,y)\log\big(2|y-x|\big).
\end{equation}
\end{lemma}
\begin{proof}
By \eqref{differentiate weights} and mean value theorem, we have
\begin{align*}
 \wpm(\tau,x) &\leq \wpm(\tau,y)+2|x-y|\\
 &\leq 2|x-y|\wpm(\tau,y).
\end{align*}
Therefore,
\begin{align*}
 \log\wpm(\tau,x)\leq \log\big(2|x-y|)+\log\wpm(\tau,y)\leq 4\log\big(2|x-y|)\log\wpm(\tau,y).
\end{align*}
This completes the proof of the lemma.
\end{proof}
By the above Lemma \ref{|x-y|>1}, we have
\begin{align*}
I_{21}&=\int_0^t\int_{\Sigma_\tau} \big(\log \wm(\tau,x) \big)^2 |z_+(\tau,x)|\Big[\big(\log \wm(\tau,x) \big)^2 |A_{21}(\tau,x)|\Big]dxd\tau\\
&\lesssim\int_0^t\int_{\Sigma_\tau} \big(\log \wm  \big)^2 |z_+| \underbrace{\Big[\int_{|y-x|\geq 1}\f{\big(\log(2|x-y|)\big)^2}{|x-y|^3}\big(\log \wm(\tau,y)  \big)^2 |z_+(\tau,y)||\nabla z_-(\tau,y)|dy\Big]}_{A_3(\tau,x)}dxd\tau.
\end{align*}
We now rewrite $A_3(\tau,x)$, i.e., the term in the bracket in last line, as follows
\begin{align*}
\int_{|y-x|\geq 1}\f{\big(\log(2|x-y|)\big)^2}{|x-y|^3}\frac{\big(\log \wm(\tau,y) \big)^2}{\wp(\tau,y)^\frac{1}{2}\log \wp(\tau,y)} |z_+(\tau,y)|\frac{\wp(\tau,y)\big(\log \wp(\tau,y) \big)^2}{\underbrace{\wp(\tau,y)^\frac{1}{2}\log \wp(\tau,y)}_{D}}|\nabla z_-(\tau,y)|dy
\end{align*}
We will change the denominator $D$, which is a function in $y$, to a function in $x$ so that we can move it to the outside of the integral. In fact, according to \eqref{swithing weights x-y >2}, we have
\begin{align*}
 \frac{1}{\wp(\tau,y)^\frac{1}{2}\log \wp(\tau,y)}\lesssim \frac{|x-y|^\frac{1}{2}\log(2|x-y|)}{\wp(\tau,x)^\frac{1}{2}\log \wp(\tau,x)}.
\end{align*}
Therefore,
\begin{align*}
I_{21}&\lesssim\int_0^t\int_{\Sigma_\tau} \frac{\big(\log \wm  \big)^2|z_+|}{\wp^\frac{1}{2} \log \wp}\cdot\\
  &\quad\cdot\underbrace{\int_{|y-x|\geq 1}\f{\big(\log(2|x-y|)\big)^3}{|x-y|^{\frac{5}{2}}}\Bigl(\frac{\big(\log \wm  \big)^2|z_+|}{\wp^\frac{1}{2} \log \wp} \cdot \wp \big(\log \wp\big)^2|\nabla z_-|dy\Bigr)(\tau,y)}_{A_4(t,x)}dxd\tau\\
&\lesssim \int_0^t\|\frac{\big(\log \wm  \big)^2|z_+|}{\wp^\frac{1}{2} \log \wp}\|_{L^2(\Sigma_\tau)}\|A_4(t,x)\|_{L^2(\Sigma_\tau)}d\tau.
\end{align*}
For $A_4(t,x)$, according to the Young's inequality, we have
\begin{equation}\label{A_4}
\begin{split}
 \|A_4(t,x)\|_{L^2(\Sigma_\tau)}&=\|\frac{(\log(2|x|))^3}{|x|^\frac{5}{2}}\chi_{|x|\geq 1}*\Big(\frac{\big(\log \wm  \big)^2|z_+|}{\wp^\frac{1}{2} \log \wp} \cdot \wp \big(\log \wp\big)^2|\nabla z_-|\Big)\|_{L^2(\Sigma_\tau)}\\
 &\leq \|\frac{(\log(2|x|))^3}{|x|^\frac{5}{2}}\chi_{|x|\geq 1}\|_{L^2(\Sigma_\tau)}\|\frac{\big(\log \wm  \big)^2|z_+|}{\wp^\frac{1}{2} \log \wp} \cdot \wp \big(\log \wp\big)^2|\nabla z_-|\|_{L^1(\Sigma_\tau)}\\
 &\lesssim \|\frac{\big(\log \wm  \big)^2}{\wp^\frac{1}{2} \log \wp} z_+\|_{L^2(\Sigma_\tau)}\|\wp \big(\log \wp\big)^2\nabla z_-\|_{L^2	 (\Sigma_\tau)}\\
 &\lesssim\bigl(E_-^0(\tau)\bigr)^{\f{1}{2}}\|\frac{\big(\log \wm  \big)^2}{\wp^\frac{1}{2} \log \wp}z_+ \|_{L^2(\Sigma_\tau)}.
 \end{split}
\end{equation}
Hence,
\begin{align*}
I_{21}&\lesssim \bigl(E_-^0\bigr)^{\f{1}{2}} \int_0^t\|\frac{\big(\log \wm  \big)^2}{\wp^\frac{1}{2} \log \wp}z_+ \|^2_{L^2(\Sigma_\tau)}d\tau.
\end{align*}
The righthand side is exactly the same as for $I_{11}$, so it is bounded by $\bigl(E_-^0\bigr)^{\f{1}{2}}\sup_{u_+} F_+(z_+)$. As a result, we also have
\begin{equation}\label{bound on I_2}
I_{2} \lesssim\sum_{k=0}^2\bigl(E_-^k\bigr)^{\f{1}{2}}\sup_{u_+}F_+(z_+).
\end{equation}
Two inequalities \eqref{bound on I_1} and \eqref{bound on I_2} complete the proof of the proposition.
\end{proof}

\subsubsection{Estimates on the viscosity terms}
The current subsection is devoted to derive the following estimates on the viscosity term:
\begin{proposition}\label{lemma estimates on viscosity}
Under the ansatz \eqref{Bootstrap on geometry}, for all $t\in [0,t^*]$ and
$R\geq100$, we have
\begin{equation}\label{estimates on viscosity}
 \mu\int_0^t\int_{\Sigma_\tau}\big(\log \wmp \big)^4 |\nabla z_\pm|^2dxd\tau\leq 1000\bigl(E_\pm(0)+\sum_{l=0}^2\bigl(E_\mp^l\bigr)^{\f12}(F_\pm+F_\pm^0)\bigr).
\end{equation}
\end{proposition}
\begin{proof} We will use \eqref{energy estimates on whole space}  twice by induction. Indeed, for the $k^{\text{th}}$-time, we will choose the weight function $\lambda_\pm = (\log\wmp)^{2k}$, where $k=1,2$. In this situation, \eqref{energy estimates on whole space} shows that
\begin{align*}
&\ \ \ \int_{\Sigma_t}(\log\wmp)^{2k}|z_\pm|^2dx + \mu\int_0^t\int_{\Sigma_\tau}(\log\wmp)^{2k}|\na z_\pm|^2dxd\tau \\
&\leq \int_{\Sigma_0}(\log\wmp)^{2k}|z_\pm|^2dx +2\int_0^t\int_{\Sigma_\tau}\big|(\log\wmp)^{2k} z_\pm\big|\big|\nabla p\big|dxd\tau +\mu\int_0^t\int_{\Sigma_\tau}\frac{|\nabla\big(\log\wmp\big)^{2k}|^2}{(\log\wmp)^{2k}}|z_\pm|^2dxd\tau.
\end{align*}
We only treat $z_+$ and the estimates on $z_-$ can be derived in the same manner. Since $k\leq 2$, the first term on the righthand side are bounded by the initial data and the second term can be bounded thanks to Proposition \ref{lemma estimates on pressure} from last subsection. Therefore, we have
\begin{align*}
\mu\int_0^t\int_{\Sigma_\tau}(\log\wm)^{2k}|\na z_+|^2dxd\tau \leq E_+(0)+\sum_{l=0}^2\bigl(E_-^l\bigr)^{\f12}(F_++F_+^0)+\mu\int_0^t\int_{\Sigma_\tau}\frac{|\nabla\big(\log\wm\big)^{2k}|^2}{(\log\wm)^{2k}}|z_+|^2dxd\tau.
\end{align*}
 According to \eqref{differentiate weights} (and its immediate consequences in the Lemma), we see that
$$\frac{|\nabla\big(\log\wm\big)^{2k}|^2}{(\log\wm)^{2k}}\leq16k^2\frac{\big(\log\wm\big)^{2(k-1)}}{\wm^2}.$$
Therefore,  we have
\begin{equation}\label{mu iteration for viscosity}
\mu\int_0^t\int_{\Sigma_\tau}(\log\wm)^{2k}|\na z_+|^2dxd\tau \leq E_+(0)+\sum_{l=0}^2\bigl(E_-^l\bigr)^{\f12}(F_++F_+^0)+16k^2\mu\int_0^t\int_{\Sigma_\tau}\frac{\big(\log\wm\big)^{2(k-1)}}{\wm^2}|z_+|^2dxd\tau.
\end{equation}
\begin{item}
 \item{\textbf{Step 1.} $k=1$.} It suffices to estimate $\int_0^t\int_{\Sigma_\tau}\frac{|z_+|^2}{\wm^2}dxd\tau$ in \eqref{mu iteration for viscosity}. Noticing that $\wm=(R^2+|x^-|^2)^2$ and $x^-(t,\psi_-(t,y))=y$, we will use Lagrangian coordinates $y$. Therefore, since $\det\bigl(\frac{\p\psi_-(t,y)}{\p y}\bigr)=1$, we have
\begin{align*}
\int_0^t\int_{\Sigma_\tau}\frac{|z_+|^2}{\wm^2} dxd\tau &= \int_0^t\int_{\Sigma_0}\frac{|z_+(\tau, \psi_-(t,y))|^2}{R^2+|y|^2}dyd\tau.
\end{align*}
Now by using the Hardy's inequality\footnote{~On $\mathbb{R}^3$, the Hardy's inequality is $$ \int_{\mathbb{R}^3} \frac{|f(x)|^2}{|x|^2}dx \leq 4 \int_{\mathbb{R}^3}|\nabla f(x)|^2dx.$$} on each $\Sigma_0$, we obtain
\begin{align*}
\int_0^t\int_{\Sigma_\tau}\frac{|z_+|^2}{\wm^2} dxd\tau \leq 4\int_0^t\int_{\Sigma_0} |\nabla_y z_+(\tau,\psi_-(t,y))|^2 dyd\tau.
\end{align*}
 On the other side, we have
\begin{equation*}
\nabla_y z_+(\tau, \psi_-(t,y)) = (\na_x z_+)|_{x=\psi_-(t,y)}\frac{\p\psi_-(t,y)}{\p y}.
\end{equation*}
Then changing back to the Eulerian coordinates on $\Sigma_\tau$ and using \eqref{Bootstrap on the flow} with small $\epsilon$, we obtain
\begin{align*}
\mu\int_0^t\int_{\Sigma_\tau}\frac{|z_+|^2}{\wm^2} dxd\tau \leq 5\mu\int_0^t\int_{\Sigma_\tau} |\nabla z_+(\tau, x)|^2 dxd\tau
\stackrel{\eqref{basic energy estimates}}{\leq} 5\int_{\Sigma_0}|z_+|^2d\tau\leq \f{5E_+(0)}{2(\log R)^4}.
\end{align*}
Here we used the most basic energy identity \eqref{basic energy estimates}.

Finally, going back to \eqref{mu iteration for viscosity}, taking $R\geq100$, we obtain
\begin{equation}\label{step 1 viscosity}
\mu\int_0^t\int_{\Sigma_\tau}(\log\wm)^2|\na z_+|^2 \leq\f{7}{6}E_+(0)+\sum_{l=0}^2\bigl(E_-^l\bigr)^{\f12}(F_++F_+^0).
\end{equation}

\item{\textbf{Step 2.} $k=2$.}  It suffices to estimate $\int_0^t\int_{\Sigma_\tau}\frac{(\log\wm)^2 |z_+|^2}{\wm^2}dxd\tau$ in \eqref{mu iteration for viscosity}. As we have observed in {\bf{Step 1}}, we can freely switch the Eulerian coordinates $x$ to Lagrangian coordinates $y$. We have
\beno\begin{aligned}
\mu\int_0^t\int_{\Sigma_\tau}&\frac{(\log \wm)^2 |z_+|^2}{\wm^2} dxd\tau \stackrel{\det(\f{\p\psi_-}{\p y})=1}{=} \mu\int_0^t\int_{\Sigma_0}\frac{(\log(R^2+|y|^2)^{\f12})^2|z_+(\tau,\psi_-(\tau,y))|^2}{R^2+|y|^2}dyd\tau\\
&\stackrel{\text{Hardy}}{\leq} 4\mu\int_0^t \int_{\Sigma_0} \Big| \nabla_y \Big[\log(R^2+|y|^2)^{\frac{1}{2}} z_+(\tau, \psi_-(\tau,y))\Big]\Big|^2 dy d\tau\\
&\leq8\mu\int_0^t \int_{\Sigma_0} \Bigl(\frac{ \big|z_+(\tau, \psi_-(\tau,y))\big|^2}{R^2+|y|^2} + (\log(R^2+|y|^2)^\frac{1}{2})^2\big|\nabla_yz_+(\tau, \psi_-(\tau,y))\big|^2\Bigr) dyd\tau\\
&\leq 8\mu\int_0^t \int_{\Sigma_\tau}\Bigl(\frac{ |z_+|^2}{\wm^2} +\f{5}{4}(\log \wm)^2 |\nabla z_+|^2\Bigr)dx d\tau.
\end{aligned}\eeno
Since both terms in the last line have been estimated in \textbf{Step 1}, we obtain that
 \begin{equation}\label{step 2 viscosity 0}
\mu\int_0^t\int_{\Sigma_\tau}\frac{(\log \wm)^2 |z_+|^2}{\wm^2} dxd\tau \leq 13E_+(0)+10\sum_{l=0}^2\bigl(E_-^l\bigr)^{\f12}(F_++F_+^0).
\end{equation}
 In view of \eqref{mu iteration for viscosity} and \eqref{step 2 viscosity 0}, we obtain
\begin{equation}\label{step 2 viscosity}
\mu\int_0^t\int_{\Sigma_\tau}(\log\wm)^4 |\na z_+|^2 dxd\tau\leq 1000\bigl(E_+(0)+\sum_{l=0}^2\bigl(E_-^l\bigr)^{\f12}(F_++F_+^0)\bigr).
\end{equation}
\end{item}
This completes the proof.
\end{proof}

\subsubsection{Completion of the estimates on lowest order terms}
In this subsection, we will end the proof of Proposition \ref{lowest proposition}.
\begin{proof}[Proof of Proposition \ref{lowest proposition}]
We specialize \eqref{energy estimates for the viscous linear system no flux on the right hand side} to the current situation: $f_\pm = z_\pm$, $\rho_\pm=\nabla p$ and $\lambda_\pm =\big(\log\wmp\big)^4$. Hence,
\begin{equation*}
\begin{split}
&\ \ \ \int_{\Sigma_t}\big(\log\wmp\big)^4 |z_\pm|^2dx + \f12\sup_{u_\pm}\int_{C_{u_\pm}^\pm}\big(\log\wmp\big)^4|z_\pm|^2d\sigma_\pm
+\frac{1}{2}\mu\int_0^t\int_{\Sigma_\tau}\big(\log\wmp\big)^4|\na z_\pm|^2 dxd\tau\\
&\leq2\int_{\Sigma_0}\big(\log\wmp\big)^4|z_\pm|^2dx + 4\int_0^t\int_{\Sigma_\tau}\big(\log\wmp\big)^4|z_\pm||\nabla p|dxd\tau  +128{\mu}\int_0^t\int_{\Sigma_\tau}\f{(\log\wmp)^2}{\wmp^2}|z_\pm|^2dxd\tau\\
&\ \ \ +2{\mu^2}\int_0^t\int_{\Sigma_\tau}\big(\log\wmp\big)^4|\na^2 z_\pm|^2dxd\tau.
\end{split}
\end{equation*}
The second and third terms have been controlled by \eqref{estimates on pressure} and \eqref{step 2 viscosity 0} in the previous two subsections (notice that for $\lambda_\pm =\big(\log\wmp\big)^4$ we have $\f{|\na\lam_\pm|^2}{\lam_\pm} \leq64 \frac{(\log\wmp)^2}{\wmp^2}$). While the last term  is controlled by $2\mu D_\pm^0(t)$. We then have
\begin{equation*}
\begin{split}
 &\int_{\Sigma_t}\big(\log\wmp\big)^4 |z_\pm|^2dx + \f12\sup_{u_\pm}\int_{C_{u_\pm}^\pm}\big(\log\wmp\big)^4|z_\pm|^2d\sigma_\pm
 +\frac{1}{2}\mu\int_0^t\int_{\Sigma_\tau}\big(\log\wmp\big)^4|\na z_\pm|^2 dxd\tau\\
&\lesssim E_\pm(0)+\sum_{l=0}^2\bigl(E_\mp^l\bigr)^{\f12}(F_\pm+F_\pm^0)+2\mu D_\pm^0(t).
\end{split}
\end{equation*}
In other words, if $\sum_{l=0}^2E_\mp^l\leq 2C_1\varepsilon^2$ with $\varepsilon$ sufficiently small, we have
\beno
 E_\pm(t) + \f14\sup_{u_\pm}F_\pm(z_\pm) +\frac{1}{2}D_\pm(t)
\lesssim  E_\pm(0)+\sum_{l=0}^2\bigl(E_\mp^l\bigr)^{\f12} F_\pm^0+2\mu D_\pm^0(t).
\eeno
This proves the proposition.
\end{proof}

\subsection{Energy estimates for the first order terms}
This section is devoted to derive energy estimates on $\nabla z_\pm$. For this purpose, we first commute one derivative with \eqref{MHD in Elsasser} and we obtain
\begin{equation*}
\begin{split}
\partial_t  \p\zp +\Zm \cdot \nabla \p\zp  - \mu \triangle \p\zp&= -\p\nabla p-\p\zm\cdot\na\zp, \\
\partial_t  \p\zm +\Zp \cdot \nabla \p\zm - \mu \triangle \p\zp&= -\p\nabla p-\p\zp\cdot\na\zm,
\end{split}
\end{equation*}
where $\p\zpm$ denotes for some $\partial_i \zpm$ with $i=1,2,3$. The main result of this section is stated as follows:
\begin{proposition}\label{first order proposition}
Assume that $\|z_\pm\|_{L^\infty}\leq\f{1}{2}$, $R\ge100$ and
\beno
\sup_{0\leq l\leq 3} E_{\mp}^l\leq 2C_1\varepsilon^2,
\eeno
for $\varepsilon$ sufficiently small. Then under the ansatz \eqref{Bootstrap on geometry} (or \eqref{Bootstrap on the flow}), for all $t\in [0,t^*]$, we have
\begin{equation}\label{first order estimates}
\begin{split}
E^0_\pm(t) +\sup_{u_\pm}F^0_\pm(\na z_\pm)
+D^0_\pm(t)
\lesssim &E_\pm^0(0)+\sup_{0\leq l\leq 3}\bigl(E_\mp^l\bigr)^{\f12}\sup_{u_\pm}F_\pm^1(j_\pm)\\
&+\mu\int_0^t\int_{\Sigma_\tau}\big(\log\wmp\big)^4|\na z_\pm|^2dxd\tau
+ 2\mu D_\pm^1(t).
\end{split}
\end{equation}
\end{proposition}

\begin{remark} Thanks to \eqref{estimates on viscosity}, we can bounded the third term in the righthand side of \eqref{first order estimates} by
$ E_\pm(0)+\sum_{l=0}^2\bigl(E_\mp^l\bigr)^{\f12}(F_\pm+F_\pm^0)$. Then we obtain
\beq\label{first order energy estimate}
E^0_\pm(t) +\sup_{u_\pm}F^0_\pm(\na z_\pm)
+D^0_\pm(t)
\lesssim E_\pm(0)+E_\pm^0(0)+\sum_{l=0}^3\bigl(E_\mp^l\bigr)^{\f12}\sup_{u_\pm}(F_\pm(\zpm)+F_\pm^1(j_\pm))+ 2\mu D_\pm^1(t).
\eeq
\end{remark}

\subsubsection{Estimates on the pressure}
The subsection is devoted to derive the following estimates concerning the pressure $p$:
\begin{proposition}\label{lemma estimates on pressure 2}
Under the assumptions of Proposition \ref{first order proposition}, for all $t\in [0,t^*]$, we have
\begin{equation}\label{estimates on pressure 2}
 \Big|\int_0^t\int_{\Sigma_\tau}\wmp^2 \big(\log \wmp \big)^4 |\nabla z_\pm||\nabla^2 p|dxd\tau\Big|
 \lesssim\sum_{k=0}^3\bigl(E_\mp^k\bigr)^{\f12}\bigl(E_\pm^0+\sup_{u_\pm}(F_\pm^0(\na\zpm)+F_\pm^1(j_\pm))\bigr).
\end{equation}
\end{proposition}
\begin{proof}
We only derive bound on $I=\int_0^t\int_{\Sigma_\tau}\wm^2 \big(\log \wm \big)^4 |\nabla z_+||\na^2 p|dxd\tau$. Similar to the proof of Proposition \ref{lemma estimates on pressure}, we choose the same cut-off function $\theta(r)$ and we have
$$\p\nabla p(\tau,x)=-\f{1}{4\pi}\int_{\R^3}\big(\p\nabla\f{1}{|x-y|}\big)\cdot (\p_iz_+^j \p_jz_-^i)(\tau,y)dy.$$
We split $\p\nabla p$ as
\begin{equation}\label{decomposition of nabla^2 p}
\begin{split}
\p\na p(\tau,x)=&\underbrace{-\f{1}{4\pi}\int_{\R^3}\p\nabla\f{1}{|x-y|} \cdot \theta(|x-y|) \cdot (\p_iz_-^j\p_jz_+^i)(\tau,y)dy}_{A_1(\tau,x)}\\
&-\underbrace{\f{1}{4\pi}\int_{\R^3} \Bigl(\p\nabla\f{1}{|x-y|}\cdot\bigl(1-\theta(|x-y|)\bigr)\Bigr)\cdot(\p_iz_+^j \p_jz_-^i)(\tau,y)dy}_{A_2(\tau,x)}.
\end{split}
\end{equation}
This gives the following decomposition for $I$:
$$
I=\underbrace{\int_0^t\int_{\Sigma_\tau}\wm^2\big(\log \wm \big)^4 |\nabla z_+||A_1|dxd\tau}_{I_{1}}+\underbrace{\int_0^t\int_{\Sigma_\tau}\wm^2\big(\log \wm \big)^4 |\nabla z_+||A_2|dxd\tau}_{I_{2}}.
$$
For $I_1$, after an integration by parts, we first rewrite $A_1$ as (to avoid the non-integrable singularity $\frac{1}{|x-y|}$)
\begin{equation}\label{A_1}
\begin{split}
 A_1&=\underbrace{\f{1}{4\pi}\int_{\R^3}\nabla \f{1}{|x-y|} \cdot \p\theta(|x-y|) \cdot (\p_iz_+^j \p_jz_-^i)(\tau,y)dy}_{A_{11}}\\
 &\ \ +\underbrace{\f{1}{4\pi}\int_{\R^3}\nabla \f{1}{|x-y|} \cdot \theta(|x-y|) \cdot (\p\p_iz_+^j \p_jz_-^i)(\tau,y)dy}_{A_{12}}\\
 &\ \ + \underbrace{\f{1}{4\pi}\int_{\R^3}\nabla \f{1}{|x-y|} \cdot \theta(|x-y|) \cdot (\p_iz_+^j \p\p_jz_-^i)(\tau,y)dy}_{A_{13}}.
\end{split}
\end{equation}
We have
\begin{align*}
I_{1}&=\int_0^t\int_{\Sigma_\tau}\f{\wm \big(\log \wm \big)^2}{\wp^\frac{1}{2}\log \wp }|\nabla z_+| \cdot \wm {\big(\log \wm \big)^2}{\wp^\frac{1}{2} \log \wp }|A_{1}|dxd\tau\\
&\leq\sum_{k=1}^3\int_0^t\|\f{\wm \big(\log \wm \big)^2}{\wp^\frac{1}{2}\log \wp}\nabla z_+\|_{L^2(\Sigma_\tau)}
\|\wm {\big(\log \wm \big)^2}{\wp^\frac{1}{2} \log \wp }A_{1k}\|_{L^2(\Sigma_\tau)}d\tau.
\end{align*}
For $A_{1k}$, since the integration is taken place for $|y-x|\leq 2$, by \eqref{swithing weights x-y <2}, we have
$$ \wpm(\tau,x) \big(\log \wpm(\tau,x)\big)^2 \lesssim \wpm(\tau,y)\big(\log \wpm(\tau,y)\big)^2.$$
In particular, it implies that
\begin{align*}
\ \ \ & \wm{\big(\log \wm \big)^2}{\wp^\frac{1}{2}\log \wp}|A_1(\tau,x)|\\
&\leq\sum_{l_1,l_2=1}^2\int_{|y-x|\leq2}\f{\wm(\tau,y)\big(\log \wm(\tau,y) \big)^2\wp(\tau,y)^\frac{1}{2} \log \wp (t,y)|\na	^{l_1} z_-(\tau,y)||\na^{l_2} z_+(\tau,y)|}{|x-y|^2}dy\\
&\leq\sum_{l_1,l_2=1}^2\|\wp (\log \wp)^2 \na^{l_1} z_-\|_{L^\infty}\int_{|x-y|\leq2}\frac{1}{|x-y|^2}\Bigl(\frac{\wm\big( \log\wm \big)^2|\nabla^{l_2} z_+|}{\wp^\frac{1}{2} \log \wp}\Bigr)(\tau,y)dy\\
&\stackrel{\eqref{Sobolev}}{\lesssim}\sum_{l=0}^3\bigl(E_-^l\bigr)^{\f12}\sum_{l_2=1}^2\int_{|x-y|\leq2}\frac{1}{|x-y|^2}\Bigl(\frac{\wm\big( \log\wm \big)^2|\na^{l_2} z_+|}{\wp^\frac{1}{2} \log \wp}\Bigr)(\tau,y)dy,
\end{align*}
where $(l_1,l_2) = (1,1), (1,2)$ or $(2,1)$. By Young's inequality, we have
\begin{equation*}
\begin{split}
\|\wm{\big(\log \wm \big)^2}{\wp^\frac{1}{2}\log \wp}A_1(\tau,x)\|_{L^2(\Sigma_\tau)}&\lesssim\sum_{l=0}^3\bigl(E_-^l\bigr)^{\f12}\|\f{1}{|x|^2}\|_{L^1(|x|\leq2)}
\sum_{l_2=1}^2\|\frac{\wm \big( \log\wm \big)^2\na^{l_2} z_+}{\wp^\frac{1}{2} \log \wp}\|_{L^2(\Sigma_\tau)}\\
&\lesssim\sum_{l=0}^3\bigl(E_-^l\bigr)^{\f12} \sum_{l_2=1}^2\|\frac{\wm \big( \log\wm \big)^2\na^{l_2} z_+}{\wp^\frac{1}{2} \log \wp}\|_{L^2(\Sigma_\tau)}.
\end{split}
\end{equation*}
Therefore, thanks to H\"older inequality and div-curl lemma, we can bound $I_{1}$ as follows:
\beq\label{estiamate involving x-y<2}\begin{aligned}
I_1 &\lesssim \sum_{l=0}^3\bigl(E_-^l\bigr)^{\f12}\int_0^t\|\frac{\wm \big( \log\wm \big)^2 \nabla z_+}{\wp^\frac{1}{2} \log \wp}\|_{L^2(\Sigma_\tau)}\sum_{l_2=1}^2\|\frac{\wm \big(\log\wm \big)^2\na^{l_2} z_+}{\wp^\frac{1}{2} \log \wp}\|_{L^2(\Sigma_\tau)}d\tau\\
&\stackrel{\eqref{z to j inequality}, {\text{H\"older}}}{\lesssim}\sum_{l=0}^3\bigl(E_-^l\bigr)^{\f12}\int_0^t\int_{\Sigma_\tau}\frac{\wm^2\big( \log\wm \big)^4 |\nabla z_+(\tau,x)|^2}{\wp\big( \log \wp\big)^2}dxd\tau\\
&\qquad
+ \sum_{l=0}^3\bigl(E_-^l\bigr)^{\f12}\int_0^t \int_{\Sigma_\tau}\frac{\wm^2\big( \log\wm \big)^4 |\nabla j_+(\tau,x)|^2}{\wp\big( \log \wp\big)^2}dxd\tau.
\end{aligned}\eeq
This is exactly the same situation as for \eqref{mu55} in the proof of Proposition \ref{lemma estimates on pressure}. We repeat the procedure to obtain
\begin{equation}\label{bound on I_1 in the second pressure estimates}
I_{1} \lesssim\sum_{l=0}^3\bigl(E_-^l\bigr)^{\f12}\sup_{u_+}\bigl(F_+^0(\na z_+)+F_+^1(j_+)\bigr).
\end{equation}

\bigskip

We move to the bound on $I_{2}$. We first make the following observation:
\begin{lemma}
For $|x-y|\geq 1$, $R\geq100$, we have
\begin{equation}\label{swithing weights x-y >2 second pressure estimates}
 \wpm(\tau,x)\big(\log\wpm(\tau,x))^2\leq 8\wpm(\tau,y)\big(\log\wpm(\tau,y)\big)^2+4|x-y|\big(\log(4|x-y|)\big)^2.
\end{equation}
\end{lemma}
\begin{proof}
To see this, we recall that by \eqref{differentiate weights} and mean value theorem, we have
$$\wpm(\tau,x) \leq \wpm(\tau,y)+2|x-y|.$$
Therefore, we either have $\frac{1}{2}\wpm(\tau,x) \leq \wpm(\tau,y)$ or $\frac{1}{2}\wpm(\tau,x) \leq 2|x-y|$.

If $\frac{1}{2}\wpm(\tau,x) \leq \wpm(\tau,y)$, we have
\begin{align*}
 \wpm(\tau,x)\big(\log\wpm(\tau,x)\big)^2 &\leq 2\wpm(\tau,y)\big(\log\big(2\wpm(\tau,y))\big)^2\\
 &\leq 4\wpm(\tau,y)\Big(\big(\log 2\big)^2+ \big(\log \wpm(\tau,y)\big)^2\Big)\\
 &\leq 8\wpm(\tau,y)\log\big(\wpm(\tau,y)).
\end{align*}

If $\frac{1}{2}\wpm(\tau,x) \leq 2|x-y|$, we have
\begin{align*}
 \wpm(\tau,x)\big(\log\wpm(\tau,x)\big)^2 &\leq 4|x-y|\big(\log\big(4|x-y|)\big)^2.
\end{align*}

This completes the proof of the lemma.
\end{proof}
According to the lemma, we have
\begin{align*}
I_{2}&=\int_0^t\int_{\Sigma_\tau} \wm(\log \wm )^2 |\nabla z_+ |\cdot\wm(\log \wm)^2 |A_{2}|dxd\tau\\
&\lesssim\underbrace{\int_0^t\int_{\Sigma_\tau} \wm(\log \wm)^2 |\nabla z_+ | \underbrace{\int_{|x-y|\geq1} \f{1}{|x-y|^3}\cdot\bigl(\wm(\log \wm)^2|\na z_+| |\nabla z_-|\bigr)(\tau,y)dy}_{B_1(\tau,x)}dxd\tau}_{I_{21}}\\
&\ + \underbrace{\int_0^t\int_{\Sigma_\tau} \wm(\log \wm )^2 |\nabla z_+|\underbrace{\int_{|x-y|\geq1} \f{\big(\log(4|x-y|)\big)^2}{|x-y|^2}\cdot \bigl(|\na z_+|  |\na z_-|\bigr)(\tau,y)dy}_{B_2(\tau,x)}}_{I_{22}}.
\end{align*}

To deal with $I_{21}$, we bound $B_{1}(\tau,x)$ by
\begin{align*}
\int_{|y-x|\geq 1}\f{1}{|x-y|^3}\frac{\wm\big(\log \wm  \big)^2}{\wp ^\frac{1}{2}\log \wp } |\nabla z_+ |\cdot\frac{\wp \big(\log \wp  \big)^2}{\underbrace{\wp ^\frac{1}{2}\log \wp }_{D}}|\nabla z_- |dy
\end{align*}
According to \eqref{swithing weights x-y >2}, we have
\begin{equation}\label{switch D}
 \frac{1}{\wp(\tau,y)^\frac{1}{2}\log \wp(\tau,y)}\lesssim \frac{|x-y|^\frac{1}{2}\log(2|x-y|)}{\wp(\tau,x)^\frac{1}{2}\log \wp(\tau,x)}.
\end{equation}
Therefore,
\begin{align*}
I_{21}&\lesssim\int_0^t\int_{\Sigma_\tau} \frac{\wm\big(\log \wm  \big)^2|\nabla z_+|}{\wp^\frac{1}{2} \log \wp}  \underbrace{\int_{|y-x|\geq 1}\f{\log(2|x-y|)}{|x-y|^{\frac{5}{2}}}\frac{\wm\big(\log \wm  \big)^2|\nabla z_+|}{\wp^\frac{1}{2} \log \wp} \cdot \wp \big(\log \wp\big)^2|\nabla z_-|dy}_{B_1'(t,x)}dxd\tau\\
&\lesssim \int_0^t\|\frac{\wm\big(\log \wm  \big)^2|\nabla z_+|}{\wp^\frac{1}{2} \log \wp}\|_{L^2(\Sigma_\tau)}\|B_1'(t,x)\|_{L^2(\Sigma_\tau)}d\tau.
\end{align*}
Since $\frac{\log(2|x|)}{|x|^\frac{5}{2}}\chi_{|x|\geq 1} \in L^2(\mathbb{R}^3)$, we can repeat the proof of \eqref{A_4} to obtain
\begin{align*}
 \|B_1'(\tau,x)\|_{L^2(\Sigma_\tau)}\lesssim \bigl(E_-^0\bigr)^{\f12}\|\frac{\wm\big(\log \wm  \big)^2}{\wp^\frac{1}{2} \log \wp}\nabla z_+ \|_{L^2(\Sigma_\tau)}.
\end{align*}
Hence,
\begin{align*}
I_{21}&\lesssim \bigl(E_-^0\bigr)^{\f12}\int_0^t\|\frac{\wm \big(\log \wm  \big)^2}{\wp^\frac{1}{2} \log \wp}\nabla z_+ \|^2_{L^2(\Sigma_\tau)}d\tau
\lesssim\bigl(E_-^0\bigr)^{\f12}\sup_{u_+}F_+^0(\na z_+).
\end{align*}

To deal with $I_{22}$, we have
\beno
I_{22}\leq\int_0^t\|\wm \big(\log \wm  \big)^2\na z_+\|_{L^2(\Sigma_\tau)}\|B_2\|_{L^2(\Sigma_\tau)}d\tau\lesssim\bigl(E_+^0\bigr)^{\f12}\int_0^t\|B_2\|_{L^2(\Sigma_\tau)}d\tau.
\eeno
Then we only need to bound $\|B_2\|_{L^2(\Sigma_\tau)}$. We rewrite $B_2$ as follows
\beno
B_2\leq\int_{|x-y|\geq1} \f{\big(\log(4|x-y|)\big)^2}{|x-y|^2}\cdot\f{\wm \big(\log \wm  \big)^2|\na z_+|(\tau,y)\cdot\wp \big(\log \wp  \big)^2|\na z_-|(\tau,y)}{ \wm \big(\log \wm  \big)^2\wp \big(\log \wp  \big)^2(\tau,y)}dy.
\eeno
Since $\|\zpm\|_{L^\infty}\leq\f12$, by virtue of the separation property \eqref{separate weight}, we have
\beq\label{se1}
\f{1}{\wm \big(\log \wm  \big)^2\wp \big(\log \wp  \big)^2(\tau,y)}\lesssim\f{1}{(R^2+\tau^2)^{\f12}(\log(R^2+\tau^2)^{\f12})^2}.
\eeq
Notice that $ \f{\big(\log(4|x|)\big)^2}{|x|^2}\chi_{|x|\geq1}\in L^2(\R^3)$. Then we obtain
\beno\begin{aligned}
\|B_2\|_{L^2(\Sigma_\tau)}&\stackrel{\eqref{se1},\text{Young's}}{\lesssim}\f{1}{(R^2+\tau^2)^{\f12}(\log(R^2+\tau^2)^{\f12})^2}\|\f{\big(\log(4|x|)\big)^2}{|x|^2}\chi_{|x|\geq1} \|_{L^2(\R^3)}\\
&\quad\cdot\|\wm \big(\log \wm  \big)^2\na z_+\|_{L^2(\Sigma_\tau)}\|\wp \big(\log \wp  \big)^2\na z_-\|_{L^2(\Sigma_\tau)}\\
&\lesssim\f{\bigl(E_+^0\bigr)^{\f12}\bigl(E_-^0\bigr)^{\f12}}{(R^2+\tau^2)^{\f12}(\log(R^2+\tau^2)^{\f12})^2},
\end{aligned}\eeno
which gives rise to
\beno
I_{22}\lesssim\varepsilon^3\int_0^t\f{1}{(R^2+\tau^2)^{\f12}(\log(R^2+\tau^2)^{\f12})^2}d\tau\lesssim\bigl(E_-^0\bigr)^{\f12}E_+^0.
\eeno

Combining all the estimates, we complete the proof of the proposition.
\end{proof}

\subsubsection{Completion of the estimates on the first order terms}

\begin{proof}[Proof of Proposition \ref{first order proposition}]
We specialize \eqref{energy estimates for the viscous linear system no flux on the right hand side} to the current situation: $f_\pm = \p z_\mp$, $\rho_\pm=\p\nabla p + \p z_\pm\cdot \nabla z_\mp$ and $\lambda_\pm =\wmp^2\big(\log\wmp\big)^4$, with $\p=\p_1,\p_2,\p_3$. Hence,
\begin{equation*}
\begin{split}
&\ \ \ \int_{\Sigma_t}\wmp^2\big(\log\wmp\big)^4 |\na z_\pm|^2dx + \f12\sup_{u_\pm}\int_{C_{u_\pm}^\pm}\wmp^2\big(\log\wmp\big)^4|\na z_\pm|^2d\sigma_\pm\\
&\qquad
+\frac{1}{2}\mu\int_0^t\int_{\Sigma_\tau}\wmp^2\big(\log\wmp\big)^4|\na^2 z_\pm|^2 dxd\tau\\
&\leq2\int_{\Sigma_0}\wm^2\big(\log\wmp\big)^4|\na z_\pm|^2dx + 4\int_0^t\int_{\Sigma_\tau}\wmp^2\big(\log\wmp\big)^4|\nabla z_\pm|\big(|\nabla^2 p| +|\nabla z_+| |\nabla z_-|\big)dxd\tau\\
&\ \ \ +2{\mu}\int_0^t\int_{\Sigma_\tau}\f{|\na\lam_\pm|^2}{\lam_\pm}|\na z_\pm|^2dxd\tau + 2{\mu^2}\int_0^t\int_{\Sigma_\tau}\wmp^2\big(\log\wmp\big)^4|\na^3 z_\pm|^2dxd\tau.
\end{split}
\end{equation*}
Notice that the first part involving $\na^2p$ of the second term on the righthand side can be estimated by \eqref{estimates on pressure 2} while the last term can be bounded by $2\mu D_\pm^1$.
For $\lambda_\pm =\wmp^2\big(\log\wmp\big)^4$, we have
$$\f{|\na\lam_\pm|^2}{\lam_\pm} \lesssim \big(\log\wmp\big)^4.$$
We then have
\begin{equation*}
\begin{split}
&\ \  E^0_\pm(t) + \sup_{u_\pm}F^0_\pm(\na z_\pm)
+D_\pm^{0}(t) \\
&\lesssim E_\pm^0(0)+\sum_{k=0}^3\bigl(E_\mp^k\bigr)^{\f12}\bigl(E_\pm^0+\sup_{u_\pm}(F_\pm^0(\na\zpm)+F_\pm^1(j_\pm))\bigr)
+\mu\int_0^t\int_{\Sigma_\tau}\big(\log\wmp\big)^4|\na z_\pm|^2dxd\tau\\
&\qquad+ 2\mu D_\pm^1+ \underbrace{\int_0^t\int_{\Sigma_\tau}\wmp^2\big(\log\wmp\big)^4|\nabla z_\pm||\nabla z_+||\nabla z_-|dxd\tau}_{\text{nonlinear interaction} I_\pm}.
\end{split}
\end{equation*}
It remains to bound the nonlinear interaction term $I_\pm$. We only handle $I_+$.
\begin{equation*}
\begin{split}
I_+&=\int_0^t\int_{\Sigma_\tau} \frac{\wm^2\big(\log\wm\big)^4 |\nabla z_+|^2 }{\wp\big(\log\wp\big)^2}\underbrace{\wp\big(\log\wp\big)^2|\na z_-|}_{L^\infty}dxd\tau\\
&\stackrel{\eqref{Sobolev}}{\lesssim}\sum_{l=0}^2\bigl(E_-^l\bigr)^{\f12}\int_0^t\int_{\Sigma_\tau} \frac{\wm^2\big(\log\wm\big)^4 |\nabla z_+|^2 }{\wp\big(\log\wp\big)^2}dxd\tau.
\end{split}
\end{equation*}
Similar to \eqref{mu15} and \eqref{mu15p}, we obtain
\begin{equation*}
I_+\lesssim\sum_{l=0}^2\bigl(E_-^l\bigr)^{\f12}\sup_{u_+}\int_{C_{u^+}^+}\wm^2\big(\log\wm\big)^4 |\nabla z_+|^2d\sigma_+
\lesssim\sum_{l=0}^2\bigl(E_-^l\bigr)^{\f12} F_+^0.
\end{equation*}
Then we have
\beq\label{first order estimates 0}
\begin{split}
&\ \  E^0_\pm(t) + \sup_{u_\pm}F_\pm^0(\na \zpm)
+D_\pm^{0}(t) \\
&\lesssim E_\pm^0(0)+\sum_{k=0}^3\bigl(E_\mp^k\bigr)^{\f12}\bigl(E_\pm^0+F_\pm^0+F_\pm^1\bigr)+\mu\int_0^t\int_{\Sigma_\tau}\big(\log\wmp\big)^4|\na z_\pm|^2dxd\tau+ 2\mu D_\pm^1(t).
\end{split}
\eeq
Since $\sup_{0\leq l\leq 3}E_\mp^l\leq 2C_1\varepsilon^2$ for sufficiently small $\varepsilon$, we obtain
\begin{equation*}
\begin{split}
\f12E^0_\pm(t) +\f12\sup_{u_\pm}F^0_\pm(\na z_\pm)
+D^0_\pm(t)
\lesssim & E_\pm(0)+E_\pm^0(0)+\sup_{0\leq l\leq 3}\bigl(E_\mp^l\bigr)^{\f12}F_\pm^1\\
&+\mu\int_0^t\int_{\Sigma_\tau}\big(\log\wmp\big)^4|\na z_\pm|^2dxd\tau
+ 2\mu D_\pm^1(t).
\end{split}
\end{equation*}
This ends the proof of the proposition.
\end{proof}

\begin{remark}
 Estimate \eqref{first order estimates} in Proposition \ref{first order proposition} is not good in the sense that we use one more derivative of flux term, i.e. $F_\pm^1$, in the righthand side which is caused by the nonlocal and nonlinear  term $\na p$. It will bring the trouble to close the energy estimates. This is the main reason that we turn to the investigation of the system of $j_{\pm}=\curl z_{\pm}$.
\end{remark}

\subsection{Energy estimates on higher order terms}\label{subsection Energy estimates on higher order terms}
To derive higher order energy estimates, we first commute derivatives with the vorticity equations. For a given multi-index $\beta$ with $1\leq|\beta|\leq N_*$, we apply $\p^\beta$ to the system \eqref{main equations for j} and we obtain
\begin{equation}\label{main equations for j beta}
\left\{\begin{aligned}
&\p_tj_+^{(\beta)}+Z_-\cdot\na j_+^{(\beta)}-\mu\D j_+^{(\beta)}=\r_+^{(\beta)},\\
&\p_tj_-^{(\beta)}+Z_+\cdot\na j_-^{(\beta)}-\mu\D j_-^{(\beta)}=\r_-^{(\beta)},
\end{aligned}\right.\end{equation}
where source terms $\r_\pm^{(\beta)}$ are defined as
\begin{align*}
\r_+^{(\beta)}&=-\p^\beta(\na z_-\wedge\na z_+)-[\p^\beta,z_-\cdot\na] j_+,\\
\r_-^{(\beta)}&=-\p^{\beta}(\na z_+\wedge\na z_-)-[\p^\beta,z_+\cdot\na] j_-.
\end{align*}

Then we could obtain the following proposition concerning the energy estimates to \eqref{main equations for j beta}.

\begin{proposition}\label{higher order proposition}
Assume that $R\ge100$, $\mu$ is very small and
\beno
E_\pm^k\leq 2C_1\varepsilon^2, \ \ \text{for}\ \ 0\leq k\leq N_*
\eeno
for $\varepsilon$ sufficiently small. Then under the assumption \eqref{Bootstrap on geometry} (or \eqref{Bootstrap on the flow}), we obtain
\begin{equation}\label{sum of the higher order energy estimate}
\begin{split}
&\ \ \sum_{k=1}^{N_*}\bigl(E_\pm^k(t) + \sup_{u_\pm}F_\pm^{k}(j_\pm)
+D^{k}_+(t)\bigr) \\
&\lesssim \sum_{k=1}^{N_*} E_\pm^{k}(0)+\sup_{k\leq N_*}\bigl(E_\pm^k\bigr)^{\f12}\sup_{u_\pm}F_\pm^0(\na\zpm)+\f{1}{R^2} E_\pm^{0}(t)+\f{2}{R^2} D_\pm^{0}(t)\\
&\qquad
+{\mu^2}\int_0^t\int_{\Sigma_\tau}\wmp^2\big(\log\wmp\big)^4|\na j_\pm^{(N_*+1)}|^2dxd\tau.
\end{split}
\end{equation}
\end{proposition}
\begin{proof}
We divide the proof into several steps:

{\bf Step 1.} {\bf  Energy estimate for the linear system.} Applying \eqref{energy estimates for the viscous linear system no flux on the right hand side} to \eqref{main equations for j beta} and choosing the weight functions $\lambda_\pm$ to be $\wmp^2 \big(\log\wmp\big)^4$ yield (we only deal with the left-traveling waves)
\begin{equation}\label{energy estimates for higher order terms}
\begin{split}
&\ \ \ \int_{\Sigma_t}\wm^2 \big(\log\wm\big)^4  |j_+^{(\beta)}|^2dx + \sup_{u_+}F_+^{(\beta)}(j_+)
+\frac{1}{2}\mu\int_0^t\int_{\Sigma_\tau}\wm^2 \big(\log\wm\big)^4|\na j_+^{(\beta	)}|^2 dxd\tau\\
&\leq2\int_{\Sigma_0}\wm^2 \big(\log\wm\big)^4  |j_+^{(\beta)}|^2dx + 4\underbrace{\int_0^t\int_{\Sigma_\tau}\wm^2\big(\log\wm\big)^4  |j_+^{(\beta)}||\rho^{(\beta)}_+|dxd\tau}_{\text{nonlinear interaction } I}  \\
&\ \ \ +2\underbrace{{\mu}\int_0^t\int_{\Sigma_\tau}(\log \wm)^4|j_+^{(\beta)}|^2dxd\tau}_{\text{diffusion term } II}
+2\underbrace{{\mu^2}\int_0^t\int_{\Sigma_\tau}\wm^2\big(\log\wm\big)^4|\na^2 j_+^{(\beta)}|^2dxd\tau}_{\text{parabolic term } III},
\end{split}
\end{equation}
where we have used the fact that $\Big|\frac{|\nabla \lambda_+|^2}{\lambda_+}\Big| \lesssim \big(\log\wm\big)^4$ for the diffusion term $II$ ($\lambda_+ = \wm^2\big(\log\wm\big)^4$).

{\bf Step 2.} {\bf Estimates on the nonlinear interactions.}
This step is devoted to  the study of the nonlinear interaction term $I$ in \eqref{energy estimates for higher order terms}. The source term $\r_+^{(\beta)}$ in $\eqref{main equations for j beta}$ can be bounded by
\begin{equation}\label{rho beta explicit}
\begin{split}
|\r_+^{(\beta)}|&\leq\sum_{\gamma\leq\beta} C_\beta^\gamma|\na z_-^{(\gamma)}||\na z_+^{(\beta-\gamma)}|
+\sum_{0\neq\gamma\leq\beta} C_\beta^\gamma|z_-^{(\gamma)}||\na j_+^{(\beta-\gamma)}|\\
&\stackrel{|\na j_+^{(\beta-\gamma)}|\leq|\na z_+^{(|\beta|-(|\gamma|-1))}|}{\lesssim}\sum_{k\leq|\beta|} |\na z_-^{(k)}||\na z_+^{(|\beta|-k)}|.
\end{split}
\end{equation}
As a consequence, we obtain
\begin{equation}\label{bound on II_1 pre}
\begin{split}
I &\lesssim \sum_{k\leq|\beta|} \int_0^t\int_{\Sigma_\tau}\wm^2\big(\log\wm\big)^4|j_+^{(\beta)}||\na z_-^{(k)}||\na z_+^{(|\beta|-k)}|dxd\tau.
\end{split}
\end{equation}
According to the size of $|\beta|$, we have two cases:

\medskip

{\textbf{Case 1.}} $1\leq|\beta| \leq N_*-2$.

In this case, we can use Sobolev inequality on $\nabla z_-^{(k)}$ because $k+2\leq N_*$. Therefore, we have
\begin{equation*}
\begin{split}
I&\lesssim \sum_{k\leq|\beta|}
\int_0^t\int_{\Sigma_\tau} \frac{\wm\big(\log\wm\big)^2 j_+^{(\beta)}\cdot \wm\big(\log\wm\big)^2|\na z_+^{(|\beta|-k)}|}{\wp\big(\log\wp\big)^2}\underbrace{\wp\big(\log\wp\big)^2|\na z_-^{(k)}|}_{L^\infty}dxd\tau\\
&\stackrel{\eqref{Sobolev}}{\lesssim}\sum_{k\leq|\beta|}\sum_{l=k}^{k+2}\bigl(E_-^l\bigr)^{\f12}\underbrace{\int_0^t\int_{\Sigma_\tau}\frac{\wm\big(\log\wm\big)^2 j_+^{(\beta)}\cdot \wm\big(\log\wm\big)^2|\na z_+^{(|\beta|-k)}|}{\wp\big(\log\wp\big)^2}dxd\tau}_{I_{1}}.
\end{split}
\end{equation*}
To bound $I_{1}$, we will make use of \eqref{z to j inequality}. In fact, we have
\begin{equation}\label{mu23}
\begin{split}
I_{1}&\lesssim \int_0^t\int_{\Sigma_\tau} \frac{\wm^2\big(\log\wm\big)^4|j_+^{(\beta)}|^2}{\wp\big(\log\wp\big)^2}dxd\tau+\sum_{k\leq|\beta|}\underbrace{\int_0^t\int_{\Sigma_\tau} \frac{\wm^2\big(\log\wm\big)^4|\nabla z_+^{(k)}|^2}{\wp\big(\log\wp\big)^2}dxd\tau}_{\text{apply } \eqref{z to j inequality}}\\
&\lesssim\sum_{1\leq k\leq|\beta|}\int_0^t\int_{\Sigma_\tau} \frac{\wm^2\big(\log\wm\big)^4|j_+^{(k)}|^2}{\wp\big(\log\wp\big)^2}dxd\tau+ \int_0^t\int_{\Sigma_\tau} \frac{\wm^2\big(\log\wm\big)^4|\nabla z_+|^2}{\wp\big(\log\wp\big)^2}dxd\tau.
\end{split}
\end{equation}

\begin{remark}
We remark that, when one applies \eqref{z to j inequality}, one has to stop at $\nabla z_+$ instead of descending one more step to $z_+$. The main obstacle is that the weight functions for the lowest order terms are different from those for higher order terms. In fact, in the course of using \eqref{z to j inequality}, the weight functions for lowest order term takes the form $\frac{|\nabla \lambda|^2}{\lambda}$. Since in \eqref{mu23} the weight function is a mixture of $w_+$ and $w_-$, the differentiation on $\lambda$ cannot lower the weights in $w_-$.
\end{remark}

From \eqref{mu23}, we can repeat the proof for \eqref{mu15} and \eqref{mu15p}. This allows us to use the flux terms to control $I_{1}$. Finally we are led to
\begin{equation*}
I
\lesssim\sup_{l\leq N_*}\bigl(E_-^l\bigr)^{\f12}\bigl(F_+^0+\sum_{1\leq k\leq |\beta|} F_+^k\bigr).
\end{equation*}

\bigskip
{\textbf{Case 2.}} $|\beta| = N_*-1$ or $N_*$.
\bigskip

We rewrite $I$ as
\begin{equation*}
\begin{split}
I&\lesssim \Big(\underbrace{\sum_{k\leq N_*-2}}_{I_1} +\underbrace{\sum_{N_*-1 \leq k\leq|\beta|}}_{I_2}\Big)
\int_0^t\int_{\Sigma_\tau}\wm^2\big(\log\wm\big)^2|j_+^{(\beta)}||\na z_-^{(k)}||\na z_+^{(|\beta|-k)}|dxd\tau.
\end{split}
\end{equation*}
The first sum $I_1$ can  be controlled in the same manner as in {\bf Case 1} so that
 \beno
 I_1\lesssim\sup_{l\leq N_*}\bigl(E_-^l\bigr)^{\f12}\bigl(F_+^0+\sum_{1\leq k\leq |\beta|} F_+^k\bigr).
 \eeno
 For $k\geq N_*-1$, one can not use $L^\infty$ estimates directly on $\na z_-^{(k)}$ since one can not afford more than $N_*$ derivatives (via Sobolev inequality). Instead, we will use $L^\infty$ estimates on $\na z_+^{(|\beta|-k)}$ in a different way:
\begin{equation*}
\begin{split}
I_2&\lesssim \sum_{k =N_*-1}^{|\beta|}
\int_0^t \int_{\Sigma_\tau}\underbrace{\frac{\wm\big(\log\wm\big)^2}{\wp^{\frac{1}{2}}\log\wp}|j_+^{(\beta)}|}_{L^2_\tau L^2_x} \cdot \underbrace{\wp\big(\log \wp\big)^2 |\na z_-^{(k)}|}_{L^\infty_\tau L^2_x} \cdot\underbrace{\frac{\wm\big(\log\wm\big)^2}{\wp^{\frac{1}{2}}\log\wp}|\na z_+^{(|\beta|-k)}|}_{L^2_\tau L^\infty_x}dxd\tau\\
&\lesssim  \sum_{k =N_*-1}^{|\beta|}\!\!\!\!\!\,\,
\underbrace{\|\frac{\wm\big(\log\wm\big)^2}{\wp^{\frac{1}{2}}\log\wp}j_+^{(\beta)} \|_{L^2_\tau L^2_x}}_{\lesssim\bigl(\sup _{u_+}F_+^{\beta}(j_+)\bigr)^{\f12}} \  \underbrace{\|\wp\big(\log \wp\big)^2 \na z_-^{(k)}\|_{L^\infty_\tau L^2_x}}_{\leq\bigl(E_-^k\bigr)^{\f12}} \ \underbrace{\|\frac{\wm\big(\log\wm\big)^2}{\wp^{\frac{1}{2}}\log\wp} \na z_+^{(|\beta|-k)}\|_{L^2_\tau L^\infty_x}}_{I_{2}'},
\end{split}
\end{equation*}
where we bounded the first term in the righthand side in the same manner as \eqref{mu15} and \eqref{mu15p}. Therefore, we obtain
\begin{equation}\label{mu255}
\begin{split}
I_2\lesssim\sup_{N_*-1\leq k\leq |\beta|}\bigl(E_-^k\bigr)^{\f12}\bigl(F_+^{|\beta|}\bigr)^{\f12}\sum_{k =0}^{|\beta|+1-N_*}\underbrace{\|\frac{\wm\big(\log\wm\big)^2}{\wp^{\frac{1}{2}}\log\wp} \na z_+^{(k)}\|_{L^2_\tau L^\infty_x}}_{I_2'}.
\end{split}
\end{equation}
For the most difficult term $I_2'$, we use Sobolev inequality with weight function $\frac{\wm\big(\log\wm\big)^2}{\wp^{\frac{1}{2}}\log\wp}$. In fact, we have
\begin{align*}
&|I_2'|^2\lesssim \|\frac{\wm\big(\log\wm\big)^2}{\wp^{\frac{1}{2}}\log\wp}\na z_+^{(k)}\|_{L^2_\tau (L^2(\Sigma_\tau))}^2+\|\nabla^2\big(\frac{\wm\big(\log\wm\big)^2}{\wp^{\frac{1}{2}}\log\wp}\na z_+^{(k)}\big)\|_{L^2_\tau(L^2(\Sigma_\tau))}^2.
\end{align*}
Since $\nabla^l\big(\frac{\wm\big(\log\wm\big)^2}{\wp^{\frac{1}{2}}\log\wp}\big) \lesssim \frac{\wm\big(\log\wm\big)^2}{\wp^{\frac{1}{2}}\log\wp}$ for $l=1,2$,  we have
\begin{align*}
|I_2'|^2&\lesssim \sum_{l=k}^{k+2}\|\frac{\wm\big(\log\wm\big)^2}{\wp^{\frac{1}{2}}\log\wp}\na z_+^{(l)}\|_{L^2_\tau L^2_x}^2\\
&\stackrel{\eqref{z to j inequality}}{\lesssim} \|\frac{\wm\big(\log\wm\big)^2}{\wp^{\frac{1}{2}}\log\wp}\nabla z_+\|_{L^2_\tau L^2_x}^2+\!\!\!\!\!\!\sum_{1\leq l\leq k+2}\|\frac{\wm\big(\log\wm\big)^2}{\wp^{\frac{1}{2}}\log\wp} j_+^{(l)}\|_{L^2_\tau L^2_x}^2.\\
&\lesssim F_+^0+\sum_{1\leq l\leq k+2} F_+^l,
\end{align*}
where we bound $I_2'$ by the flux terms in a similar manner as for  \eqref{mu15} and \eqref{mu15p}.
Then we have
\beno\begin{aligned}
I_2&\lesssim\sup_{N_*-1\leq k\leq |\beta|}\bigl(E_-^k\bigr)^{\f12}\bigl(F_+^{|\beta|}\bigr)^{\f12}\bigl(F_+^0+\sum_{1\leq l\leq 3} F_+^l\bigr)^{\f12}\\
&\stackrel{N_*\geq 5,\ \text{H\"older}}{\lesssim}\sup_{k\leq N_*}\bigl(E_-^k\bigr)^{\f12}\bigl(F_+^0+\sum_{1\leq k\leq |\beta|} F_+^k\bigr).
\end{aligned}\eeno

Finally, we can bound the nonlinear interaction term $I$ by
\begin{equation}\label{bound on nonlinear interaction}
 I \lesssim \sup_{k\leq N_*}\bigl(E_-^k\bigr)^{\f12}\bigl(F_+^0+\sum_{1\leq k\leq |\beta|} F_+^k\bigr).
\end{equation}

{\bf Step 3.} {\bf Completion of the higher order energy estimates.}
For the diffusion term $II$, we have for $1\leq|\beta|\leq N_*$,
\beno
II\leq\f{2\mu}{R^2}\int_0^t\int_{\Sigma_\tau}\wm^2(\log \wm)^4|\na z_+^{(|\beta|)}|^2dxd\tau=\f{2}{R^2} D_+^{|\beta|-1}(t).
\eeno
Thanks to the div-curl lemma (Lemma \ref{div-curl lemma}), we have
\beno
 \int_{\Sigma_t}\wm^2 \big(\log\wm\big)^4  |\na z_+^{(|\beta|)}|^2dx\leq\int_{\Sigma_t}\wm^2 \big(\log\wm\big)^4  |j_+^{(|\beta|)}|^2dx
+\f{1}{R^2}\int_{\Sigma_t}\wm^2 \big(\log\wm\big)^4  |z_+^{(|\beta|)}|^2dx\eeno and
\beno
&&\mu\int_0^t\int_{\Sigma_\tau}\wm^2 \big(\log\wm\big)^4|\na z_+^{(|\beta|+1)}|^2dxd\tau\\
&&\leq\mu\int_0^t\int_{\Sigma_\tau}\wm^2 \big(\log\wm\big)^4|\na j_+^{(|\beta|)}|^2dxd\tau
+\f{\mu}{R^2}\int_0^t\int_{\Sigma_\tau}\wm^2 \big(\log\wm\big)^4|\na z_+^{(|\beta|)}|^2dxd\tau.
\eeno
These estimates enable us to replace the first and the third term in the left hand side of
\eqref{energy estimates for higher order terms} by the
terms in the left hand side of the above two estimates respectively.
Then with the estimates on $I$ and $II$, we obtain that
\begin{equation}\label{higher order energy estimates 1}
\begin{split}
&\ \ \ E_+^{|\beta|}(t) + \sup_{u_+}F_+^{|\beta|}(j_+)
+D^{|\beta|}_+(t) \\
&\lesssim E_+^{|\beta|}(0)+\sup_{k\leq N_*}\bigl(E_-^k\bigr)^{\f12} F_+^0+\sup_{k\leq N_*}\bigl(E_-^k\bigr)^{\f12}\sum_{1\leq k\leq |\beta|} F_+^k+\f{1}{R^2} E_+^{|\beta|-1}(t)\\
&\qquad+\f{2}{R^2} D_+^{|\beta|-1}(t)
+\underbrace{{\mu^2}\int_0^t\int_{\Sigma_\tau}\wm^2\big(\log\wm\big)^4|\na^2 j_+^{(|\beta|)}|^2dxd\tau}_{\text{parabolic term } III\ \lesssim\ \mu D^{|\beta|+1}_+(t)}.
\end{split}
\end{equation}

Then we sum up \eqref{higher order energy estimates 1} for all $1\leq|\beta|\leq N_*$: each flux term $\sup_{k\leq N_*}\bigl(E_-^k\bigr)^{\f12}\sum_{1\leq k\leq |\beta|} F_+^k$  from the righthand side of \eqref{higher order energy estimates 1}, by virtue of assumption $\sup_{k\leq N_*} E_-^k\leq 2C_1\varepsilon^2$ with sufficiently small $\varepsilon$, they are absorbed by the sum of the lower flux for $1\leq k\leq |\beta|$ on the lefthand side;  each energy term $\f{1}{R^2} E_+^{|\beta|-1}(t)$ and each diffusion term $\f{2}{R^2} D_+^{|\beta|-1}(t)$ except for $|\beta|=1$  can be controlled from the estimates for lower order terms, by taking $R$ large, they are absorbed by lower order energy and diffusion terms on the lefthand side; all parabolic terms $III$ except for $|\beta|=N_*$ can also be controlled from the viscosity terms on the lefthand side for higher order terms($\mu<<1$). Therefore, we finally obtain that
\begin{equation}\label{higher order energy estimates 2}
\begin{split}
&\ \ \sum_{k=1}^{N_*}\bigl(E_+^k(t) + \sup_{u_+}F_+^{k}(j_+)
+D^{k}_+(t)\bigr) \\
&\lesssim \sum_{k=1}^{N_*} E_+^{k}(0)+\sup_{k\leq N_*}\bigl(E_-^k\bigr)^{\f12}F_+^0+\f{1}{R^2} E_+^{0}(t)+\f{2}{R^2} D_+^{0}(t)\\
&\qquad
+{\mu^2}\int_0^t\int_{\Sigma_\tau}\wm^2\big(\log\wm\big)^4|\na j_+^{(N_*+1)}|^2dxd\tau.
\end{split}
\end{equation}
This ends the proof.
\end{proof}

Combining \eqref{lowest order estimates}, \eqref{first order energy estimate} and \eqref{sum of the higher order energy estimate}, we could obtain the following proposition.
\begin{proposition}\label{hyperbolic proposition}
Assume that $R$ is very large, $\mu$ is very small, $\|\zpm\|_{L^\infty}\leq\f12$ and
\beno
E_\pm^k\leq 2C_1\varepsilon^2, \ \ \text{for}\ \ 0\leq k\leq N_*
\eeno
for $\varepsilon$ sufficiently small. Then under the assumption \eqref{Bootstrap on geometry} (or \eqref{Bootstrap on the flow}), we obtain
\begin{equation}\label{hyperbolic energy estimates}
\begin{split}
&\ \ E_\pm+\sum_{0\leq k\leq N_*}E_\pm^{k}+ \sup_{u_\pm}F_\pm(z_\pm)+\sup_{u_\pm}F_\pm^{0}(\na z_\pm)+\sum_{1\leq k\leq N_*} \sup_{u_\pm}F_\pm^{k}(j_\pm)
+D_\pm+\sum_{0\leq k\leq N_*}D_{\pm}^{k}\\
&\lesssim E_\pm(0)+\sum_{0\leq k\leq N_*} E_\pm^{k}(0) + \underbrace{{\mu^2}\int_0^t\int_{\Sigma_\tau}\wm^2\big(\log\wm\big)^4|\na j_+^{(N_*+1)}|^2dxd\tau}_{\text{top order parabolic term}}.
\end{split}
\end{equation}
\end{proposition}

\subsection{Top order parabolic estimates}
This section is devoted to a typical parabolic type estimate designed to control the highest order terms due to the presence of non-zero viscosity. We only study the estimates for the left-traveling Alfv\'{e}n wave $z_+$. The estimates for the right-traveling waves can be derived exactly in the same manner.

For $|\beta|= N_*+1$, we work with the following system of equations
\begin{equation}\label{top system}
\left\{\begin{aligned}
&\p_tj_+^{(\beta)}+Z_-\cdot\na j_+^{(\beta)}-\mu\D j_+^{(\beta)}=\r_+^{(\beta)},\\
&\p_tj_-^{(\beta)}+Z_+\cdot\na j_-^{(\beta)}-\mu\D j_-^{(\beta)}=\r_-^{(\beta)}.
\end{aligned}\right.\end{equation}
Then we shall prove the following proposition.

 \begin{proposition}\label{Top proposition}
Assume that $R\ge100$  and $\mu$ is sufficiently small. Then under the ansatz \eqref{Bootstrap on geometry}, we have
\begin{equation}\label{top order parabolic estimates}\begin{aligned}
 &\mu E_+^{N_*+1}(t)+\mu D_+^{N_*+1}(t)\\
 &\lesssim\mu E_+^{N_*+1}(0) +\mu E_+^{N_*}(t)+\mu D_+^{N_*}(t)+\big(\sup_{l\leq N_*} \bigl(E_-^l\bigr)^{\f12} + \bigl(\sum_{k=1}^{N_*} D_-^k(t)\bigr)^{\f12}\big)\big(\sup_{l\leq N_*}E_+^l+\sum_{k=1}^{N_*} D_+^k(t)\big).
\end{aligned}\end{equation}
\end{proposition}
\begin{proof}

Applying \eqref{energy estimates on whole space} to \eqref{top system} and choosing the weight functions $\lambda_\pm=\mu\wmp^2 \big(\log\wmp\big)^4$
(we only deal with the left-traveling waves) yield
\begin{equation}\label{mu22}
\begin{split}
&\ \ \ \mu\int_{\Sigma_t}\wm^2 \big(\log\wm\big)^4|j_+^{(\beta)}|^2dxd\tau+\mu^2\int_{0}^t\int_{\Sigma_\tau} \wm^2 \big(\log\wm\big)^4 |\na j_+^{(\beta)}|^2dxd\tau\\
&\lesssim\mu E_+^{|\beta|+1}(0)
 +\underbrace{ \mu \int_0^t\int_{\Sigma_\tau}\wm^2 \big(\log\wm\big)^4|\r_+^{(\beta)}|\cdot|j_+^{(\beta)}|dxd\tau}_{\text{nonlinear interaction } I}
 +\underbrace{ \mu^2 \int_0^t\int_{\Sigma_\tau}\big(\log\wm\big)^4|j_+^{(\beta)}|^2dxd\tau}_{\text{diffusion term } II\lesssim\mu D_+^{|\beta|-1}(t)}
\end{split}
\end{equation}
 It remains to control the nonlinear term $I$. We recall \eqref{rho beta explicit} ($|\beta|=N_*+1$)
$$
|\r_+^{(\beta)}|\lesssim \sum_{k \leq N_*+1}|\na z_-^{(k)}||\na z_+^{(N_*+1-k)}|.
$$

We rewrite $I$ as
\begin{align*}
&I\lesssim \mu\Big(\underbrace{\sum_{k \leq N_*-2}}_{I_{1}}+\underbrace{\sum_{N_*-1\leq k\leq N_*+1}}_{I_{2}}\Big)\int_0^t\int_{\Sigma_\tau}\wm^2 \big(\log\wm\big)^4|\na z_-^{(k)}||\na z_+^{(N_*+1-k)}||j_+^{(\beta)}|dxd\tau.
\end{align*}
For $I_{1}$, since $k\leq N_*-2$, we bound $\wp \big(\log\wp\big)^2 \nabla z_-^{(k)}$ in $L^\infty$.
  Hence,
\begin{align*}
I_{1}&\lesssim\sum_{k \leq N_*-2}\underbrace{\|\nabla z_-^{(k)}\|_{L^\infty_\tau L^\infty_x}}_{\stackrel{\eqref{Sobolev}}{\lesssim}\sum_{l=k}^{k+2}\bigl(E_-^l\bigr)^{\f12}}\cdot\underbrace{\sqrt\mu\|\wm\big(\log\wm\big)^2\na z_+^{(N_*+1-k)}\|_{L^2_\tau L^2_x}}_{\lesssim\bigl(D_+^{N_*-k}(t)\bigr)^{\f12}}\cdot\underbrace{\sqrt\mu
\|\wm\big(\log\wm\big)^2j_+^{(\beta)}\|_{L^2_\tau L^2_x}}_{\lesssim\bigl(D_+^{N_*}(t)\bigr)^{\f12}}\\
&\stackrel{\text{H\"older}}{\lesssim}\sup_{l\leq N_*}\bigl(E_-^l\bigr)^{\f12}\cdot \sum_{k=1}^{N_*}D_+^k(t).
\end{align*}
For $I_{2}$, we proceed as follows:
\begin{equation*}
\begin{split}
I_{2}&\leq\sum_{k=N_*-1}^{N_*+1}\underbrace{\|\wm \big(\log\wm\big)^2\na z_+^{(N_*+1-k)}\|_{L^\infty_\tau L^\infty_x}}_{\stackrel{\eqref{Sobolev}}{\lesssim}\sum_{l=N_*-k+1}^{N_*-k+3}\bigl(E_+^l\bigr)^{\f12}}\cdot\underbrace{\sqrt\mu \|\na z_-^{(k)}\|_{L^2_\tau L^2_x}}_{\lesssim\bigl(D_-^{k-1}(t)\bigr)^{\f12}}\cdot\underbrace{\sqrt\mu\|\wm \big(\log\wm\big)^2 j_+^{(\beta)}\|_{L^2_\tau L^2_x}}_{\lesssim\bigl(D_+^{N_*}(t)\bigr)^{\f12}}\\
&\stackrel{\text{H\"older}}{\lesssim}  \bigl(\sum_{k=1}^{N_*} D_-^k(t)\bigr)^{\f12}\big(\sup_{l\leq N_*} E_+^l+\sum_{k=1}^{N_*} D_+^k(t)\big).
\end{split}
\end{equation*}

Finally, we have
\begin{equation}\label{bound on I top parabolic}
I\lesssim\big(\sup_{l\leq N_*} \bigl(E_-^l\bigr)^{\f12} + \bigl(\sum_{k=1}^{N_*} D_-^k(t)\bigr)^{\f12}\big)\big(\sup_{l\leq N_*} E_+^l+\sum_{k=1}^{N_*} D_+^k(t)\big).
\end{equation}

Going back to \eqref{mu22}, by virtue of div-curl lemma (Lemma \ref{div-curl lemma}), we can replace the first term and the second term in the lefthand side of \eqref{mu22}
by $\mu E_+^{N_*+1}(t)-\mu E_+^{N_*}(t)$ and $\mu D_+^{N_*+1}(t)-\mu D_+^{N_*}(t)$. Then
thanks to \eqref{bound on I top parabolic},
we obtain the top order parabolic estimates \eqref{top order parabolic estimates}. This ends the proof of the proposition.
\end{proof}

Combining \eqref{hyperbolic energy estimates} and \eqref{top order parabolic estimates}, we obtain the total energy estimates and then close the energy estimates.
\begin{proposition}\label{total proposition}
Assume that $R\ge 100$, $\mu$ is very small, $\|\zpm\|_{L^\infty}\leq\f12$ and
  \beno
E_\pm^k+D_\pm^k\leq 2C_1\varepsilon^2, \ \ \text{for}\ \ 0\leq k\leq N_*
\eeno
for $\varepsilon$ sufficiently small. Then under the assumption \eqref{Bootstrap on geometry} (or \eqref{Bootstrap on the flow}), we obtain
\begin{equation}\label{total energy estimates}
\begin{split}
&E_\pm+\sum_{0\leq k\leq N_*}E_\pm^{k}+\mu E_\pm^{N_*+1}+ \sup_{u_\pm}F_\pm(z_\pm)+\sup_{u_\pm}F_\pm^{0}(\na z_\pm)+\sum_{1\leq k\leq N_*} \sup_{u_\pm}F_\pm^{k}(j_\pm)
\\
&\ \ +D_\pm+\sum_{0\leq k\leq N_*}D_{\pm}^{k}+\mu D_{\pm}^{N_*+1} \lesssim E_\pm(0)+\sum_{k=0}^{N_*} E_\pm^{k}(0)+ \mu E_\pm^{N_*+1}(0).
\end{split}
\end{equation}
\end{proposition}

\subsection{Proof of the main a priori estimates and Theorem \ref{global existence for MHDmu}}
We now complete the continuity argument (from Section \ref{continuity subsection}) and hence the proof of Theorem \ref{global existence for MHDmu}. It consists of four steps.

{\bf Step 1.} {\it Improving ansatz \eqref{Bootstrap on energy}.}
Under the ansatz \eqref{Bootstrap on geometry}, \eqref{Bootstrap on amplitude} and \eqref{Bootstrap on energy}
by virtue of Proposition \ref{total proposition}, taking $\varepsilon$ sufficiently small, we can find such $C_1>0$ such that
\begin{equation*}
\begin{split}
&E_\pm+\sum_{0\leq k\leq N_*}E_\pm^{k}+\mu E_\pm^{N_*+1}+ \sup_{u_\pm}F_\pm(z_\pm)+\sup_{u_\pm}F_\pm^{0}(\na z_\pm)+\sum_{1\leq k\leq N_*} \sup_{u_\pm}F_\pm^{k}(j_\pm)
\\
&\ \ +D_\pm+\sum_{0\leq k\leq N_*}D_{\pm}^{k}+\mu D_{\pm}^{N_*+1} \leq C_1\mathcal{E}_0^\mu\leq C_1\varepsilon^2.
\end{split}
\end{equation*}
This improves the Ansatz \ref{Bootstrap on energy}.

{\bf Step 2.} {\it Improving ansatz \eqref{Bootstrap on geometry} (or equivalently \eqref{Bootstrap on the flow}).}
We just improve the  ansatz \eqref{Bootstrap on the flow}.

We recall that $\psi_{\pm}(t,y)$ be the flow generated by $Z_{\pm}$ and they are given by
\begin{equation*}
\psi_{\pm}(t,y)=y+\int_0^tZ_{\pm}(\tau,\psi_{\pm}(\tau,y))d\tau =y \pm t B_0+\int_0^tz_{\pm}(\tau,\psi_{\pm}(\tau,y))d\tau.
\end{equation*}

We only give the proof for $\psi_+$.  According to \eqref{flow in integration form}, we have
\begin{equation}\label{mu9}
\frac{\partial \psi_+(t,y)}{\partial y}=\I+\int_0^t(\nabla z_+)(\tau,\psi_+(\tau,y))\frac{\partial\psi_+(\tau,y)}{\partial y}d\tau.
\end{equation}
Therefore, we have
\begin{equation*}
|\f{\p\psi_+(t,y)}{\p y}-\I|\leq\int_0^t|(\na z_+)(\tau,\psi_+(\tau,y))||\f{\p\psi_+(\tau,y)}{\p y}-\I|d\tau +\int_0^t|(\na z_+)(\tau,\psi_+(\tau,y))|d\tau.
\end{equation*}
It suffices to bound the righthand side of the above equation, which is denoted by $G(t,y)$. We deduce from \eqref{mu9} that
\begin{align*}
\frac{d}{dt}G(t,y)&=|(\na z_+)(t,\psi_+(t,y))||\f{\p\psi_+(t,y)}{\p y}-I|+|(\na z_+)(t,\psi_+(t,y))|\\
&\leq|(\na z_+)(t,\psi_+(t,y))|G(t,y)+|(\na z_+)(t,\psi_+(t,y))|.
\end{align*}
By virtue of Gronwall's inequality, we obtain
\begin{align*}
G(t,y)&\leq\int_0^t\exp\Bigl(\int_s^t|(\na z_+)(\tau,\psi_+(\tau,y))|d\tau\Bigr)|(\na z_+)(s,\psi_+(s,y))|ds\\
&\leq\exp\Bigl(\int_0^t|(\na z_+)(\tau,\psi_+(\tau,y))|d\tau\Bigr)\underbrace{\int_0^t|(\na z_+)(\tau,\psi_+(\tau,y))|d\tau}_{A}.
\end{align*}
Then, to get the bound of $  G(t,y)$, we have to bound the integration $A$. Firstly,
by \eqref{Sobolev} and ansatz \eqref{Bootstrap on energy},  we have
\begin{align*}
 |\na z_+(\tau,\psi_+(\tau,y))|\lesssim\f{\varepsilon}{\wm(\log\wm)^2}
 \lesssim\f{\varepsilon}{(R^2+|u_-|^2)^{\f12}(\log(R^2+|u_-|^2)^{\f12})^2\big|_{x=\psi_+(\tau,y)}}.
 \end{align*}
Then we will switch the variable $\tau$ to $u_-$ in the integration $A$ . To do this, we have to calculate the Jacobian as follows
\beno
\f{d}{d\tau}u_-(\tau,\psi_+(\tau,y))=(\p_t u_-)(\tau,\psi_+(\tau,y))+\p_t\psi_+(\tau,y)\cdot(\na u_-)(\tau,\psi_+(\tau,y))\\
\eeno
Notice that $L_-u_-=0$ and $\p_t\psi_+(\tau,y)=Z_+(\tau,\psi_+(\tau,y))$, we have
\beno\begin{aligned}
\f{d}{d\tau}u_-(\tau,\psi_+(\tau,y))&=(Z_+-Z_-)(\tau,\psi_+(\tau,y))\cdot(\na u_-)(\tau,\psi_+(\tau,y))\\
&=2(\p_3 u_-)\big|_{x=\psi_+(\tau,y)}+\bigl((z_+-z_-)\cdot(\na u_-)\bigr)\big|_{x=\psi_+(\tau,y)}\\
&\stackrel{\eqref{Bootstrap on geometry},\eqref{Bootstrap on amplitude}}{\geq}1-4\sqrt{2C_0}\varepsilon.
\end{aligned}\eeno
By taking $\varepsilon$ small, we obtain
\beq\label{change variable}
\f{d}{d\tau}u_-(\tau,\psi_+(\tau,y))\geq\f{1}{2}.
\eeq
With the above inequality,
by changing variables, we have
\begin{align*}
A&\lesssim\int_0^t\f{\varepsilon}{(R^2+|u_-|^2)^{\f12}(\log(R^2+|u_-|^2)^{\f12})^2\big|_{x=\psi_+(t,y)}}d\tau\\
&\lesssim\int_0^\infty \f{\varepsilon}{(R^2+|u_-|^2)^{\f12}(\log(R^2+|u_-|^2)^{\f12})^2}du_-\sup_{\tau}\f{1}{\f{d}{d\tau}u_-(\tau,\psi_+(\tau,y))}\\
&\stackrel{\eqref{change variable}}{\lesssim}\varepsilon.
\end{align*}
This implies
\beq\label{part 1}
\Big|\frac{\partial \psi_{\pm}(t,y)}{\partial y}-\I\Big|\leq e^AA\leq C_0' \varepsilon.\eeq
This improves the first part of ansatz \eqref{Bootstrap on the flow}.

To improve the second part, applying $\p_k$ (with $k=1,2,3$) to \eqref{mu9}, one gets by the chain rule that
\beno\begin{aligned}
\p_k\Bigl(\f{\p\psi_+(t,y)}{\p x}\Bigr)=&\int_0^t(\na z_+)(\tau,\psi_+(\tau,y))\p_k\Bigl(\f{\p\psi_+(\tau,y)}{\p y}\Bigr)d\tau\\
&+\int_0^t\p_k\Bigl((\na z_+)(\tau,\psi_+(\tau,y))\Bigl)\Bigl(\f{\p\psi_+(\tau,y)}{\p y}\Bigr)d\tau,
\end{aligned}\eeno
from which and Gronwall's inequality, we obtain that
\beno\begin{aligned}
|\p_k\Bigl(\f{\p\psi_+(t,y)}{\p y}\Bigr)|\leq\exp\Bigl(\int_0^t|(\na z_+)(\tau,\psi_+(\tau,y))|d\tau\Bigr)\int_0^t|(\p^2 z_+)(\tau,\psi_+(\tau,y))||\f{\p\psi_+(\tau,y)}{\p y}|^2d\tau.
\end{aligned}\eeno
Thanks to \eqref{part 1}, we then obtain that
\beno
|\p_k\Bigl(\f{\p\psi_+(t,y)}{\p y}\Bigr)|\leq 2 \exp\Bigl(\underbrace{\int_0^t|(\na z_+)(\tau,\psi_+(\tau,y))|d\tau}_{A}\Bigr)\underbrace{\int_0^t|(\p^2 z_+)(\tau,\psi_+(\tau,y))|d\tau}_{B}.
\eeno
The previous proof shows that $A\lesssim\varepsilon$. By virtue of \eqref{Sobolev},  we also have
\begin{align*}
|\na^2 z_+(\tau,\psi_+(\tau,y))|\lesssim\f{\varepsilon}{\wm(\log\wm)^2}
 \lesssim\f{\varepsilon}{(R^2+|u_-|^2)^{\f12}(\log(R^2+|u_-|^2)^{\f12})^2\big|_{x=\psi_+(\tau,y)}}.
 \end{align*}
The same argument as $A$, we obtain that
\beno
B\lesssim\varepsilon.
\eeno
Therefore, by taking $\varepsilon$ sufficiently small, we could obtain that
\beq\label{part 2}
\Big|\na_y\frac{\partial \psi_{\pm}(t,y)}{\partial y}\Big|\leq 2e^AB\leq C_0'' \varepsilon.
\eeq
This improves the second part of ansatz \eqref{Bootstrap on the flow}. Notice that one may take $C_0\geq\max\{C_0',C_0''\}$ by taking $\varepsilon$ small enough.

{\bf Step 3. } {\it Improving ansatz \eqref{Bootstrap on amplitude}.}  The Sobolev inequality \eqref{Sobolev} shows that
\beno
\|\zpm\|_{L^\infty}\stackrel{\eqref{Sobolev}}{\leq}\f{C}{(\log R)^2}\bigl(E_\pm+E_\pm^0+E_\pm^1\bigr)^{\f12}\stackrel{\eqref{Bootstrap on energy}}{\leq}
\f{C\sqrt{6C_1}}{(\log R)^2}\varepsilon\stackrel{\varepsilon\ll1}{\leq}\f{1}{4}.
\eeno
This improves ansatz \eqref{Bootstrap on amplitude}.

{\bf Step 4.} {\it Existence and uniqueness.} The local existence for smooth data is well-known. The global existence and uniqueness of the solution is a direct consequence of the \emph{a priori} energy estimate \eqref{global energy estimate for MHDmu}.

\medskip

The above four steps complete the proof of Theorem \ref{global existence for MHDmu}.

\section{Proof of main theorems}

\subsection{Proof of Theorem \ref{global existence for ideal MHD with small amplitude}}
The proof is indeed very similar to and much easier than that of Theorem \ref{global existence for MHDmu}: first of all, we do not have diffusion terms; secondly, we can deal with the first order energy estimates and higher order energy estimates in the same way. The treatment of the pressure estimates will be different due to the choice of different weight functions. We only sketch the necessary modifications.

We fix a small number $\delta>0$ and let $\omega=1+\delta$. Let $\Up = (R^2+|\up|^2)^\frac{1}{2}$ and $\Um = (R^2+|\um|^2)^\frac{1}{2}$. We define the energy and flux norms as follows:
\begin{align*}
&E_{\mp}^{(\alpha)}(t) = \int_{\Sigma_t} \Upm^{2\omega} |\nabla z_{\mp}^{(\alpha)}|^2 dx,  \ \ F_{\mp}^{(\alpha)}(j_{\mp})=\int_{C_{u_{\mp}}^{\mp}}\Upm^{2\omega} |j_{\mp}^{(\alpha)}|^2d\sigma_{\mp}, \ \ |\alpha|\geq 0.
\end{align*}
The lowest order energy and flux are defined as
\begin{equation*}
E_{\mp}(t) = \int_{\Sigma_t} \Upm^{2\omega} | z_{\mp}|^2 dx, \ \ \ F_{\mp}(z_{\mp})=\int_{C_{u_{\mp}}^{\mp}}\Upm^{2\omega}|z_{\mp}|^2d\sigma_{\mp}.
\end{equation*}
The total energy norms and total flux norms as defined as before, e.g.,
\begin{align*}
&E_{\mp} = \sup_{0\leq t\leq t^*} E_{\mp}(t),\ \ E_{\mp}^k = \sup_{0\leq t\leq t^*} \sum_{|\alpha|=k}E_{\mp}^{(\alpha)}(t).
\end{align*}

\medskip

The three sets of ansatz for continuity method remain the same. Since the energy and flux norms are stronger than the original norms, all the estimates in the Subsection \ref{subsection Preliminary estimates} still hold. We can improve the Sobolev inequalities to
\begin{equation}\label{Sobolev new}
\begin{split}
|z_{\mp}|&\lesssim \frac{1}{\Upm^{\omega}}\big(E_{\mp} + E^0_{\mp}+E^1_{\mp}\big)^\frac{1}{2},\ \ |\na z_{\mp}^{(\alpha)}|\lesssim   \frac{1}{\Upm^{\omega}} \big(E_{\mp}^{k}+E_{\mp}^{k+1}+E_{\mp}^{k+2}\big)^\frac{1}{2}\ \ \text{for}\ \ |\alpha|=k.
\end{split}
\end{equation}

\subsubsection{A better control on the underlying geometry}
The essential improvement in the ideal case is that we can obtain a much more precise picture for the characteristic hypersurfaces.

We recall and repeat some definition and argument from last section. The defining equation for the flow $\psi_{\pm}(t,x)$ generated by $Z_{\pm}$ is $\frac{d}{dt}\psi_{\pm}(t,x)=Z_{\pm}(t,\psi_{\pm}(t,x))$. where $x\in \mathbb{R}^3$. Since $\zp = \Zp \mp  B_0$, we obtain $\psi_{\pm}(t,x)=x \pm t B_0+\int_0^tz_{\pm}(\tau,\psi_{\pm}(\tau,x))d\tau$. This is exactly \eqref{flow in integration form}.

Let $\frac{\partial\psi_{\pm}(t,x)}{\partial x}$ be the differential of $\psi(t,x)$. Repeat the proof for \eqref{part 1} and \eqref{part 2}, we obtain for $k=0,1$
\begin{equation}\label{estimates on the flow}
|\partial^k \Bigl(\frac{\partial \psi_{\pm}(t,x)}{\partial x}-\I\Bigr)|\lesssim \varepsilon.
\end{equation}
Similarly, it follows that
\begin{equation*}
|\na u_{\pm}|\leq 2, \ \  |\nabla^2u_{\pm}|\lesssim \varepsilon.
\end{equation*}

The key improvement can be stated in the following lemma:
\begin{lemma}
  For sufficiently small $\varepsilon$, we have
\begin{equation}\label{approximation of upm}
|u_\pm(t,x)-(x_3\mp t)|\leq\f{C_0 \varepsilon}{\delta R^{\delta}}.
\end{equation}
In particular, we can measure the separation of $u_\pm$:
\begin{equation*}
\Big|\big(u_+-u_-\big)-2t\Big|\lesssim \varepsilon.
\end{equation*}
\end{lemma}
\begin{proof}  By the definition of $\psi_\pm$, we have
$$\psi_{\pm}^3(t,y)=y_3\pm t+\int_0^tz_{\pm}^3(\tau,\psi_{\pm}(\tau,x))d\tau,$$
where $\psi_{\pm}^3$ and $z_{\pm}^3$ are the $x_3$-coordinate component of $\psi_{\pm}$ and $z_{\pm}$ respectively.
Since $u_\pm(t,\psi_\pm(t,y))=y_3$, we have
$$
u_\pm(t,\psi_\pm(t,y))=\psi_{\pm}^3(t,y)\mp t-\int_0^tz_{\pm}^3(\tau,\psi_{\pm}(\tau,x))d\tau.$$
We can repeat the proof \eqref{part 1} to derive
$$
\int_0^t|z_\pm(\tau,\psi_\pm(\tau,x))|d\tau\leq\f{C_0\varepsilon}{\delta R^{\delta}}.$$
This completes the proof of the lemma.
\end{proof}
\begin{remark}
In the viscous case, the decay of $z_\pm$ in the ansatz is roughly $\big(\log(1+|u_\pm|)\big)^{-2}$; in the current situation, the decay of $z_\pm$ in the ansatz is roughly $(1+|u_\pm|)^{-(1+\delta)}$ which is \emph{integrable}. The faster decay in the ideal case allows us to integrate the equation $z_\pm$. This is why we can control $u_\pm$ in a great precision.
\end{remark}
As a corollary, we can measure the separation of $z_\pm$ in terms of decay in $t$:
\begin{lemma}[\bf{Separation Estimates}]\label{separation lemma}
For all $\alpha$ and $\beta$ with $|\alpha|,|\beta|\leq 2$, we have
\begin{equation}\label{Product estimates}
|\zp^{(\alpha)}(t,x) \zm^{(\beta)}(t,x)|\lesssim \frac{\varepsilon^2}{(1+t)^{\omega}}.
\end{equation}
\end{lemma}
\begin{proof}The bootstrap assumptions and the previous lemma immediately imply
\begin{equation*}
|\zp^{(\alpha)}(t,x) \zm^{(\beta)}(t,x)| \leq \frac{4\varepsilon^2}{( 1+|x_3 + t|)^{\omega}(1+|x_3- t|)^{\omega}}.
\end{equation*}
Since for all $x_3$, at least one of the inequalities $|x_3 + t|\geq \frac{t}{2}$ and $|x_3 - t|\geq \frac{t}{2}$ holds. The above inequality yields the lemma.
\end{proof}
On the contrary, for the self-intersections such as $\zp^{(\alpha)}(t,x) \zp^{(\beta)}(t,x)$, we can not obtain a decay factor in $t$. Since near the center of $\zp$, the wave is approximately of size $\varepsilon$. The best pointwise estimate one can hope is
\begin{equation*}
|\zp^{(\alpha)}(t,x) \zp^{(\beta)}(t,x)|\lesssim \varepsilon^2.
\end{equation*}

\subsubsection{The \emph{a prori} energy estimates in the ideal case}

We now prove the energy estimates on the lowest order terms. This part corresponds to the estimates derived in Subsection \ref{subsection Energy estimates on the lowest order terms}. We first prove the following pressure estimates: for all $t\in [0,t^*]$, we have
\begin{equation}\label{new estimates on pressure ideal case}
 \Big|\int_0^t\int_{\Sigma_\tau}\Ump^{2\omega} |z_\pm||\na p|dxd\tau\Big| \lesssim \varepsilon^3.
\end{equation}

 Firstly, by H\"older inequality, we have
\beno
 \Big|\int_0^t\int_{\Sigma_\tau}\Ump^{2\omega} |z_\pm||\na p|dxd\tau\Big|
 \lesssim\Bigl(\int_0^t\int_{\Sigma_\tau}\f{\Ump^{2\omega}}{\Upm^{\omega}} |z_\pm|^2dxd\tau\Bigr)^{\f12}
 \Bigl(\int_0^t\int_{\Sigma_\tau}\Ump^{2\omega}\Upm^{\omega}|\na p|^2dxd\tau\Bigr)^{\f12}.
\eeno
Changing the variables from $(x_1,x_2,x_3,t)$ to $(x_1,x_2,u_+,u_-)$ (see in subsection 2.4) and noticing that the denominator $\Upm^{\om}$ is integral in $\Upm$,  we have
\beno
\int_0^t\int_{\Sigma_\tau}\f{\Ump^{2\omega}}{\Upm^{\omega}} |z_\pm|^2dxd\tau\lesssim F(z_\pm)\lesssim\varepsilon^2.
\eeno
Thus, to prove \eqref{new estimates on pressure ideal case}, we only need to verify the following inequality:
\beq\label{integrable for p}
\int_0^t\int_{\Sigma_\tau}\Ump^{2\omega}\Upm^{\omega}|\na p|^2dxd\tau\lesssim\varepsilon^4.
\eeq

To derive the above estimates, we first decompose $\nabla p$ as
\begin{equation}\label{decomposition of nabla p new}
\begin{split}
\na p(t,x)=&\underbrace{-\f{1}{4\pi}\int_{\R^3}\na\f{1}{|x-y|}\theta(|x-y|)(\pa_iz_-^j\pa_jz_+^i)(t,y)dy}_{A_1(t,x)}\\
&\underbrace{-\f{1}{4\pi}\int_{\R^3}\pa_i\pa_j\Bigl(\na\f{1}{|x-y|}\bigl(1-\theta(|x-y|)\bigr)\Bigr)(z_-^iz_+^j)(t,y)dy}_{A_2(t,x)}.
\end{split}
\end{equation}
where the smooth cut-off function $\theta(r)$ is chosen in such way that $\theta(r)\equiv 1$ for $r\leq 1$ and $\theta(r)\equiv 0$ for $r\geq 2$. Therefore, it suffices to bound the following two terms:
$$
\underbrace{\int_0^t\int_{\Sigma_\tau}\Um^{2\om}\Up^{\om}|A_1|^2dxd\tau}_{I_{1}}
+\underbrace{\int_0^t\int_{\Sigma_\tau}\Um^{2\om}\Up^{\om}|A_2|^2dxd\tau}_{I_{2}}.
$$

By definition, we have
$$\Um^{\om}\Up^{\f{\omega}{2}}|A_1|
\leq\int_{|x-y|\leq2}\f{|\na z_-(\tau,y)||\na z_+(\tau,y)|\Um^{\om}(\tau,x)\Up^{\f{\omega}{2}}(\tau,x)}{|x-y|^2}dy.$$
For $|x-y|\leq 2$, it is straightforward to check that $\Upm(\tau,x) \lesssim \Upm(\tau,y)$. Hence,
\begin{equation}\label{mu1111new}
\begin{split}
\Um^{\om}\Up^{\f{\omega}{2}}|A_1|
&\lesssim \int_{|x-y|\leq2}\f{|\na z_-(\tau,y)||\na z_+(\tau,y)|\Um^{\om}(\tau,y)\Up^{\f{\omega}{2}}(\tau,y)}{|x-y|^2}dy\\
&\leq \|\Up^{\om}\na z_-\|_{L^\infty}\int_{|x-y|\leq2}\f{\Um^{{\om}}(\tau,y)|\na z_+(\tau,y)|}{\Up^{\f{\om}{2}}(\tau,y)|x-y|^2}dy\\
&\stackrel{\eqref{Sobolev new}}{\leq} \varepsilon\int_{|x-y|\leq2}\frac{1}{|x-y|^2}\f{\Um^{\om}(\tau,y)}{\Up^{\f{\om}{2}}(\tau,y)} |\na z_+(\tau,y)|dy.
\end{split}
\end{equation}
By Young's inequality, similar to \eqref{mu16}, we obtain
\begin{equation}\label{mu16 new}
\begin{split}
\|\Um^{\om}\Up^{\f{\omega}{2}}A_1\|_{L^2(\Sigma_\tau)}&\lesssim \varepsilon \|\frac{\Um^{\om}}{\Up^{\f{\om}{2}}}\na z_+\|_{L^2(\Sigma_\tau)}.
\end{split}
\end{equation}
Therefore, by the virtue of div-curl lemma (Lemma \ref{div-curl lemma}), we can bound $I_{1}$ as
\begin{equation}\label{bound on I_1 new}
I_{1}\lesssim \varepsilon^2\int_0^t \int_{\Sigma_\tau}\frac{\Um^{2\om}}{\Up^{\om}}|\na z_+|^2 dx d\tau\lesssim \varepsilon^4.
\end{equation}
To bound $I_{2}$, we split $A_{2}(t,x)$ as
\begin{equation}\label{mu11112new}
\begin{split}
|A_2(t,x)|&\lesssim \underbrace{\int_{\R^3}\f{1-\theta(|x-y|)}{|x-y|^4}|(z_-^iz_+^j)(t,y)|dy}_{A_{21}(t,x)}+
\underbrace{\int_{\R^3}\bigl(\f{\theta''(|x-y|)}{|x-y|^2}+\f{\theta'(|x-y|)}{|x-y|^3}\bigr)|(z_-^iz_+^j)(t,y)|dy}_{A_{22}(t,x)}.
\end{split}
\end{equation}
We then split $I_2$ as
\begin{align*}
I_{2}\leq \underbrace{\int_0^t\int_{\Sigma_\tau}\Um^{2\om}\Up^{\om}|A_{21}|^2dxd\tau}_{I_{21}}
+\underbrace{\int_0^t\int_{\Sigma_\tau}\Um^{2\om}\Up^{\om}|A_{22}|^2dxd\tau}_{I_{22}}.
\end{align*}
In view of the property of the cut-off function $\theta(r)$, we can bound $A_{22}$ as
$$ A_{22}(t,x) \lesssim \int_{|x-y|\leq 2}\f{1}{|x-y|^2}|z_-^i(t,y)||z_+^j(t,y)|dy.
$$
Therefore, $I_{22}$ can be bounded similarly as $I_1$. This leads to
\begin{equation}\label{bound on I_22 new}
I_{22} \lesssim \varepsilon^4.
\end{equation}

We turn to $I_{21}$. Since $|u_\pm(\tau,x)|\leq|u_\pm(\tau,y)|+2|x-y|$, we conclude that $$
\Um^{\om}(\tau,x)\lesssim \Um^{\om}(\tau,y)+|x-y|^{\om},$$
$$\text{and}\ \ \bigl(\Um^{\om}\Up^{\f{\om}{2}}\bigr)(\tau,x)\lesssim\bigl(\Um^{\om}\Up^{\f{\om}{2}}\bigr)(\tau,y)+|x-y|^{\f{3\om}{2}}.$$
Therefore, we can bound $I_{21}$ by
\begin{align*}
I_{21}&\lesssim \underbrace{\int_0^t \int_{\Sigma_{\tau}}\Big(\int_{|x-y|\geq 1}\f{|z_-(\tau,y)||z_+(\tau,y)|}{|x-y|^{4-\f{3\om}{2}}}dy\Big)^2dxd\tau}_{I_{211}}\\
&\ +\underbrace{\int_0^t \int_{\Sigma_\tau}\Big(\underbrace{\int_{|x-y|\geq 1}\f{|z_-(\tau,y)||z_+(\tau,y)|}{|x-y|^4}\Um^{\om}\Up^{\f{\om}{2}}(\tau,y)dy}_{A_3(\tau,x)}\Big)^2dxd\tau}_{I_{212}}.
\end{align*}
We use H\"{o}lder and Young inequalities to bound $I_{211}$:
\beq\label{mu76}\begin{aligned}
I_{211}&\lesssim \int_0^t\|\big(\f{1}{|x|^{4-\f{3\om}{2}}}\chi_{|x|\geq 1}\big)*(z_-z_+)\|_{L^2(\Sigma_\tau)}^2d\tau\\
&\lesssim \int_0^t\|\f{1}{|x|^{4-\f{3\om}{2}}}\|_{L^2(|x|\geq1)}^2\|z_-z_+\|_{L^1(\Sigma_\tau)}^2d\tau\\
&\stackrel{\om\in(1,\f{5}{3})}{\lesssim}\int_0^t\|z_-z_+\|_{L^1(\Sigma_\tau)}^2d\tau.
\end{aligned}\eeq
Since $\f{1}{\Um^{\om}\Up^{\om}}\lesssim \f{1}{(R+\tau)^{\om}}$, we have
\begin{equation}\label{bound on I_211 new}\begin{aligned}
I_{211}&\lesssim\int_0^t\|\f{1}{\Up^{\om}\Um^{\om}}\Up^{\om}z_-   \Um^{\om}z_+\|_{L^1(\Sigma_\tau)}^2d\tau\\
&\lesssim\int_0^t\f{1}{(R+\tau)^{2\om}}\|\Um^{\om}z_+\|_{L^2(\Sigma_\tau)}^2\|\Up^{\om}z_-\|_{L^2(\Sigma_\tau)}^2d\tau\\
&\lesssim \varepsilon^4\int_0^t\f{1}{(R+\tau)^{2\om}}d\tau\lesssim\varepsilon^4.
\end{aligned}\end{equation}

For $I_{212}$, we first bound $A_{3}(\tau,x)$ as follows:
\begin{align*}
\|A_{3}\|_{L^2(\Sigma_\tau)}&\lesssim \|\f{\chi_{|x|\geq 1}}{|x|^4}*\big(\f{\Um^{\om}}{\Up^{\f{\om}{2}}}|z_+|\cdot\Up^{\om}|z_-|\big)\|_{L^2(\Sigma_\tau)}\\
&\lesssim\|\f{\chi_{|x|\geq 1}}{|x|^4}\|_{L^1(\Sigma_\tau)}\|\f{\Um^{\om}}{\Up^{\f{\om}{2}}}|z_+|\cdot\Up^{\om}|z_-|\|_{L^2(\Sigma_\tau)}\\
&\lesssim\|\Up^{\om}z_-\|_{L^\infty(\Sigma_\tau)}\|\f{\Um^{\om}}{\Up^{\f{\om}{2}}}z_+\|_{L^2(\Sigma_\tau)}
\lesssim\varepsilon\|\f{\Um^{\om}}{\Up^{\f{\om}{2}}}z_+\|_{L^2(\Sigma_\tau)}.
\end{align*}
This implies
\begin{align*}
I_{212}&\lesssim
\varepsilon^2\int_0^t\int_{\R^3}\f{\Um^{2\om}}{\Up^{\om}}|z_+|^2dxd\tau\lesssim\varepsilon^4.
\end{align*}

Combined with \eqref{bound on I_1 new}, \eqref{bound on I_22 new} and \eqref{bound on I_211 new}, we obtain  \eqref{integrable for p}. Then we finally have \eqref{new estimates on pressure ideal case}.

\medskip

We then turn to the first and higher order terms. The proof goes exactly as in Subsection \ref{subsection Energy estimates on higher order terms}. Indeed, in Subsection \ref{subsection Energy estimates on higher order terms}, by virtue of the flux, the only essential use of the weight is the fact that $\frac{1}{\wpm \big(\log(\wpm)\big)^2}$ is integrable in $u_{\pm}$. In the current situation, the factor is replaced by $\frac{1}{\Upm^{\omega}}$ which is still integrable.

\medskip

According to the above discussion, we can control all the nonlinear terms in the \emph{a priori} energy estimates. This completes the proof of Theorem \ref{global existence for ideal MHD with small amplitude}.

\subsection{Proof of Theorem \ref{theorem scattering theory for ideal Alfven waves}}
We divide the proof into three steps.

{\bf Step 1.} {\it  Explicit formulas for scattering fields.}

We only prove this for $z_+$. We integrate $\partial_t  \zp +Z_{-} \cdot \nabla \zp  = -\nabla p$ along $L_-$:   for any given point $q=(y_1,y_2,y_3,0)$ on the initial hypersurface $\Sigma_0$, along the $L_-$ direction, the characteristic line emanated from this point hits $\Sigma_t$ at the point  $(y_1,y_2, u_-,t)$. We then integrate the equation over the characteristic line segment between $(y_1,y_2, u_-,0)$ and $(y_1,y_2, u_-,t)$. Therefore,
\begin{equation}\label{expression for zp}
z_+(y_1,y_2, u_-,t) = z_+(y_1,y_2,u_-,0) -\int_{0}^t (\nabla p) (y_1,y_2, u_-,\tau) d\tau.
\end{equation}

In order to understand \eqref{expression for zp}, we now derive \eqref{expression for zp} by characteristics method. Indeed, we first introduce the coordinate transformations on $\R^3\times[0,\infty)$ as follows:
\beno\begin{aligned}
\Phi_-:\ & \R^3\times[0,\infty) \rightarrow \R^3\times[0,\infty)\\
&(x_1,x_2,x_3,t)\mapsto(x_1,x_2,u_-,t)=(x_1,x_2,u_-(x_1,x_2,x_3,t),t),
\end{aligned}\eeno
where $u_-$ is defined by \eqref{definition for um} with $u_-(x,0)=y_3$ at point $q$.
Thanks to Theorem \ref{global existence for ideal MHD with small amplitude}, we have $\det(d\Phi_-)=\p_3u_-=1+O(\varepsilon)$ (see also \eqref{Jacobian}). Denoting by $\wt{\zpm}(x_1,x_2,u_-,t)\eqdefa\zpm|_{(x,t)=\Phi_-^{-1}(x_1,x_2,u_-,t)}$, we deduce that $\wt{z_+}$ satisfies
\beno
\p_t\wt{z_+}+\wt{z_-}^h\cdot\na_h\wt{z_+}=-(\na_xp)|_{(x,t)=\Phi_-^{-1}(x_1,x_2,u_-,t)},
\eeno
where $\Phi_-^{-1}$ is the inverse of the mapping $\Phi_-$, and  $\wt{z_-}^h=(\wt{z_-}^1,\wt{z_-}^2)$, $\na_h=(\p_1,\p_2)$. Notice that on $C_{u_-}^-$, $u_-$ is a constant. Then for fixed $u_-$, we define the flow $\phi^-_{(u_-)}(y_1, y_2,t)$ (the mapping from $S_{0,u_-}$ to $S_{t,u_-}$) associated to $\wt{z_-}^h$ as follows:
\beno
\f{d}{dt}\phi^-_{(u_-)}(y_1, y_2,t)=\wt{z_-}^h(x_1,x_2,u_-,t)|_{(x_1,x_2)=\phi^-_{(u_-)}(y_1, y_2,t)},\quad
\phi^-_{(u_-)}(y_1, y_2,0)=(y_1,y_2).
\eeno
Thanks to Theorem \ref{global existence for ideal MHD with small amplitude}, the Jacobian $d\phi^-_{(u_-)}$ satisfies
$\det(d\phi^-_{(u_-)})=1+O(\varepsilon)$.
 Then denoting by $\overline{z_+}(y_1, y_2,u_-,t)=\wt{z_+}(x_1,x_2,u_-,t)|_{(x_1,x_2)=\phi^-_{(u_-)}(y_1, y_2,t)}$, we have
\beno
\f{d\overline{z_+}}{dt}=-(\na_xp)|_{(x,t)=\Phi_-^{-1}(\phi^-_{(u_-)}(y_1, y_2,t),u_-,t)},
\eeno
which implies that
\beno
\overline{z_+}(y_1, y_2,u_-,t)=\overline{z_+}(y_1, y_2,u_-,0)-\int_0^t(\na_xp)|_{(x,\tau)=\Phi_-^{-1}(\phi^-_{(u_-)}(y_1, y_2,\tau),u_-,\tau)}d\tau.
\eeno
Notice that $$\overline{z_+}(y_1, y_2,u_-,0)=\wt{z_+}(y_1, y_2,u_-,0)=z_+(y_1,y_2,y_3,0).$$
Without confusion, we use notation $z_+(y_1,y_2, u_-,t)$ to present $\overline{z_+}(y_1, y_2, u_-,t)$ which is the expression for $z_+$ in terms of the coordinates $(y_1, y_2,u_-,t)$. So does $(\nabla p) (y_1,y_2, u_-,\tau)$. Then we obtain \eqref{expression for zp}.

Similarly, we integrate $\partial_t j_+ +Z_{-} \cdot \nabla j_+  = -\na z_-\wedge\na z_+$ to derive
\beq\label{expression for curl zp}
(\curl z_+)(y_1,y_2, u_-,t) =j_+(y_1,y_2,u_-,0) -\int_{0}^t (\na z_-\wedge\na z_+) (y_1,y_2, u_-,\tau) d\tau.
\eeq

{\bf Step 2.} {\it The scattering fields are well-defined.}

To show that $z_+^{(\text{scatter})}$ is well defined, in view of \eqref{expression for zp}, it suffices to prove that $\nabla p$ is integrable (in time) along any left-traveling characteristic line.

In view of \eqref{decomposition of nabla p new}, we have $(\nabla p)(t,x) =A_1(t,x)+A_2(t,x)$, where \beno\begin{aligned}
&A_1=-\f{1}{4\pi}\int_{\R^3}\na\f{1}{|x-x'|}\theta(|x-x'|)(\pa_iz_-^j\pa_jz_+^i)(t,x')dx',\\
&A_2=-\f{1}{4\pi}
 \int_{\R^3}\pa_i\pa_j\Bigl(\na\f{1}{|x-x'|}\bigl(1-\theta(|x-x'|)\bigr)\Bigr)(z_-^iz_+^j)(t,x')dx'.
\end{aligned}\eeno

Similar to \eqref{mu1111new}, setting $\om=1+\delta$, we have
\begin{equation*}
\begin{split}
\Um^{\om}\Up^{{\omega}}|A_1|
&{\lesssim}\, \varepsilon\int_{|x-x'|\leq2}\frac{1}{|x-x'|^2}\Um^{{\om}}(t,x')|\na z_+(t,x')|dx'\\
&\lesssim \varepsilon\|\f{1}{|x|^2}\|_{L^1(|x|\leq2)}
\| \Um^{\om}\na z_+\|_{L^\infty(\Sigma_t)}\\
&\lesssim \varepsilon^2.
\end{split}
\end{equation*}
Hence,
\begin{equation}\label{mu26 new}
\big|A_1\big| \lesssim \frac{\varepsilon^2}{\Um^\omega\Up^\omega}.
\end{equation}

According to \eqref{mu11112new}, we further split $A_2(t,x)$ as $A_{21}(t,x)+A_{22}(t,x)$. Since
$$ A_{22}(t,x) \lesssim \int_{|x-x'|\leq 2}\f{1}{|x-x'|^2}|z_-^i(t,x')||z_+^j(t,x')|dx',
$$
it can be estimated in the same manner as for $A_{1}(t,x)$. Therefore, we can ignore this term.

It remains to bound $A_{21}(t,x) =\int_{\R^3}\f{1-\theta(|x-x'|)}{|x-x'|^4}(z_-^iz_+^j)(t,x')dx'$. Since $(1+t)\lesssim \Up \Um$, we have
\begin{equation*}
\begin{split}
(1+t)^{{\omega}}|A_{21}(t,x)|&\lesssim  \int_{|x-x'|\geq 1}\f{1}{|x-x'|^4}|\Up^{{\om}}(t,x') z_-^i(t,x')||\Um^{{\om}}(t,x')z_+^j(t,x')|dx'\\
&\stackrel{Young}\lesssim \|\f{1}{|x|^4}\|_{L^1(|x|\geq 1)}
\| \Um^{\om}) z_+\|_{L^\infty(\Sigma_t)}\| \Up^{\om}) z_-\|_{L^\infty(\Sigma_t)}.
\end{split}
\end{equation*}
Therefore,
\begin{equation*}
\big|A_{21}\big| \lesssim \frac{\varepsilon^2}{(1+t)^{\omega}}.
\end{equation*}
Combined with \eqref{mu26 new}, we obtain that
$$\big|(\nabla p )(y_1,y_2, u_-,\tau)\big| \leq \frac{\varepsilon^2}{(1+\tau)^{\omega}}.$$
This implies that $\lim_{t\rightarrow \infty}\int_{0}^t (\nabla p) (y_1,y_2, u_-,\tau) d\tau$ is well-defined.

To show that $(\curl z)_+^{(\text{scatter})}$ is well defined, in view of \eqref{expression for curl zp}, it suffices to bound $\na z_-\wedge\na z_+$. According to Lemma \ref{separation lemma}, we have
\beno
|\na z_-\wedge\na z_+|\lesssim\f{\varepsilon^2}{(1+\tau)^{\om}},
\eeno
which is integrable in $t$. Therefore $(\curl z)_+^{(\text{scatter})}$ is well defined.

Similarly, the higher derivatives of the scattering fields are well-defined and we omit the routine details.

\medskip

{\bf Step 3.} {\it Calculate the differential of $\mathbf{S}:H^{N_*+1,\delta}(\Sigma_0)\times H^{N_*+1,\delta}(\Sigma_0) \rightarrow H^{0,\delta}(\mathcal{C}_-) \times  H^{0,\delta}(\mathcal{C}_+)$.}

We first clarify the relation of the measure $d\tilde{\sigma}_\pm$ on $\mathcal{C}_\pm$ and the measure $d\sigma_\pm$ on $C_{u_\pm}^\pm$. Recall that in the proof of Lemma \ref{Trace Estimates}, $d\sigma_+$ on $C_{u_+}^+$ was calculated as follows:
\beno
d\sigma_+=\sqrt{1+|Z_+\cdot\na u_+|^2+|\na_{x_h}u_+|^2}dx_1dx_2dt\stackrel{\eqref{Bootstrap on geometry}}{=}(\sqrt2+O(\varepsilon))d{x_1}d{x_2}dt.
\eeno

Similar to the definitions of $\Phi_-$ and $\phi^-_{(u_-)}$, we introduce the coordinates transformation $\Phi_+:\ (x_1,x_2,x_3,t)\mapsto(x_1,x_2,u_+,t)=(x_1,x_2,u_+(x_1,x_2,x_3,t),t)$  (the mapping from $\R^3\times[0,\infty)$ to $\R^3\times[0,\infty)$)
and the flow $\phi^+_{(u_+)}(y_1, y_2,t)$ (the mapping from $S_{0,u_+}$ to $S_{t,u_+}$) which is generated by $z_+^h(x_1,x_2,u_+,t)$ for fixed $u_+$.
Since $u_+$ is a constant on $C_{u_+}^+$, then by the fact that $\det(d\phi^+_{(u_+)})=1+O(\varepsilon)$, we have
\beno
d\sigma_+=(\sqrt2+O(\varepsilon))d{y_1}d{y_2}dt.
\eeno

Observe that for fixed $u_+$,
\ben\label{Jacobi2}\begin{aligned}
&\f{d}{dt}u_-(y_1, y_2,u_+,t)
=\f{d}{dt}u_-(\Phi_+^{-1}(\phi^+_{(u_+)}(y_1, y_2,t),u_+,t))\\
&=(\na_{x,t}u_-)(y_1, y_2,u_+,t)\cdot\f{d}{dt}\Phi_+^{-1}(\phi^+_{(u_+)}(y_1,y_2,t),u_+,t)\\
&=\Bigl(\p_tu_-+Z_+\cdot\na u_-\Bigr)(y_1,y_2,u_+,t)\\
&\stackrel{L_-u_-=0}{=}\Bigl(\bigl(Z_+-Z_-\bigr)\cdot\na u_-\Bigr)(y_1, y_2,u_+,t)\stackrel{\eqref{Bootstrap on geometry}}{=} 2+O(\varepsilon).
\end{aligned}\een
Since $u_+$ is a constant on $C_{u_+}^+$, by changing the variable $t$ to $u_-$ via $u_-=u_-(y_1, y_2,u_+,t)$, we obtain
\beno
d\sigma_+=(2\sqrt2+O(\varepsilon))d{y_1}d{y_2}du_-.
\eeno

Finally, to compare the measure $d\tilde{\sigma}_+$ (on $\mathcal{C}_+$) with $d\sigma_+$ (on $C_{u_+}^+$), we use the common coordinates $(y_1,y_2,u_-)$. Since by definition we take $d\tilde{\sigma}_+ = d{y_1}d{y_2}du_-$, we finally claim that
\beno
d\tilde{\sigma}_+ \sim d\sigma_+,
\eeno
where the difference is a universal constant which will not effect any estimate thereafter.

We remark that the continuity of $\mathbf{S}$ at $0$ is an immediate consequence of the {\it a priori} estimates for the ideal MHD system. For the differential of $\mathbf{S}$, we derive weighted $L^2$-estimates for $z_+^{(\text{scatter})}-z_+^{(\text{linear})}$.

According to \eqref{expression for zp} (or \eqref{formula for scatter +}), we have
\beno
(z_+^{(\text{scatter})}-z_+^{(\text{linear})})(y_1,y_2, u_-)=-\int_{0}^\infty (\nabla p) (y_1,y_2, u_-,\tau) d\tau
\eeno
We will switch the $\tau$-variable to $u_+$ by $\tau\mapsto u_+(y_1,y_2, u_-,\tau)$  in the integral. Indeed, similar to \eqref{Jacobi2}, we have
\ben\label{Jacobi3}
\f{d}{d\tau}u_+(y_1,y_2, u_-,\tau)=-2+O(\varepsilon).
\een
Therefore,
\beno\begin{aligned}
\int_{0}^\infty|(\nabla p) (y_1,y_2, u_-,\tau)|d\tau&\lesssim\int_{\R}|(\nabla p) (y_1,y_2, u_-,u_+)|du_+\\
&\lesssim\Bigl(\int_{\R}\f{1}{\Up^\om}du_+\Bigr)^{\f12}\Bigl(\int_{\R}\Up^\om|(\nabla p) (y_1,y_2, u_-,u_+)|^2du_+\Bigr)^{\f12}\\
&\lesssim\Bigl(\int_{\R}\Up^\om|(\nabla p) (y_1,y_2, u_-,u_+)|^2du_+\Bigr)^{\f12}.
\end{aligned}\eeno

Thus, we have
\beq\label{new 1}\begin{aligned}
&\int_{\mathcal{C_+}}\Um^{2\om}|z_+^{(\text{scatter})}-z_+^{(\text{linear})}|^2d\tilde{\sigma}_+\\
&\lesssim
\int_{\mathcal{C_+}}\Um^{2\om}\Bigl(\int_{\R} |(\nabla p) (y_1,y_2, u_-,u_+)|du_+\Bigr)^2d\sigma_+\\
&\lesssim\int_{\R^3}\int_{\R}\Um^{2\om}\Up^\om|(\nabla p) (y_1,y_2, u_-,u_+)|^2du_+d{y_1}d{y_2}du_-.
\end{aligned}\eeq

We then use the coordinate $(x_1,x_2,x_3,\tau)$ instead of $(y_1,y_2,u_-,u_+)$. Since
\beno\begin{aligned}
&d{y_1}d{y_2}du_-du_+=|\f{d}{d\tau}u_+(y_1,y_2,u_-,\tau)|d{y_1}d{y_2}du_-d\tau\\
&=\det(d\phi^-_{(u_-)})^{-1}|\f{d}{d\tau}u_+(y_1,y_2,u_-,\tau)|d{x_1}d{x_2}du_-d\tau\\
&=\det(d\Phi_+)\det(d\phi^-_{(u_-)})^{-1}|\f{d}{d\tau}u_+(y_1,y_2,u_-,\tau)|d{x_1}d{x_2}d{x_3}d\tau
\end{aligned}\eeno
in view of \eqref{Jacobi3} and the facts that $\det(d\Phi_+)=1+O(\varepsilon)$ and $\det(d\phi^-_{(u_-)})=1+O(\varepsilon)$, we have
\beno
d{y_1}d{y_2}du_-du_+=(2+O(\varepsilon))d{x_1}d{x_2}d{x_3}d\tau.
\eeno
Therefore, \eqref{new 1} yields the following estimate:
\beq\label{new 2}
\int_{\mathcal{C_+}}\Um^{2\om}|z_+^{(\text{scatter})}-z_+^{(\text{linear})}|^2d\tilde{\sigma}_+
\lesssim\int_0^t\int_{\R^3}\Um^{2\om}\Up^\om|\nabla p(x,\tau)|^2dxd\tau.
\eeq
Thanks to \eqref{integrable for p}, we obtain that
\beq\label{estimate for zero}
\int_{\mathcal{C_+}}\Um^{2\om}|z_+^{(\text{scatter})}-z_+^{(\text{linear})}|^2d\tilde{\sigma}_+\lesssim\varepsilon^4.
\eeq

In other words, we obtain
\beq\label{estimate for the difference}
\|z_+^{(\text{scatter})}-z_+^{(\text{linear})}\|_{H^{0,\delta}(\mathcal{C}_+)}\lesssim\varepsilon^2.
\eeq
\medskip
The similar estimate also holds for $z_-$. Since $\|(z^{(0)}_-,z^{(0)}_+)\|_{H^{N_*+1,\delta}(\Sigma_0)\times H^{N_*+1,\delta}(\Sigma_0) } \sim \varepsilon$, for $\varepsilon \rightarrow 0$, this implies
\begin{equation*}
d \,\mathbf{S}	\big|_{\textbf{0}}=\mathbf{S}^{\text{linear}}.
\end{equation*}

\subsection{Proof of Theorem \ref{thm nonlinear stability of viscous waves}}
There are two statements in the theorem and we will prove them one by one.
\subsubsection{Proof of the first statement}
We fix $\mu$ and $T$. Let $\zz_\pm = z_\pm^\mu-z_\pm$. By \eqref{MHD in Elsasser}, we have
\begin{equation}\label{MHD difference equation}
\begin{split}
\partial_t  \zz_\pm +Z_\mp \cdot \nabla \zz_\pm  &=-\zz_\mp\cdot \nabla z^\mu_\pm -\nabla (p^\mu-p)+ {\mu \triangle z_\pm^\mu}.
\end{split}
\end{equation}
We remark that $\div \zz_\pm =0$ and $\zz_\pm\big|_{t=0}\equiv 0$. We multiply both sides of \eqref{MHD difference equation} by $\zz_\pm$ and we integrate over $\Sigma_t$. By virtue of the divergence-free property of $\zz_\pm$, this yields
\begin{equation*}
\begin{split}
&\ \ \ \frac{1}{2}\frac{d}{dt}\big(\|\zz_+\|^2_{L^2(\Sigma_t)}+\|\zz_-\|^2_{L^2(\Sigma_t)}\big)\\
&=\underbrace{-\int_{\Sigma_t}(\zz_-\cdot \nabla  z^\mu_+)\cdot\zz_+dx}_{I_1} -\int_{\Sigma_t}(\zz_+\cdot \nabla  z^\mu_-)\cdot\zz_- dx -\underbrace{\mu\int_{\Sigma_t}\nabla z_+^\mu \cdot \nabla \zz_+dx}_{I_2}- \mu\int_{\Sigma_t}\nabla z_-^\mu \cdot \nabla \zz_-dx
\end{split}
\end{equation*}
According to the $\mu$-independent \emph{a priori} estimates derived in Theorem \ref{global existence for MHDmu}, we have
\begin{align*}
|I_1|&\lesssim \|\nabla z_+^\mu\|_{L^\infty} \|\zz_+\|_{L^2(\Sigma_t)} \|\zz_-\|_{L^2(\Sigma_t)}\\
&\lesssim \varepsilon\big( \|\zz_+\|^2_{L^2(\Sigma_t)}+ \|\zz_-\|^2_{L^2(\Sigma_t)}\big),
\end{align*}
and
\begin{align*}
|I_2|&\lesssim \mu\|\nabla z_+^\mu\|_{L^2(\Sigma_t)} \|\na \zz_+\|_{L^2(\Sigma_t)}  \lesssim \mu\varepsilon^2.
\end{align*}
Therefore, we obtain
\begin{equation*}
\begin{split}
&\ \ \ \frac{d}{dt}\big(\|\zz_+\|^2_{L^2(\Sigma_t)}+\|\zz_-\|^2_{L^2(\Sigma_t)}\big)\lesssim \varepsilon \big(\|\zz_+\|^2_{L^2(\Sigma_t)}+\|\zz_-\|^2_{L^2(\Sigma_t)}\big) + \mu\varepsilon^2.
\end{split}
\end{equation*}
For all $\tau \in [0,T]$, we integrate this equation over $[0,\tau]$ and we use Gronwall's inequality to obtain
\begin{equation*}
\begin{split}
 \|\zz_+\|^2_{L^2(\Sigma_\tau)}+\|\zz_-\|^2_{L^2(\Sigma_\tau)}\lesssim \mu\varepsilon(e^{\varepsilon \tau}-1).
\end{split}
\end{equation*}
This completes the proof for the first statement.

\subsubsection{Proof of the second statement}

Since we have many coordinate systems in the proof, to make the notations simpler, we define the so-called Lagrangian forms $\wt {v_\pm}(t,y)$ of $v_\pm(t,x)$ as
\beno
\wt{ v_\pm}(t,y)=v_\pm(t,x)|_{x=\psi_\mp(t,y)},
\eeno
where $\psi_\mp(t,y)$ is the flow generated by $Z_\mp$ and $y \in \Sigma_0$. In other words, $\wt {v_\pm}(t,y)$ is the expression of the vector field $v$ in the $(t,y_1,y_2,y_3)$ coordinates.

We  divide the proof into several steps.
\medskip

{\bf Step 1. The linear and nonlinear decomposition of solutions.}

For the given solution $(z_+,z_-)$, we decompose it as $\zpm\eqdefa\zpm^{(\text{lin})}+\zpm^{(\text{non})}$ where the linear part $\zpm^{(\text{lin})}$ and nonlinear part $\zpm^{(\text{non})}$ satisfy
\begin{equation}\label{quasilinear part 0}
\begin{split}
\p_t\zpm^{(\text{lin})}+&Z_\mp\cdot\na\zpm^{(\text{lin})}-\mu\Delta\zpm^{(\text{lin})}=0,\\
\zpm^{(\text{lin})}|_{t=0}&=\zpm(0,x),
\end{split}
\end{equation}
and
\begin{equation}\label{nonlinear part 0}
\begin{split}
 \p_t\zpm^{(\text{non})}+&Z_\mp\cdot\na\zpm^{(\text{non})}-\mu\Delta\zpm^{(\text{non})}=-\na p,\\
 \dv \zpm^{(\text{non})}&=- \dv \zpm^{(\text{lin})},\\
 \zpm^{(\text{non})}|_{t=0}&=0.
\end{split}
\end{equation}

 We shall use $E_{\pm,(\text{lin})}^{(\al)}(t)$ to denote the energies for $z_\pm^{(\text{lin})}$ while  we use $E_{\pm,(\text{non})}^{(\al)}(t)$ to denote the energies for $z_\pm^{(\text{non})}$. We shall also use $D_\pm^{(\text{lin})}$, $D_\pm^{(\text{lin}),k}$ to denote the diffusions for $z_\pm^{(\text{lin})}$ while we use $D_\pm^{(\text{non})}$, $D_\pm^{(\text{non}),k}$ to denote the diffusions for $z_\pm^{(\text{non})}$. All the above notations are defined in the same manner as that for $z_\pm$. We define the following total energy for the linear part:
\begin{align*}
\mathcal{E}^\mu_{(\text{lin}),\pm}(t)&\eqdefa\|(\log\wmp)^2\zpm^{(\text{lin})}(t,x)\|_{L^2(\mathbb{R}^3)}^2
+\sum_{|\alpha|\le N_*}E_{\pm,(\text{lin})}^{(\alpha)}(t)  +\mu\sum_{|\alpha|=N_*+1}E_{\pm,(\text{lin})}^{(\alpha)}(t),
\end{align*}
Similarly, we can define $\mathcal{E}^\mu_{(\text{non}),\pm}(t)$.

For linear system \eqref{quasilinear part 0}, we regard $Z_\pm$ as given divergence-free vectore fields, similar to \eqref{global energy estimate for MHDmu}, for all $t\geq0$, $u_\pm \in \mathbb{R}$ we have
\beq\label{estimate for quasilinear part}\begin{aligned}
& \mathcal{E}^\mu_{(\text{lin}),\pm}(t)+ F_{\pm}(\zpm^{(\text{lin})})+\sum_{k=0}^{N_*}F_{\pm}^k(\na\zpm^{(\text{lin})})  +\bigl(D_{\pm}^{(\text{lin})}+\sum_{k=0}^{N_*}D_{\mp}^{(\text{lin}),k}+\mu D_{\mp}^{(\text{lin}),N_*+1}\bigr)|_{t^*=\infty} \lesssim \mathcal{E}^\mu(0).
\end{aligned}\eeq

To derive energy estimates for \eqref{nonlinear part 0}, we first point out a modification of Lemma \ref{div-curl lemma} for general vector field $\vv v$:
 \beq\label{div-curl 1}
 \|\sqrt\lam\na\vv v\|_{L^2}^2\lesssim \|\sqrt\lam\dv\vv v\|_{L^2}^2+\|\sqrt\lam\curl\vv v\|_{L^2}^2+\|\f{|\na\lam|}{\sqrt\lam}\vv v\|_{L^2}^2.
 \eeq
Since the initial data  $\zpm^{(\text{non})}|_{t=0}$ are zero, in view of \eqref{global energy estimate for MHDmu} and \eqref{estimate for quasilinear part}, for all $t\geq0$, $u_\pm \in \mathbb{R}$, we have
\beq\label{estimate for nonlinear part}\begin{aligned}
&  \mathcal{E}^\mu_{(\text{non}),\pm}(t)+F_{\pm}(\zpm^{(\text{non})})+F_{\pm}^0(\na\zpm^{(\text{non})})+\sum_{k=1}^{N_*}F_{\pm}^k(\curl\zpm^{(\text{non})})\\
&\qquad+\bigl(D_{\pm}^{(\text{non})}+\sum_{k=0}^{N^*}D_{\mp}^{(\text{non}),k}+\mu D_{\mp}^{(\text{non}),N^*+1}\bigr)\big|_{t^*=\infty} \lesssim 	 \bigl(\mathcal{E}^\mu(0)\bigr)^{\f32}.
\end{aligned}\eeq

As a consequence, we have
\ben\label{decompqnl}  \mathcal{E}^\mu(t)\lesssim  \sum_{+,-}\big(\mathcal{E}^\mu_{(\text{lin}),\pm}(t)+\mathcal{E}^\mu_{(\text{non}),\pm}(t)\big)\lesssim \sum_{+,-}\mathcal{E}^\mu_{(\text{lin}),\pm}(t)+\bigl(\mathcal{E}^\mu(0)\bigr)^{\f32}. \een

\medskip

{\bf Step 2. Estimate on the total linear energy $\mathcal{E}^\mu_{(\text{lin}),\pm}$.}

By symmetry, it suffices to bound $\mathcal{E}^\mu_{(\text{lin}),+}$.  For simplicity, we set $ z = \zp^{(\text{lin})}$. Therefore, we have
\beq\label{quasilinear part 1}
\p_tz+Z_-\cdot\na z-\mu\Delta z=0,
\eeq
and $z|_{t=0}=z_+(0,x)$.
By taking $L^2$ inner product of \eqref{quasilinear part 1} with $z$, $(\log\wm)^2 z$ and $(\log\wm)^4 z$ respectively, we have
\begin{align*}
\f{1}{2}\f{d}{dt}\|z\|_{L^2}^2+\mu\|\na z\|_{L^2}^2&=0,\\
\f{1}{2}\f{d}{dt}\|\log\wm z\|_{L^2}^2+\mu\|\log\wm \na z\|_{L^2}^2&\leq 4\mu\|\log\wm \na z\|_{L^2}\|\f{z}{\wm}\|_{L^2},
\end{align*}
\begin{align*}
\f{1}{2}\f{d}{dt}\|(\log\wm)^2z\|_{L^2}^2+\mu\|(\log\wm)^2\na z\|_{L^2}^2&\leq 8\mu\|(\log\wm)^2\na z\|_{L^2}\|\f{\log\wm}{\wm}z\|_{L^2}.
\end{align*}
We remark that we have already proved that $\|\f{z}{\wm}\|_{L^2}\lesssim\|\na z\|_{L^2}$ and $\|\f{\log\wm}{\wm}z\|_{L^2}\lesssim\|\na z\|_{L^2}+\|\log\wm \na z\|_{L^2}^2$   by Hardy's inequality.

For higher order energy estimates, we apply $\p^\al$ with $|\al|\geq1$ to \eqref{quasilinear part 1} to derive
\beq\label{quasilinear part 2}
\p_t(\p^\al z)+Z_-\cdot\na(\p^\al z)-\mu\Delta (\p^\al z)=-[\p^\al, z_-]\cdot\na z,
\eeq
where $\p^\al z|_{t=0}=(\p^\al z_+)(0,x)$. By taking $L^2$ product with $\wm^2(\log\wm)^4 \p^\al z$,
 we obtain	
\beno\begin{aligned}
&\f{1}{2}\f{d}{dt}\|\wm(\log\wm)^2\p^\al z\|_{L^2}^2+\mu\|\wm(\log\wm)^2\na(\p^\al z)\|_{L^2}^2\\
&\leq 12\mu\|\wm(\log\wm)^2\na (\p^\al z)\|_{L^2}\|(\log\wm)^2(\p^\al z)\|_{L^2}+f_\al(t),
\end{aligned}\eeno

The nonlinear terms $f_\alpha$
is defined as follows:
\beno\begin{aligned}
f_\al(t)&=\bigg|\int_{\Sigma_t}\bigl([\p^\al, z_-]\cdot\na z\bigr)\cdot(\p^\al z)\wm^2(\log\wm)^4dx\bigg|.
\end{aligned}\eeno
 It is straightforward to see that $\int_{t\geq 0}f_\al(t)$ (for $1\leq|\al|\leq N_*+1$) can be controlled by the flux terms in \eqref{estimate for quasilinear part} while $\mu\int_{t\geq 0}f_{\al}(t)$ (for $|\al|=N_*+2$) can be bounded by the diffusion terms. Therefore, thanks to \eqref{estimate for quasilinear part} and \eqref{global energy estimate for MHDmu}, we have
\beq\label{estimates for nonlinear term 1}
\bigl(\sum_{1\leq|\al|\leq N_*+1}+\mu\sum_{|\al|=N_*+2}\bigr)\int_0^\infty f_\al(t)dt\lesssim\bigl(\mathcal{E}^\mu(0)\bigr)^{\f32},
\eeq
where we use the notation
$$\bigl(\sum_{1\leq|\al|\leq N_*+1}+\mu\sum_{|\al|=N_*+2}\bigr) F_\al\eqdefa\sum_{1\leq|\al|\leq N_*+1} F_\al+\mu\sum_{|\al|=N_*+2} F_\al.$$

Putting the above differential inequalities together, we obtain there exist constants $c$, $c_{00}$, $c_{01}$, $c_{02}$ and $\{c_\al\}_{1\leq|\al|\leq N_*+2}$ such that
\beq\label{quasilinear 1} \begin{aligned}
&\f{d}{dt}\Bigl(c_{00}\|z\|_{L^2}^2+ c_{01}\|\log\wm z\|_{L^2}^2+ c_{02}\|(\log\wm)^2z\|_{L^2}^2\\
&\quad+\sum_{1\leq|\al|\leq N_*+1} c_\al\|\wm(\log\wm)^2\p^\al z\|_{L^2}^2+\mu\sum_{|\al|=N_*+2} c_\al\|\wm(\log\wm)^2\p^\al z\|_{L^2}^2\Bigr) \\
&\quad+c\mu\Bigl(\|\na z\|_{L^2}^2+\|\log\wm \na z\|_{L^2}^2+\|(\log\wm)^2\na z\|_{L^2}^2+\!\!\!\!\!\!\!\!\!\!\sum_{1\leq|\al|\leq N_*+1}\!\!\!\!\!\!\!\!\|\wm(\log\wm)^2\na(\p^\al z)\|_{L^2}^2\Bigr)\\
&\quad+c\mu^2\sum_{|\al|=N_*+2}\|\wm(\log\wm)^2\na(\p^\al z)\|_{L^2}^2 \leq \bigl(\sum_{1\leq|\al|\leq N_*+1}+\mu\sum_{|\al|=N_*+2}\bigr) f_\al(t).
\end{aligned}\eeq

To further simplify the notations, we introduce
\beno
z_{00}= z,\ \ z_{01}=\log\wm z,\ \ z_{02}=(\log\wm)^2 z,\ \ z_\al=\wm(\log\wm)^2\p^\al z.
\eeno
Using the new notations, we have
\ben\label{chaeql}\mathcal{E}^\mu_{(\text{lin}),+}(t)\sim  \sum_{k=0}^2\|z_{0k}(t)\|_{L^2}^2
+\!\!\!\!\!\!\!\!\sum_{1\leq|\al|\leq N_*+1}\!\!\!\!\!\!\!\!\|z_\al(t)\|_{L^2}^2+\mu\!\!\!\!\sum_{|\al|=N_*+1}\!\!\!\!\!\|\na z_\al(t)\|_{L^2}^2. \een

By virtue of \eqref{estimates for nonlinear term 1}, for all $t\geq 0$, we have
\begin{equation}\label{estimate for z}
\begin{split}
& \sum_{k\leq 2}\|z_{0k}\|_{L^2}^2
+\!\!\!\!\!\!\!\!\!\!\sum_{1\leq|\al|\leq N_*+1}\|z_\al\|_{L^2}^2+\mu\!\!\!\!\!\!\sum_{|\al|=N_*+2}\!\!\!\!\!\!\|z_\al\|_{L^2}^2 +\mu\int_0^\infty\Bigl(\sum_{k\leq 2}	 \|\na z_{0k}\|_{L^2}^2
+\!\!\!\!\!\!\!\!\sum_{1\leq|\al|\leq N_*+1}\!\!\!\!\!\!\!\!\|\na z_\al\|_{L^2}^2\Bigr)dt\\
&\quad +\mu^2\!\!\!\!\sum_{|\al|=N_*+2}\!\int_0^\infty\|\na z_\al\|_{L^2}^2dt\lesssim \mathcal{E}^\mu(0),
\end{split}
\end{equation}
where $L^2$ should be understood as $L^2(\Sigma_t)$.

\medskip

{\bf Step 3. Decomposition of $z_\al$ and the refined energy for \eqref{quasilinear part 1}.}

\medskip

By definition, $z_{0k}$ and $z_\al$ satisfy the following equations:
\begin{equation}\label{refined quasilinear system}\begin{split}
\p_tz_{00}+Z_-\cdot\na z_{00}-\mu\Delta z_{00}&=0,\\
\p_tz_{01}+Z_-\cdot\na z_{01}-\mu\Delta z_{01}&=-2\mu\bigl(\na(\log\wm)\cdot\na\bigr)z-\mu z\Delta(\log\wm),\\
\p_tz_{02}+Z_-\cdot\na z_{02}-\mu\Delta z_{02}&=-2\mu\bigl(\na\bigl((\log\wm)^2\bigr)\cdot\na\bigr)z-\mu z\Delta\bigl((\log\wm)^2\bigr),\\
\p_tz_\al+Z_-\cdot\na z_\al-\mu\Delta z_\al&=-2\mu\bigl(\na\bigl(\wm(\log\wm)^2\bigr)\cdot\na\bigr)\p^\al z\\
&\ \ \ -\mu\p^\al z\Delta\bigl(\wm(\log\wm)^2\bigr)-[\p^\al, z_-]\cdot\na z \cdot	 \wm(\log\wm)^2.
\end{split}
\end{equation}

\medskip

{\it Step 3.1. Decomposition of  $z_\al$.}
We split  $z_\al$ into two parts  $z_\al=Y_\al+R_\al$. The vector fields $Y_\al$  and $R_\al$ satisfy
\begin{equation}\label{quasilinear part 3}\begin{split}
\p_tY_\al+Z_-\cdot\na Y_\al-\mu\Delta Y_\al&=-2\mu\bigl(\na\bigl(\wm(\log\wm)^2\bigr)\cdot\na\bigr)\p^\al z-\mu\p^\al z\Delta\bigl(\wm(\log\wm)^2\bigr),\\
Y_\al|_{t=0}&=z_\al|_{t=0}\big(=(\wm(\log\wm)^2\p^\al z_+)(0,x)\big),
\end{split}
\end{equation}
and
\begin{equation}\label{nonlinear part 1}
\begin{split}
\p_tR_\al+Z_-\cdot\na R_\al-\mu\Delta R_\al&=-[\p^\al, z_-]\cdot\na z \wm(\log\wm)^2,\\
R_\al|_{t=0}&=0.
\end{split}
\end{equation}

\medskip

{\it Step 3.2. Energy estimates for \eqref{quasilinear part 3}.}
By taking $L^2(\Sigma_t)$-product with $Y_\al$, \eqref{quasilinear part 3} yields
\beno\begin{aligned}
&\f{1}{2}\f{d}{dt}\|Y_\al\|_{L^2}^2+\mu\|\na Y_\al\|_{L^2}^2= -\mu\int_{\Sigma_t}[2\bigl(\na\bigl(\wm(\log\wm)^2\bigr)\cdot\na\bigr)\p^\al z\cdot Y_\al +\p^\al z\Delta\bigl(\wm(\log\wm)^2\bigr)\cdot Y_\al] dx.
\end{aligned}\eeno
Since $|\na\bigl(\wm(\log\wm)^2\bigr)|\lesssim (\log\wm)^2$ and $\|\f{Y_\al}{\wm}\|_{L^2} {\lesssim}\|\na Y_\al\|_{L^2}$, integrating the second term in the righthand side by parts will lead to  the following upper bound for the righthand side:
$$
\mu\|\wm(\log\wm)^2\na(\p^\al z)\|_{L^2}\|\na Y_\al\|_{L^2}+\mu\|(\log\wm)^2\p^\al z\|_{L^2}\|\na Y_\al\|_{L^2}.
$$
Hence,
\beno\label{mu 1 decay}
\f{d}{dt}\|Y_\al\|_{L^2}^2+\mu\|\na Y_\al\|_{L^2}^2\lesssim\mu\|\wm(\log\wm)^2\na(\p^\al z)\|_{L^2}^2+\mu\|(\log\wm)^2\p^\al z\|_{L^2}^2.
\eeno

Since $\|\wm(\log\wm)^2\na(\p^\al z)\|_{L^2}$ can be bounded by
\beno\begin{aligned}
& \sum_{1\leq|\beta|\leq|\al|}\|\na\bigl(\wm(\log\wm)^2(\p^\beta z)\bigr)\|_{L^2}
+\|\na\bigl((\log\wm)^2 z\bigr)\|_{L^2} +\|\na\bigl(\log\wm z\bigr)\|_{L^2},
\end{aligned}\eeno
Then we have
\beq\label{quasilinear 4}\begin{aligned}
&\f{d}{dt}\bigl(\sum_{1\leq|\al|\leq N_*+1}\|Y_\al\|_{L^2}^2+\mu\sum_{|\al|=N_*+2}\|Y_\al\|_{L^2}^2\bigr)+\mu\sum_{1\leq|\al|\leq N_*+1}\|\na Y_\al\|_{L^2}^2
+\mu^2\sum_{|\al|=N_*+2}\|\na Y_\al\|_{L^2}^2\\
 &\lesssim\mu\sum_{k=0}^2\|\na z_{0k}\|_{L^2}^2+\mu\!\!\!\!\sum_{1\leq|\al|\leq N_*+1}\!\!\!\!\|\na z_\al\|_{L^2}^2+\mu^2\sum_{|\al|=N_*+2}\|\na z_\al\|_{L^2}^2.
\end{aligned}\eeq
Integrating over $t$, \eqref{quasilinear 4} together with \eqref{estimate for z} gives the following bound on $Y_\alpha$:
\beq\label{Estimate for Y}\begin{aligned}
&\sup_{t\geq0}\bigl(\sum_{1\leq|\al|\leq N_*+1}\|Y_\al\|_{L^2}^2+\mu\sum_{|\al|=N_*+2}\|Y_\al\|_{L^2}^2\bigr)\\
&\qquad
+\mu\sum_{1\leq|\al|\leq N_*+1}\int_0^\infty\|\na Y_\al\|_{L^2}^2dt
+\mu^2\sum_{|\al|=N_*+2}\int_0^\infty\|\na Y_\al\|_{L^2}^2dt\lesssim\mathcal{E}^\mu(0).
\end{aligned}\eeq

\medskip

{\it Step 3.3. Energy estimate for  \eqref{nonlinear part 1}.}
Once again, similar to the derivation of \eqref{global energy estimate for MHDmu}, for all $t\geq 0$ and $u_+ \in \R$, we have
\beno\begin{aligned}
&\sum_{1\leq|\al|\leq N_*+1}\|R_\al(t)\|_{L^2}^2+\mu\sum_{|\al|=N_*+2}\|R_\al(t)\|_{L^2}^2+ \sum_{1\leq|\al|\leq N_*+1}\int_{C_{u_+}}|R_\al|^2 d\sigma_+ \\
&\quad+\mu\sum_{1\leq|\al|\leq N_*+1}\int_0^t\|\na R_\al(\tau)\|_{L^2}^2 d\tau
+\mu^2\sum_{|\al|=N_*+2}\int_0^t\|\na R_\al(\tau)\|_{L^2}^2 d \tau\\
&\lesssim\bigl(\sum_{1\leq|\al|\leq N_*+1}+\mu\sum_{|\al|=N_*+2}\bigr)\underbrace{\bigg|\int_0^t\int_{\R^3}\biggl(\big([\p^\al, z_-]\cdot\na z\big) \wm(\log\wm)^2\biggr)\cdot R_\al dxd\tau  \bigg|  }_{I_\al}
\end{aligned}\eeno
Since  $z=z_+^{(\text{lin})}$, in view of \eqref{global energy estimate for MHDmu} and \eqref{estimate for quasilinear part}, we have
for $1\leq|\al|\leq N_*+1$
\beno\begin{aligned}
I_\al\lesssim\mathcal{E}^\mu(0)\Bigl(\sup_{u_+}\int_{C_{u_+}}|R_\al|^2d\sigma_+\Bigr)^{\f12}.
\end{aligned}\eeno
Whereas for $|\al|=N_*+2$, we have
\beno\begin{aligned}
\mu I_\al&\lesssim\sqrt\mu\int_0^t\|\big([\p^\al, z_-]\cdot\na z\big) \wm(\log\wm)^2\|_{L^2}d\tau\cdot\Bigl(\sup_{0\leq\tau\leq t}\mu\|R_\al(\tau)\|_{L^2}^2\Bigr)^{\f12}\\
&\lesssim\bigl(\underbrace{\sum_{0\leq k\leq N_*-2}}_{A_1}+\underbrace{\sum_{N_*-1\leq k\leq N_*+1}}_{A_2}\bigr)\sqrt\mu\int_0^t\|(\p^{N_*+2-k}z_-)\cdot(\p^k\na z)\wm(\log\wm)^2\|_{L^2}d\tau
\cdot\Bigl(\sup_{0\leq\tau\leq t}\mu\|R_\al(\tau)\|_{L^2}^2\Bigr)^{\f12}.
\end{aligned}\eeno
For $A_1$, we have
\beno\begin{aligned}
A_1&\lesssim\sum_{0\leq k\leq N_*-2}\sqrt\mu\|\wp\bigl(\log\wp\bigr)^2\p^{N_*+2-k}z_-\|_{L^2_\tau L^2_x}\cdot\|\f{\wm(\log\wm)^2}{\wp\bigl(\log\wp\bigr)^2}\p^k\na z\|_{L^2_\tau L^\infty_x}\\
&\lesssim\sum_{k\leq N_*}\bigl(D_-^k\bigr)^{\f12}\sum_{k\leq N_*}\bigl(F_+^k(\na z)\bigr)^{\f12}.
\end{aligned}\eeno
For $A_2$, we have
\beno\begin{aligned}
A_2&\lesssim\sum_{N_*-1\leq k\leq N_*+1}\sqrt\mu\|\wp\bigl(\log\wp\bigr)^2\p^{N_*+2-k}z_-\|_{L^2_\tau L^\infty_x}\cdot\|\f{\wm(\log\wm)^2}{\wp\bigl(\log\wp\bigr)^2}\p^k\na z\|_{L^2_\tau L^2_x}\\
&\lesssim\sum_{k\leq N_*}\bigl(D_-+D_-^k\bigr)^{\f12}\sum_{k\leq N_*}\bigl(F_+^k(\na z)\bigr)^{\f12},
\end{aligned}\eeno
where we used the Gagliardo-Nirenberg interpolation inequality $\|u\|_{L^\infty}\lesssim\|\na u\|_{L^2}^{\f12}\|\na^2u\|_{L^2}^{\f12}$.
Then by virtue of \eqref{global energy estimate for MHDmu} and \eqref{estimate for quasilinear part}, we have
\beno
\mu I_\al\lesssim\mathcal{E}^\mu(0)\Bigl(\sup_{0\leq\tau\leq t}\mu\|R_\al(\tau)\|_{L^2}^2\Bigr)^{\f12}.
\eeno

Thus, by the smallness of $\mathcal{E}^\mu(0)$, for all $t\geq0$, $u_+\in\R$, we finally obtain
\beq\label{estimate for nonlinear term 2}\begin{aligned}
&\sum_{1\leq|\al|\leq N_*+1}\|R_\al(t)\|_{L^2}^2+\mu\sum_{|\al|=N_*+2}\|R_\al(t)\|_{L^2}^2+ \sum_{1\leq|\al|\leq N_*+1}\int_{C_{u_+}}|R_\al|^2 d\sigma_+ \\
&\quad+\mu\sum_{1\leq|\al|\leq N_*+1}\int_0^t\|\na R_\al(\tau)\|_{L^2}^2 d\tau
+\mu^2\sum_{|\al|=N_*+2}\int_0^t\|\na R_\al(\tau)\|_{L^2}^2 d \tau \lesssim\bigl(\mathcal{E}^\mu(0)\bigr)^{\f32}.
\end{aligned}\eeq

\medskip

{\it Step 3.4. Energy estimates for the Lagrangian forms.}
We recall that $\wt v(t,y)=v(t,\psi_-(t,y))$ is a Lagrangian form of $v$, i.e., in the Lagrangian coordinates system.  Since $\det\bigl(\f{\p\psi_-}{\p y}\bigr)=1$, we have $\|v\|_{L^2}=\|\wt v\|_{L^2}$. When we write $\nabla\wt v(t,y)$, the derivative $\nabla$ is always understood as taken with respect to $y$. Therefore, \eqref{quasilinear 1} and \eqref{quasilinear 4} together give the following estimates (the $L^2$ norms are taken on  $\Sigma_0$):
\beq\label{quasilinear estimate on Lagrangian 0} \begin{aligned}
&\ \ \ \f{d}{dt}\Bigl(\sum_{k=0}^2 c_{0k}\|\wt{z_{0k}}\|_{L^2}^2
+\!\!\!\!\!\!\!\!\sum_{1\leq|\al|\leq N_*+1} \!\!\!\!\!\!\!\! c_\al\|\wt{z_\alpha}\|_{L^2}^2+\mu\!\!\!\!\!\!\sum_{|\al|=N_*+2}\!\!\!\!\!\! c_\al\|\wt{z_\alpha}\|_{L^2}^2
+\!\!\!\!\!\!\sum_{1\leq|\al|\leq N_*+1}\!\!\!\!\!\! c'_\al\|\wt{Y_\alpha}\|_{L^2}^2+\mu\!\!\!\!\!\!\sum_{|\al|=N_*+2}\!\!\!\!\!\!c'_\al\|\wt{ Y_\al}\|_{L^2}^2\Bigr)\\
&+c\mu\sum_{k=0}^2\|\na\wt{z_{0k}}\|_{L^2}^2
+\mu\!\!\!\!\!\!\sum_{1\leq|\al|\leq N_*+1}\!\!\!\!\!\!\big(c\|\na\wt{z_\al}\|_{L^2}^2+c' \|\na\wt{Y_\al}\|_{L^2}^2\big)+\mu^2\!\!\!\!\!\!\sum_{|\al|=N_*+2}\!\!\!\!\!\!\big(c\|\na\wt{z_\al}\|_{L^2}^2+c'\|\na\wt{ Y_\al}\|_{L^2}^2\big)\\
&\leq\bigl(\sum_{1\leq|\al|\leq N_*+1}+\mu\sum_{|\al|=N_*+2}\bigr)f_\al(t).
\end{aligned}\eeq

{\it Step 3.5.  The refined energy $X(t)$.}
Let
\beq\label{energy functional}
X(t)=\sum_{k=0}^2 c_{0k} \|\wt{z_{0k}}\|_{L^2}^2
+\sum_{1\leq|\al|\leq N_*+1}(c_\al+c'_\al) \|\wt{Y_\al}\|_{L^2}^2
+\mu\sum_{|\al|=N_*+2}(c_\al+ c'_\al)\|\wt{Y_\al}\|_{L^2}^2.
\eeq
In view of the fact that $z_\al=Y_\al+R_\al$ and estimates \eqref{chaeql},
\eqref{Estimate for Y} and \eqref{estimate for nonlinear term 2}, we have
\ben\label{characterize X}  \mathcal{E}^\mu_{(\text{lin}),+}+\bigl(\mathcal{E}^\mu(0)\bigr)^{\f32}\sim X(t)+\bigl(\mathcal{E}^\mu(0)\bigr)^{\f32}. \een
We rewrite \eqref{quasilinear estimate on Lagrangian 0} as
\beq\label{quasilinear estimate on Lagrangian} \begin{aligned}
&\f{d}{dt}X(t)+c\mu\Bigl(\sum_{k=0}^2\|\na\wt{z_{0k}}\|_{L^2}^2+\!\!\!\!\!\!\sum_{1\leq|\al|\leq N_*+1}\!\!\!\!\!\!\|\na\wt{Y_\al}\|_{L^2}^2+\mu\!\!\!\!\!\!\sum_{|\al|=N_*+2}\!\!\!\!\!\!\|\na\wt{Y_\al}\|_{L^2}^2\Bigr)\leq F(t),
\end{aligned}\eeq
where
\ben\label{defintion for F(t)}
F(t)&=&\bigl(\sum_{1\leq|\al|\leq N_*+1}+\mu\sum_{|\al|=N_*+2}\bigr)\Bigl(f_\al(t)
-2c_\al \f{d}{dt}\int_{\Sigma_t}\wt{Y_\al}\cdot\wt{R_\al}dy -c_\al \f{d}{dt}\|\wt{R_\al}\|_{L^2}^2\Bigr).
\een	

To further refine the estimates, we quickly introduce a dyadic decomposition. Let $\psi$ and $\phi$ be non-negative smooth functions so that
$\text{supp}\,\psi\subset B_{\f43}=\{y\,|\, |y|\leq\f{4}{3}\}$, $\text{supp}\,\phi\subset \mathcal{C}=\{y\,|\, \f34\leq|y|\leq\f83\}$ and
$$\psi(y)+\sum_{j\geq0}\phi(2^{-j}y)=1.$$
Let $p_{-1}(y)= \psi(y)$. For $j\geq0$, we define $p_j(y) = \phi(2^{-j}y)$.

We use $\widehat{f}$ to denote the Fourier transform of a function or a vector field $f$ on $\mathbb{R}^3$.

By Hardy inequality, we have
$$
2^{-j}\|p_j\wt z\|_{L^2} \lesssim \|\f{p_j\wt z}{\langle y\rangle}\|_{L^2} \lesssim \|\na(p_j\wt z)\|_{L^2}^2.$$
Hence,
\beq\label{equivalent 2}\begin{aligned}
\|\na\wt z\|_{L^2}^2 \sim \sum_{j=-1}^\infty\big(\|p_j\na\wt z\|_{L^2}^2 +2^{-2j}\|p_j\wt z\|_{L^2}^2\big)\sim\sum_{j=-1}^\infty\big(\|\na(p_j\wt z)\|_{L^2}^2+ 2^{-2j}\|p_j\wt z\|_{L^2}^2\big).
\end{aligned}\eeq
We now pick up a function $h(t)\geq 0 $ and it will be determined later on. By Plancherel theorem, we have
\ben\label{inht}
\|\na\wt z\|_{L^2}^2\simeq\int_{\R^3}|\xi|^2|\widehat{\wt z}(\xi)|^2d\xi
\geq h(t)^2\|\wt z\|_{L^2}^2-{h(t)^2}\int_{|\xi|\leq h(t)}|\widehat{\wt z}(\xi)|^2d\xi.
\een
By \eqref{equivalent 2} and \eqref{inht}, we deduce that
\begin{align*}
&\ \ \ \sum_{k=0}^2\|\na\wt{z_{0k}}\|_{L^2}^2
+\!\!\!\!\!\!\sum_{1\leq|\al|\leq N_*+1}\!\!\!\!\!\!	\|\na\wt{Y_\al}\|_{L^2}^2
+\mu\!\!\!\!\!\!\sum_{|\al|=N_*+2}\!\!\!\!\!\!\|\na\wt{Y_\al}\|_{L^2}^2\\
&\gtrsim\sum_{k\leq 2 \atop j\geq -1} \|\na(p_j\wt{z_{0k}})\|_{L^2}^2
+\!\!\!\!\!\!\!\!\sum_{1\leq|\al|\leq N_*+1 \atop j\geq -1}\!\!\!\!\!\!\!\!\|\na(p_j\wt{Y_\al})\|_{L^2}^2
+\mu\!\!\!\!\!\!\sum_{|\al|=N_*+2 \atop j\geq -1}\!\!\!\!\!\!\|\na(p_j\wt{Y_\al})\|_{L^2}^2
\\
&\gtrsim h(t)^2X(t)-{h(t)^2}\int_{|\xi|\leq h(t)}\bigg(\sum_{k\leq 2 \atop j\geq -1} |\widehat{p_j\wt{z_{0k}}}(\xi)|^2
 +\!\!\!\!\!\!\!\!\sum_{1\leq|\al|\leq N_*+1 \atop j\geq -1}  |\widehat{p_j\wt{Y_\al}}(\xi)|^2
 +\mu\!\!\!\!\!\!\sum_{|\al|=N_*+2 \atop j\geq -1} |\widehat{p_j\wt{Y_\al}}(\xi)|^2 \bigg)d\xi.
\end{align*}
We then deduce from \eqref{quasilinear estimate on Lagrangian} that
\beq\label{decay estimate 1}\begin{aligned}
&\f{d}{dt} X(t)+c\mu h(t)^2 X(t)\\
&\leq F(t)
+c\mu h(t)^2\int_{|\xi|\leq h(t)}\bigg(\sum_{k\leq 2 \atop j\geq -1} |\widehat{p_j\wt{z_{0k}}}(\xi)|^2
 +\!\!\!\!\!\!\!\!\sum_{1\leq|\al|\leq N_*+1 \atop j\geq -1}  |\widehat{p_j\wt{Y_\al}}(\xi)|^2
 +\mu\!\!\!\!\!\!\sum_{|\al|=N_*+2 \atop j\geq -1} |\widehat{p_j\wt{Y_\al}}(\xi)|^2 \bigg)d\xi.
\end{aligned}\eeq
The integral terms on the right-side will be called low frequency terms.

\medskip

{\bf Step 4. Estimates on the low frequency terms in \eqref{decay estimate 1}.}

\smallskip

{\it Step 4.1. Estimates on $\int_{|\xi|\leq h(t)}|\widehat{p_j\wt{Y_\al}}|^2d\xi$.}
Since $\psi_-(t,y)$ is the flow generated by $Z_-$ which defines the coordinates   $(t, y_1,y_2,y_3)$, we have
\beno
(\p_t+Z_{-}\cdot\na)f|_{x=\psi_-(t,y)}=\p_t\wt f(t,y),\quad \na f|_{x=\psi_-(t,y)}=\bigl(\f{\p\psi_-(t,y)}{\p y}\bigr)^{-T}\na_y\wt f(t,y).
\eeno
Let $A_{-}=\bigl(\f{\p\psi_-(t,y)}{\p y}\bigr)^{-1}\bigl(\f{\p\psi_-(t,y)}{\p y}\bigr)^{-T}$. By the divergence free property of $Z_-$,  we have
$\det\bigl(\f{\p\psi_-(t,y)}{\p y}\bigr)=1$ so that $\bigl(\f{\p\psi_-(t,y)}{\p y}\bigr)^{-1}$ is the adjoint matrix of $\f{\p\psi_-(t,y)}{\p y}$.
Then we have $\dv \bigl(\f{\p\psi_-(t,y)}{\p y}\bigr)^{-1}=0$ and
$$
\Delta f|_{x=\psi_-(t,y)} =\na_y\cdot\Bigl(A_{-}\na_y \wt f(t,y)\Bigr),
$$

Thus, \eqref{quasilinear part 3} can be written as
\beq\label{Lagarangian formulation for Y}\begin{aligned}
\p_t\wt{Y_\al}-\mu\Delta\wt{Y_\al}&=\mu\dv\bigl((A_{-}-I)\na\wt{Y_\al}\bigr)-2\mu\bigl(\f{\p\psi_-}{\p y}\bigr)^{-T}\na\bigl(\langle y\rangle(\log\langle y\rangle)^2\bigr)\cdot\wt{\na\p^\al z}\\
&\ \ \ -\mu\wt{\p^{\al} z}\dv\bigl(A_-\na[\langle y\rangle(\log\langle y\rangle)^2]\bigr),
\end{aligned}\eeq
In the above expression, we used the fact that $\wpm|_{x=\psi_\pm(t,y)}=\langle y\rangle$.

We decompose \eqref{Lagarangian formulation for Y} by multiplying $p_j$ for each $j\geq -1$:
\begin{equation}\label{Lagarangian formulation for pjY}
\begin{split}
\p_t(p_j\wt{Y_\al})-\mu\Delta(&p_j\wt{Y_\al})=\underbrace{-2\mu\na p_j\cdot\na \wt{Y_\al}-\mu\Delta p_j \wt{Y_\al}}_{L_\al^1}+\underbrace{\mu p_j \dv\bigl((A_--I)\na\wt{Y_\al}\bigr)}_{L_\al^2}\\
&\ \ \underbrace{-2\mu\bigl(\f{\p\psi_-}{\p y}\bigr)^{-T}\na\bigl(\langle y\rangle(\log\langle y\rangle)^2\bigr)\cdot(p_j\wt{\na\p^\al z})}_{L_\al^3}-\underbrace{\mu (p_j\wt{\p^\al z})\dv\bigl(A_-\na[\langle y\rangle(\log\langle y\rangle)^2]\bigr)}_{L_\alpha^4}.
\end{split}
\end{equation}

We use $L_\alpha$ to denote the righthand side of \eqref{Lagarangian formulation for pjY}. In frequency space, we have
\beno
\f{d}{dt}|\widehat{p_j\wt{Y_\al}}|^2+2\mu|\xi|^2|\widehat{p_j\wt{Y_\al}}|^2
=2\mathcal{R}e ( \widehat{L}_\al\cdot\overline{\widehat{p_j\wt{Y_\al}}}),
\eeno
which implies
\beno
|\widehat{p_j\wt{Y_\al}}|^2(t,\xi)=e^{-2\mu t|\xi|^2}|\widehat{p_j\wt{Y_\al}}(0,\xi)|^2+2\int_0^te^{-2\mu(t-s)|\xi|^2} \mathcal{R}e
\bigl(\widehat{L}_\al\cdot\overline{\widehat{p_j\wt{Y_\al}}}\bigr)(s,\xi)ds,
\eeno
Therefore,
\ben\label{Estimate for pjY}
\int_{|\xi|\leq h(t)}|\widehat{p_j\wt{Y_\al}}|^2(t,\xi)d\xi
&\leq&\int_{\R^3}e^{-2\mu t|\xi|^2}|\widehat{p_j\wt{Y_\al}}\big|_{t=0}|^2\psi(\f{3|\xi|}{4h(t)}) d\xi \nonumber\\
&&+\underbrace{2\mathcal{R}e\int_0^t\int_{\R^3}e^{-2\mu(t-s)|\xi|^2}\psi(\f{3|\xi|}{4h(t)})
\bigl(\widehat{L}_\al\cdot\overline{\widehat{p_j\wt{Y_\al}}}\bigr)(s,\xi) d\xi ds}_{T_j^{\al}}.
\een
By Plancherel theorem, we have
\beno
|T_j^{\al}|\lesssim|\int_0^t\int_{\R^3}L_\al\cdot a_1*a_2*\bigl(p_j\wt{Y_\al}\bigr)dy ds|.
\eeno
where we take the following $a_1,\ a_2\in\mathcal{S}(\R^3)$:
\begin{align*}
a_1=(2\sqrt{\mu(t-s)})^{-3}e^{-\f{|x|^2}{8\mu(t-s)}} &\Leftrightarrow \hat{a}_1=e^{-2\mu (t-s)|\xi|^2},\\
a_2=(\f43h(t))^3\breve{\psi}(\f43h(t)x) &\Leftrightarrow \hat{a}_2=\psi(\f{3|\xi|}{4h(t)}).
\end{align*}
There exists a universal constant  $C$ independent of $t,s,\mu$, such that
\beno
\|a_1\|_{L^1}+\|a_2\|_{L^1}\leq C.
\eeno
To proceed, we first make the following observation
\beq\label{property for pj}
\na p_j=\na p_j(\sum_{|j-k|\leq 2}p_k)=\na p_j p'_j,\quad p_j\langle y\rangle^{-1}\leq 2^{-j} p_j,\quad |\p^k p_j|\lesssim 2^{-kj}p'_{j},
\eeq
where $p'_j=\sum_{|j-k|\leq 2}p_k$.

We first bound $T_{j1}^\al\eqdefa \int_0^t\int_{\R^3}L_\al^1\cdot a_1*a_2*\bigl(p_j\wt{Y_\al}\bigr)dyds$, i.e., the contribution from $L^1_\al$ in $T_j^\al$. By integration by parts, we have
\beno\begin{aligned}
T_{j1}^\al
&=\mu\int_0^t\int_{\R^3}(\Delta p_j) \wt{Y_\al}\cdot a_1*a_2*\bigl(p_j\wt{Y_\al}\bigr)dyds
+2\mu\int_0^t\int_{\R^3}(\p p_j) \wt{Y_\al}\cdot\p\bigl(a_1*a_2*(p_j\wt{Y_\al})\bigr)dyds.
\end{aligned}\eeno
According to \eqref{property for pj}, we have
\beno\begin{aligned}
|T_{j1}^\al|&
\lesssim\mu\int_0^t\int_{\R^3}|p'_j\wt{Y_\al}|\cdot\bigl( 2^{-2j}|a_1*a_2*(p_j\wt{Y_\al})|+2^{-j}|\p\bigl((a_1*a_2*(p_j\wt{Y_\al})\bigr)|\bigr)dy ds\\
&\stackrel{|y|\sim 2^j}{\lesssim}\mu\int_0^t\|p'_j\wt{Y_\al}\|_{L^2}\bigl(2^{-j}\|\f{1}{|y|}a_1*a_2*\bigl(p_j\wt{Y_\al}\bigr)\|_{L^2}
+2^{-j}\|\na\bigl(a_1*a_2*(p_j\wt{Y_\al})\bigr)\|_{L^2}\bigr)ds.
\end{aligned}\eeno
In view of the fact that $supp\,\hat{a}_2\subset\{|\xi|\leq h(t)\}$, by Young's and Hardy's inequalities and Plancherel  theorem, we have
\beno\begin{aligned}
\|\f{1}{|y|}a_1*a_2*\bigl(p_j\wt{Y_\al}\bigr)\|_{L^2}&\lesssim\|\na\bigl(a_1*a_2*(p_j\wt{Y_\al})\bigr)\|_{L^2} \lesssim  \|a_1*(|\na|\na a_2)*\bigl(|\nabla|^{-1}(p_j\wt{Y_\al})\bigr)\|_{L^2}\\
&\lesssim h(t)^2
\|a_1\|_{L^1}\|a_2\|_{L^1}\||\nabla|^{-1}(p_j\wt{Y_\al})\bigr)\|_{L^2}
\lesssim h(t)^2\|\f{1}{|\xi|}\widehat{p_j\wt{Y_\al}}\|_{L^2}\\
&\stackrel{Hardy}{\lesssim} h(t)^2\|\na_\xi\widehat{p_j\wt{Y_\al}}\|_{L^2}\lesssim h(t)^2\|\widehat{yp_j\wt{Y_\al}}\|_{L^2}
\stackrel{|y|\lesssim 2^j}\lesssim 2^j h(t)^2\|p_j\wt{Y_\al}\|_{L^2}.
\end{aligned}\eeno
Hence,
\beq\label{T}
2^{-j}\|\f{1}{|y|}a_1*a_2*\bigl(p_j\wt{Y_\al}\bigr)\|_{L^2}\lesssim2^{-j}\|\na\bigl(a_1*a_2*(p_j\wt{Y_\al})\bigr)\|_{L^2}
\lesssim h(t)^2\|p_j\wt{Y_\al}\|_{L^2}.
\eeq

Thus, we obtain
\beq\label{Estimate for Lal 1}
|T_{j1}^\al|\lesssim\mu h(t)^2\int_0^t\|p_j\wt{Y_\al}\|_{L^2}\|p_j'\wt{Y_\al}\|_{L^2}ds\lesssim \mu h(t)^2\int_0^t\|p_j'\wt{Y_\al}\|_{L^2}^2ds.
\eeq

We then bound $T_{j2}^\al=\int_0^t\int_{\R^3}L_\al^2\cdot a_1*a_2*\bigl(p_j\wt{Y_\al}\bigr)dyds$, i.e., the contribution from $L^2_\al$ in $T_j^\al$.  By integration by parts, we have
\beno
T_{j2}^\al&=&-\mu\int_0^t\int_{\R^3}\bigl(\na p_j\cdot(A_--I)\na\wt{Y_\al}\bigr)\cdot a_1*a_2*\bigl(p_j\wt{Y_\al}\bigr)dyds\\&&
-\mu\int_0^t\int_{\R^3}p_j(A_--I)\na\wt{Y_\al}\cdot\na\bigl(a_1*a_2*(p_j\wt{Y_\al})\bigr)dyds.
\eeno
Similar to the argument used to bound $T_{j1}^\al$, we have
\beno\begin{aligned}
T_{j2}^\al&\lesssim\mu    \int_0^t \|A_--I\|_{L^\infty_y} \bigl(\|p_j'\na\wt{Y_\al}\|_{L^2}+\|p_j\na\wt{Y_\al}\|_{L^2}\bigr)
\|a_1*a_2*\na(p_j\wt{Y_\al})\|_{L^2}ds.
\end{aligned}\eeno
By Young's inequality, we have $\|a_1*a_2*\na(p_j\wt{Y_\al})\|_{L^2} \lesssim\|\na(p_j\wt{Y_\al})\|_{L^2}$. We then conclude that
\beq\label{Estimate for Lal 2}
|T_{j2}^\al|
\lesssim \mu  \|A_--I\|_{L^\infty_{t,y}} \int_0^t\|p_j'\na\wt{Y_\al}\|_{L^2}\|\na(p_j\wt{Y_\al})\|_{L^2}ds.
\eeq

We move to the bound on $T_{j3}^\al=\int_0^t\int_{\R^3}L_\al^3\cdot a_1*a_2*\bigl(p_j\wt{Y_\al}\bigr)dyds$. First of all, since  $\wt{\na\p^\al z}=\bigl(\f{\p\psi_-}{\p y}\bigr)^{-T}\na \wt{\p^\al z}$, we rewrite $L_\alpha^3$ as
\beno\begin{aligned}
L_\al^3
&=\underbrace{-2\mu\na\bigl(\langle y\rangle(\log\langle y\rangle)^2\bigr)\cdot(p_j\na \wt{\p^\al z})}_{L_\al^{31}}
\underbrace{-2\mu\na\bigl(\langle y\rangle(\log\langle y\rangle)^2\bigr)\cdot\Bigl(\bigl(\f{\p\psi_-}{\p y}\bigr)^{-T}-I\Bigr)(p_j\na\wt{\p^\al z})}
_{L_\al^{32}}\\
&\quad
\underbrace{-2\mu\Bigl(\bigl(\f{\p\psi_-}{\p y}\bigr)^{-T}-I\Bigr)\na\bigl(\langle y\rangle(\log\langle y\rangle)^2\bigr)\cdot(p_j\wt{\na\p^\al z})}
_{L_\al^{33}}.
\end{aligned}\eeno
For $i\leq 3$, let $T_{j3i}^\al$ be the contribution of $L_\alpha^{3i}$ in $T_{j3}^\al$. For $T_{j31}^\al$, by integration by parts and \eqref{property for pj}, we have
\beno\begin{aligned}
|T_{j31}^\al|&\lesssim\mu\int_0^t\int_{\R^3}\bigl(|p_j\bigl(\langle y\rangle(\log\langle y\rangle)^2\wt{\p^\al z}\bigr)||\f{a_1*a_2*\bigl(p_j\wt{Y_\al}\bigr)}{\langle y\rangle^2} | +|p'_j\bigl(\langle y\rangle(\log\langle y\rangle)^2\wt{\p^\al z}\bigr)| 2^{-j} |\f{a_1*a_2*\bigl(p_j\wt{Y_\al}\bigr)}{\langle y\rangle} |\bigr)dyds \\
&\quad+\mu\int_0^t\int_{\R^3}|p_j\bigl(\langle y\rangle(\log\langle y\rangle)^2\wt{\p^\al z}\bigr)|\cdot|\f{\na\bigl(a_1*a_2*(p_j\wt{Y_\al})\bigr)}{\langle y\rangle}|dyds.
\end{aligned}\eeno
We observe that $p_j\f{1}{{\langle y\rangle}}\lesssim p_j2^{-j}$ and $\langle y\rangle(\log\langle y\rangle)^2\wt{\p^\al z}=\wt{z_\al}$.

Thanks to \eqref{T}, 
we have
\beno
|T_{j31}^\al|
\lesssim\mu h(t)^2\int_0^t\|p'_j\wt{z_\al}\|_{L^2}\|p_j\wt{Y_\al}\|_{L^2}ds.
\eeno
To bound $T_{j32}^\al$ and $T_{j33}^\al$, we have
\beno\begin{aligned}
|T_{j32}^\al+T_{j33}^\al|&
\lesssim\mu\|\bigl(\f{\p\psi_-}{\p y}\bigr)^{-T}-I\|_{L^\infty_{t,y}}\int_0^t\bigl(\|p_j\bigl(\langle y\rangle(\log\langle y\rangle)^2\na\wt{\p^\al z}\bigr)\|_{L^2}
\\&\quad+\|p_j\bigl(\langle y\rangle(\log\langle y\rangle)^2\wt{\na\p^\al z}\bigr)\|_{L^2}\bigr)\|\f{a_1*a_2*\bigl(p_j\wt{Y_\al}\bigr)}{\langle y\rangle} \|_{L^2}ds\\
&\stackrel{Hardy}{\lesssim}\mu\|\bigl(\f{\p\psi_-}{\p y}\bigr)^{-T}-I\|_{L^\infty_{t,y}}\int_0^t\|p_j\bigl(\langle y\rangle(\log\langle y\rangle)^2\wt{\na\p^\al z}\bigr)\|_{L^2}\|a_1*a_2*\na(p_j\wt{Y_\al})\|_{L^2}ds\\
&\lesssim\mu\|\bigl(\f{\p\psi_-}{\p y}\bigr)^{-T}-I\|_{L^\infty_{t,y}}\int_0^t\|p_j\bigl(\langle y\rangle(\log\langle y\rangle)^2\wt{\na\p^\al z}\bigr)\|_{L^2}\|\na(p_j\wt{Y_\al})\|_{L^2}ds.
\end{aligned}\eeno

As a result, we finally have
\beq\label{Estimate for Lal 3}\begin{aligned}
|T_{j3}^\al|
&\lesssim\mu h(t)^2\int_0^t\|p'_j\wt{z_\al}\|_{L^2}\|p_j\wt{Y_\al}\|_{L^2}ds
\\&+\mu\|\bigl(\f{\p\psi_-}{\p y}\bigr)^{-T}-I\|_{L^\infty_{t,y}}\int_0^t\|p_j\bigl(\langle y\rangle(\log\langle y\rangle)^2\wt{\na\p^\al z}\bigr)\|_{L^2}\|\na(p_j\wt{Y_\al})\|_{L^2}ds .
\end{aligned}\eeq

It remains to bound $T_{j4}^\al=\int_0^t\int_{\R^3}L_\al^4\cdot a_1*a_2*\bigl(p_j\wt{Y_\al}\bigr)dyds$. Since $L^4_\alpha$ can be written as
\beno
L_\al^4=-\mu (p_j\wt{\p^\al z})\Delta[\langle y\rangle(\log\langle y\rangle)^2]\bigr)-\mu (p_j\wt{\p^\al z})\dv\bigl((A_--I)\na[\langle y\rangle(\log\langle y\rangle)^2]\bigr),
\eeno
we can proceed exactly in the same manner as for $T_{j3}^\al$ and we obtain
\beq\label{Estimate for Lal 4}\begin{aligned}
|T_{j4}^\al|&
\lesssim\mu h(t)^2\int_0^t\|p'_j\wt{z_\al}\|_{L^2}\|p_j\wt{Y_\al}\|_{L^2}ds
+\mu\|A_--I\|_{L^\infty_{t,y}} \int_0^t\|p_j\bigl(\langle y\rangle(\log\langle y\rangle)^2\wt{\na\p^\al z}\bigr)\|_{L^2}\|\na(p_j\wt{Y_\al})\|_{L^2}ds.
\end{aligned}\eeq

Finally, in view of \eqref{equivalent 2}, when we sum over $j$,  \eqref{Estimate for Lal 1}-\eqref{Estimate for Lal 4} together yield
\beno
\sum_{j=-1}^\infty|T_j^\al|&\lesssim& \mu  h(t)^2\int_0^t \bigl(\|\wt{z_\al}\|_{L^2}^2+\|\wt{Y_\al}\|_{L^2}^2\bigr)ds
+\mu\big(\|A_--I\|_{L^\infty_{t,y}}+\|\bigl(\f{\p\psi_-}{\p y}\bigr)^{-T}-I\|_{L^\infty_{t,y}}\big)\\&&\times\int_0^t\bigl(\|\langle y\rangle(\log\langle y\rangle)^2\wt{\na\pa^\al z}\|_{L^2}^2+\|\na \wt{Y_\al}\|_{L^2}^2\bigr)ds
\\&\lesssim& \mu t h(t)^2\sup_{t\geq 0}\bigl(\|z_\al\|_{L^2}^2+\|Y_\al\|_{L^2}^2\bigr)
+\mu\big(\|A_--I\|_{L^\infty_{t,y}}+\|\bigl(\f{\p\psi_-}{\p y}\bigr)^{-T}-I\|_{L^\infty_{t,y}}\big)\\&&\times\int_0^t\bigl(\|\wm(\log\wm)^2\na\p^\al z\|_{L^2}^2+\|\na Y_\al\|_{L^2}^2\bigr)ds.
\eeno
By \eqref{estimate for z} and \eqref{Estimate for Y}, we obtain
\beq\label{Estimate for III}
\bigl(\sum_{1\leq|\al|\leq N_*+1}+\mu\sum_{|\al|=N_*+2}\bigr)\sum_{j=-1}^\infty|T_j^\al|
\lesssim\mu th(t)^2\mathcal{E}^\mu(0)+\big(\|A_--I\|_{L^\infty_{t,y}}+\|\bigl(\f{\p\psi_-}{\p y}\bigr)^{-T}-I\|_{L^\infty_{t,y}}\big)\mathcal{E}^\mu(0).
\eeq

As a consequence,  \eqref{Estimate for pjY} yields
\beq\label{Low frequence for Y}\begin{aligned}
&\sum_{1\leq|\al|\leq N_*+1 \atop j\geq -1}\int_{|\xi|\leq h(t)}|\widehat{p_j\wt{Y_\al}}(\xi)|^2d\xi
 +\mu\sum_{|\al|=N_*+2 \atop j\geq -1}\int_{|\xi|\leq h(t)}|\widehat{p_j\wt{Y_\al}}(\xi)|^2d\xi\\
&\lesssim\mu th(t)^2\mathcal{E}^\mu(0)+\big(\|A_--I\|_{L^\infty_{t,y}}+\|\bigl(\f{\p\psi_-}{\p y}\bigr)^{-T}-I\|_{L^\infty_{t,y}}\big)\mathcal{E}^\mu(0)\\
&\quad+\bigl(\sum_{1\leq|\al|\leq N_*+1}+\mu\sum_{|\al|=N_*+2}\bigr)\sum_{j\geq -1}\int_{\R^3}e^{-2\mu t|\xi|^2}|\widehat{p_j\wt{Y_\al}}(0,\xi)|^2\psi(\f{3|\xi|}{4h(t)})d\xi .
 \end{aligned}\eeq

\smallskip

{\it Step 4.2. Estimates on $\int_{|\xi|\leq h(t)}|\widehat{p_j\wt{z_{0k}}}|^2d\xi$ for $k=0,1,2$.}

From \eqref{refined quasilinear system}, we deduce that
\beq\label{Lagarangian formulation for z}
\p_t\wt{z_{0k}}-\mu\Delta\wt{z_{0k}}=\mu\dv\bigl((A_--I)\na\wt{z_{0k}}\bigr)-2\mu\bigl(\f{\p\psi_-}{\p y}\bigr)^{-T}\na\bigl((\log\langle y\rangle)^k\bigr)\cdot\wt{\na z}-\mu\wt z\dv\bigl(A_-\na[(\log\langle y\rangle)^k]\bigr).
\eeq

We can compared with this equation with \eqref{Lagarangian formulation for Y} by replacing $\p^\al z$, $Y_\al$ and $\langle y\rangle(\log\langle y\rangle)^2$ by $z$, $z_{0k}$ and $(\log\langle y\rangle)^k$ respectively. Therefore, following the simialr derivation, we also have similar estimates (compared to \eqref{Low frequence for Y}):
\beq\label{Low frequence for z}\begin{aligned}
\sum_{j\geq -1} \int_{|\xi|\leq h(t)}|\widehat{p_j\wt{z_{0k}}}|^2(t,\xi)d\xi\lesssim&\mu th(t)^2\mathcal{E}^\mu(0)+  \big(\|A_--I\|_{L^\infty_{t,y}}+\|\bigl(\f{\p\psi_-}{\p y}\bigr)^{-T}-I\|_{L^\infty_{t,y}}\big)\mathcal{E}^\mu(0)\\
&+\sum_{j\geq -1}\int_{\R^3}e^{-2\mu t|\xi|^2}|\widehat{p_j\wt{z_{0k}}}(0,\xi)|^2\psi(\f{3|\xi|}{4h(t)})d\xi.
\end{aligned}\eeq

\medskip

{\bf Step 5. The decay of $X(t)$.}

By virtue of \eqref{decay estimate 1}, \eqref{Low frequence for Y}
and \eqref{Low frequence for z}, we conclude that there exists universal constant $c$ and $C$ (independent of $\mu$) such that
\beq\label{decay estimate 2}\begin{aligned}
\f{d}{dt} X(t)+c\mu h(t)^2 X(t)&\leq F(t)
+C\mu h(t)^2I^\mu(t)+C\mu^2 th(t)^4\mathcal{E}^\mu(0)\\
&+C\mu h(t)^2 \big(\|A_--I\|_{L^\infty_{t,y}}+\|\bigl(\f{\p\psi_-}{\p y}\bigr)^{-T}-I\|_{L^\infty_{t,y}}\big)\mathcal{E}^\mu(0),
\end{aligned}\eeq
where the new function $I^\mu(t)$ are determined by the initial data as follows:
\beq\label{functional for initial data}\begin{aligned}
I^\mu(t)= \int_{\R^3} \!\! \sum_{j=-1}^\infty e^{-2\mu t|\xi|^2}\psi(\f{3|\xi|}{4h(t)})\Bigl(\sum_{k=0}^2|\widehat{p_j\wt{z_{0k}}}(0,\xi)|^2 +\!\!\!\!\!\!\!\!\!\sum_{1\leq|\al|\leq N_*+1}\!\!\!\!\!\!\!\!\!|\widehat{p_j\wt{Y_\al}}(0,\xi)|^2+\mu\!\!\!\!\!\!\sum_{|\al|=N_*+2}\!\!\!\!\!\!|\widehat{p_j\wt{ Y_\al}}(0,\xi)|^2\Bigr)d\xi.
\end{aligned}\eeq
We also notice that $|I^\mu(t)|\lesssim\mathcal{E}^\mu(0)$.

Multiplying both sides of \eqref{decay estimate 2} by the factor $e^{\f{c}{2}\mu\int_0^th(\tau)^2d\tau}$, we obtain
\beq\label{decay estimate 3}\begin{aligned}
&\ \ \ \f{d}{dt}\Bigl(e^{\f{c}{2}\mu\int_0^th(\tau)^2d\tau} X(t)\Bigr)+\f{c}{2}\mu h(t)^2 e^{\f{c}{2}\mu\int_0^th(\tau)^2d\tau}X(t)&\leq e^{\f{c}{2}\mu\int_0^th(\tau)^2d\tau}\underbrace{\Big(F(t)+\cdots\Big)}_{\text{Righthand side of} \eqref{decay estimate 2}}.
\end{aligned}\eeq
We first give the bound of $F(t)$. In view of its definition in \eqref{defintion for F(t)}, we obtain that
\beno\begin{aligned}
e^{\f{c}{2}\mu\int_0^th(\tau)^2d\tau}F(t)&=\bigl(\sum_{1\leq|\al|\leq N_*+1}+\mu\sum_{|\al|= N_*+2}\bigr)e^{\f{c}{2}\mu\int_0^th(\tau)^2d\tau}f_\al(t)
-\f{d}{dt}G(t)+\f{c}{2}\mu h(t)^2 G(t),
\end{aligned}\eeno
where
\beno\begin{aligned}
G(t)&=e^{\f{c}{2}\mu\int_0^th(\tau)^2d\tau}\bigl(\sum_{1\leq|\al|\leq N_*+1}+\mu\sum_{|\al|= N_*+2}\bigr)\bigl(2c_\al \int_{\R^3}\wt{Y_\al}\cdot\wt{R_\al} dy+c_\al\|\wt{R_\al}\|_{L^2}^2\bigr) .
\end{aligned}\eeno
We can bound $G(t)$ as
\beno\begin{aligned}
|G(t)|\leq &e^{\f{c}{2}\mu\int_0^th(\tau)^2d\tau}\bigl(\sum_{1\leq|\al|\leq N_*+1}+\mu\sum_{|\al|= N_*+2}\bigr)\bigl(c_\al\|{Y_\al}\|_{L^2}^2+2c_\al\|R_\al\|_{L^2}^2\bigr).
\end{aligned}\eeno
Thanks to \eqref{estimate for nonlinear term 2} and the expression of $X(t)$ in \eqref{energy functional}, we obtain that
\beq\label{error term 1}
e^{\f{c}{2}\mu\int_0^th(\tau)^2d\tau} X(t)+G(t)\gtrsim e^{\f{c}{2}\mu\int_0^th(\tau)^2d\tau} X(t)-Ce^{\f{c}{2}\mu\int_0^th(\tau)^2d\tau}\bigl(\mathcal{E}^\mu(0)\bigr)^{\f32},
\eeq
and
\beno  \f{c}{2}\mu h(t)^2G(t)\le \f{c}{2}\mu h(t)^2  e^{\f{c}{2}\mu\int_0^th(\tau)^2d\tau} X(t)+C \f{c}{2}\mu h(t)^2e^{\f{c}{2}\mu\int_0^th(\tau)^2d\tau}\bigl(\mathcal{E}^\mu(0)\bigr)^{\f32}.\eeno
Thus,  \eqref{decay estimate 3} implies that
\beq\label{decay estimate 4}\begin{aligned}
&\ \ \ \f{d}{dt}\Bigl(e^{\f{c}{2}\mu\int_0^th(\tau)^2d\tau} X(t)+G(t)\Bigr)\\
&\leq\bigl(\sum_{1\leq|\al|\leq N_*+1}+\mu\sum_{|\al|= N_*+2}\bigr)e^{\f{c}{2}\mu\int_0^th(\tau)^2d\tau}f_\al(t)+C\mu h(t)^2e^{\f{c}{2}\mu\int_0^th(\tau)^2d\tau}\bigl(\mathcal{E}^\mu(0)\bigr)^{\f32}\\
&\ \ \ +C\mu h(t)^2e^{\f{c}{2}\mu\int_0^th(\tau)^2d\tau}I^\mu(t)+C\mu^2 th(t)^4e^{\f{c}{2}\mu\int_0^th(\tau)^2d\tau}\mathcal{E}^\mu(0)\\
&\ \ \ +C\mu h(t)^2e^{\f{c}{2}\mu\int_0^th(\tau)^2d\tau} \big(\|A_--I\|_{L^\infty_{t,y}}+\|\bigl(\f{\p\psi_-}{\p y}\bigr)^{-T}-I\|_{L^\infty_{t,y}}\big)\mathcal{E}^\mu(0).
\end{aligned}\eeq

We set from now on that $$\f{c}{2}h(t)^2=(\mu t+e)^{-1}\bigl(\log(\mu t+e)\bigr)^{-1}.$$
Thus, we have $e^{\f{c}{2}\mu\int_0^th(\tau)^2d\tau}=\log(\mu t+e)$. By integrating \eqref{decay estimate 4} on $[0,t]$ and by the estimate \eqref{estimates for nonlinear term 1} on $f_\al$ and the estimate \eqref{error term 1}, we obtain
\beno\begin{aligned}
\log(\mu t+e) X(t)\leq &2X(0)
+C\mu\int_0^t\frac{I^\mu(s)}{\mu s +e}ds+C\log\bigl(\log(\mu t+e)\bigr)\mathcal{E}^\mu(0)\\
&+C\log(\mu t+e)\big[\bigl(\mathcal{E}^\mu(0)\bigr)^{\f32}+\big(\|A_--I\|_{L^\infty_{t,y}}+\|\bigl(\f{\p\psi_-}{\p y}\bigr)^{-T}-I\|_{L^\infty_{t,y}}\big)\mathcal{E}^\mu(0)\big].
\end{aligned}\eeno
Since $X(0)\lesssim\mathcal{E}^\mu(0)$, we get
\beno
X(t)&\lesssim& \f{\log\bigl(\log(\mu t+e)+e\bigr)}{\log(\mu t+e)}\mathcal{E}^\mu(0)+\frac{\mu\int_0^t\frac{I^\mu(s)}{\mu s+e}ds}{\log(\mu t+e)}
\\&&+\bigl(\mathcal{E}^\mu(0)\bigr)^{\f32}+\big(\|A_--I\|_{L^\infty_{t,y}}+\|\bigl(\f{\p\psi_-}{\p y}\bigr)^{-T}-I\|_{L^\infty_{t,y}}\big)\mathcal{E}^\mu(0).
\eeno
We bound the second term on the righthand side as follows:

For $s\le  \f{1}{\mu} \log (\mu t+e)$, we use $I^\mu(t)\lesssim\mathcal{E}^\mu(0)$ and we obtain
\beno \frac{\mu\int_0^{\f{1}{\mu} \log (\mu t+e)}\frac{I^\mu(s)}{ \mu s+e }ds}{\log(\mu t+e)} \lesssim \f{\log\bigl(\log(\mu t+e)+e\bigr)}{\log(\mu t+e)}\mathcal{E}^\mu(0). \eeno

For $s\ge  \f{1}{\mu} \log (\mu t+e)$, we use  $|h(s)|^2\leq\f{2}{c} (\log(\mu t+e)+e)^{-1}\bigl[\log\bigl(\log(\mu t+e)+e\bigr)\bigr]^{-1}$ and the definition of $I^\mu(t)$ to obtain
\beno
&&\frac{\mu\int_{\f{1}{\mu} \log (\mu t+e)}^t\frac{I^\mu(s)}{\mu s+e}ds}{\log(\mu t+e)}\\&&
\lesssim \underbrace{\sum_{j=-1}^\infty \int_{|\xi|^2\le \f{8}{c}(\log (\mu t+e)+e)^{-1}} \Bigl(\sum_{k=0}^2|\widehat{p_j\wt{z_{0k}}}(0,\xi)|^2 +\!\!\!\!\!\!\!\!\sum_{1\leq|\al|\leq N_*+1}\!\!\!\!\!\!\!\!|\widehat{p_j\wt{Y_\al}}(0,\xi)|^2+\mu\!\!\!\!\!\!\!\!\sum_{|\al|=N_*+2}\!\!\!\!\!\!\!\!	 |\widehat{p_j\wt{Y_\al}}(0,\xi)|^2\Bigr)d\xi}_{I^\mu(t;0)}. \eeno
We remark that $I^\mu(t;0)$ is determined by the initial data and $\lim_{t\rightarrow \infty} I^\mu(t;0)=0$.

Finally, we obtain that the decay estimate on $X(t)$:
\beq\label{decay estimate 5}\begin{aligned}
X(t)&\lesssim \f{\log\bigl(\log(\mu t+e)+e\bigr)}{\log(\mu t+e)}\mathcal{E}^\mu(0) + I^\mu(t;0)\\
&+\bigl(\mathcal{E}^\mu(0)\bigr)^{\f32}+\big(\|A_--I\|_{L^\infty_{t,y}}+\|\bigl(\f{\p\psi_-}{\p y}\bigr)^{-T}-I\|_{L^\infty_{t,y}}\big)\mathcal{E}^\mu(0).
\end{aligned}\eeq

\smallskip

{\bf Step 6. Decay mechanism on energy and convergence to the parabolic regime.}

By virtue of \eqref{decompqnl}, \eqref{characterize X} and \eqref{decay estimate 5}, we have the following decay estimates for the total energy:
\begin{equation}\label{decay estimate on total energy}\begin{aligned}
\mathcal{E}^\mu(t)&\lesssim \f{\log\bigl(\log(\mu t+e)+e\bigr)}{\log(\mu t+e)}\mathcal{E}^\mu(0) + I^\mu(t;0)\\&+\underbrace{\bigl(\mathcal{E}^\mu(0)\bigr)^{\f32}}_{H_1}+\underbrace{\sum_{+,-}\big(\|A_\pm-I\|_{L^\infty_{t,y}}+\|\bigl(\f{\p\psi_\pm}{\p y}\bigr)^{-T}-I\|_{L^\infty_{t,y}}\big)\mathcal{E}^\mu(0)}_{H_2}.\end{aligned}
\end{equation}
We remark that the higher order term $H_1$ comes from the non-linear structure of the system  while the  term $H_2$ comes from the change of the coordinates from $(x_1, x_2, x_3)$ to   $(y_1, y_2, y_3)$ (the label of $\Sigma_0$). And we also emphasize that the estimates of $\|A_\pm-I\|_{L^\infty_{t,y}}$ and  $\|\bigl(\f{\p\psi_\pm}{\p y}\bigr)^{-T}-I\|_{L^\infty_{t,y}}$ depend only on the total energy at the time $t=0$ since at the beginning $t=0$, it holds $(x_1^\pm, x_2^\pm, x_3^\pm)|_{t=0}=(x_1, x_2, x_3)$.

For $t\geq s$, we define
\beno
\begin{split}
  \mathcal{I}^\mu(t;s) &=\sum_{+,- \atop j\geq-1}\int_{|\xi|^2
\le\f{8}{c}(\log (\mu (t-s)+e)+e)^{-1}} \Bigl(\sum_{k=0}^2|\widehat{p_j\wt{z_{\pm,0k}}}(s,\xi)|^2 +\!\!\!\!\!\!\!\!\sum_{1\leq|\al|\leq N_*+1}\!\!\!\!\!\!\!\!\!\!|\widehat{p_j\wt{z_{\pm,\al}}}(s,\xi)|^2+\mu\!\!\!\!\!\!\!\!\sum_{|\al|=N_*+2}\!\!\!\!\!\!|\widehat{p_j\wt{z_{\pm,\al}}}(s,\xi)|^2\Bigr)d\xi,
\end{split}
\eeno
where \beno
z_{\pm, 00}:= z_{\pm},\ \ z_{\pm,01}:=\log\wmp z_{\pm},\ \ z_{\pm,02}:=(\log\wmp)^2 z_{\pm},\ \ z_{\al,\pm}:=\wmp(\log\wmp)^2\p^\al z_{\pm}.
\eeno
Then it is easy to generalize \eqref{decay estimate on total energy} to
\begin{equation}\label{general  decay estimate on total energy}\begin{aligned}
\mathcal{E}^\mu(t)&\lesssim \f{\log\bigl(\log(\mu (t-s)+e)+e\bigr)}{\log(\mu (t-s)+e)}\mathcal{E}^\mu(s) + I^\mu(t;s)\\&+\bigl(\mathcal{E}^\mu(s)\bigr)^{\f32}+\sum_{+,-}\big(\|A_\pm-I\|_{L^\infty_{t,y}}+\|\bigl(\f{\p\psi_\pm}{\p y}\bigr)^{-T}-I\|_{L^\infty_{t,y}}\big)\mathcal{E}^\mu(s),\end{aligned}
\end{equation}
where we use the global energy estimate \eqref{global energy estimate for MHDmu}.

An easy but useful observation on $\mathcal{I}^\mu(t;s)$ is that, for a fixed $s$, we have $\mathcal{I}^\mu(t;s)\rightarrow 0$ as $t\rightarrow \infty$. Since $\mathcal{E}^\mu(0) \sim \varepsilon^2$ and $$ \|A_\pm-I\|_{L^\infty_{t,y}}+\|\bigl(\f{\p\psi_\pm}{\p y}\bigr)^{-T}-I\|_{L^\infty_{t,y}}\lesssim \bigl(\mathcal{E}^\mu(0)\bigr)^{\f12},$$there exists $T_1>0$ and a universal constant $C$ such that
\beno
\mathcal{E}^\mu(T_1)\le \bigl(C\mathcal{E}^\mu(0)\bigr)^{\f32}.
\eeno
Therefore, at time $T_1$, the total energy drops for an order of $\varepsilon$. It is obvious that the time $T_1$ depends on the profile of the initial data and there is no uniform (with respect to the energy norms) control on $T_1$.

We can now iterate the above decay process: we treat $T_1$ as an initial time and by   \eqref{general decay estimate on total energy}, for $t\geq T_1$, we obtain that
\begin{equation*}
 \mathcal{E}^\mu(t)\lesssim \mathcal{I}^\mu(t;T_1)+\f{\log\bigl(\log(\mu(t-T_1)+e)+e\bigr)}{\log(\mu(t-T_1)+e)}\mathcal{E}^\mu(T_1)
+ \bigl(\mathcal{E}^\mu(0)\bigr)^{\f12}\mathcal{E}^\mu(T_1).
\end{equation*}
Since $\lim_{t\rightarrow \infty}\mathcal{I}^\mu(t;T_1)= 0$, there also exists $T_2>T_1$ in such a way that
\beno
\mathcal{E}^\mu(T_2)\le \bigl(C\mathcal{E}^\mu(0)\bigr)^{\f12}\mathcal{E}^\mu(T_1)\le \bigl(C\mathcal{E}^\mu(0)\bigr)^{2}.
\eeno
By repeating the process, we can find time $T_1$, $T_2$, $\cdots$, $T_{n_0}$ such that
\beno
\mathcal{E}^\mu(T_{n_0})\le  \bigl(C\mathcal{E}^\mu(0)\bigr)^{\f{n_0}{2}+1}.
\eeno
We take $n_0 = 2\lfloor\frac{\log\varepsilon_{\mu}}{\log(\sqrt{C}\varepsilon)}\rfloor+1$ where $\lfloor m\rfloor$ denotes the maximum integer which does not exceed $m$. Therefore it holds
\beno
\mathcal{E}^\mu(T_{n_0})\leq \varepsilon_\mu^2.
\eeno
In particular, the $H^2$-norm of the system at $T_{n_0}$ are bounded above by $\varepsilon_\mu$. Therefore, the solutions are in the classical small-data parabolic regime. This completes the proof of the theorem.

\end{document}